\documentclass[12pt, reqno]{amsart}
\usepackage[english]{babel}
\usepackage[utf8]{inputenc}
\usepackage{lipsum}
\usepackage{dsfont}
\usepackage{multirow}
\usepackage{mathrsfs}
\usepackage{stix}
\usepackage{fullpage}
\usepackage{mathtools}
\usepackage{amsmath}
\usepackage{bbm}
\usepackage{amsfonts}
\usepackage{tikz}
\usepackage{tikz-cd}
\usetikzlibrary{decorations.pathreplacing,angles,quotes}
\usetikzlibrary{arrows,chains,positioning,scopes,quotes}
\newtheorem{defi}{Definition}[section]

\newtheorem*{conjecture}{Conjecture}
\newtheorem{lema}[defi]{Lemma}
\newtheorem{teo}[defi]{Theorem}
\newtheorem{rem}[defi]{Remark}

\newtheorem*{rem*}{Remark}

\newcommand{\C}{\mathbb{C}}
\newcommand{\Q}{\mathbb{Q}}
\newcommand{\K}{\mathbb{K}}
\newcommand{\R}{\mathbb{R}}
\newcommand{\N}{\mathbb{N}}
\newcommand{\Z}{\mathbb{Z}}

\newcommand{\1}{\mathds{1}}

\DeclareMathOperator{\spn}{\mathrm{span}}

\newcommand{\interior}[1]{%
  {\kern0pt#1}^{\mathrm{o}}%
}

\newcommand{\esp}{\text{  }}

\usepackage[colorlinks=false]{hyperref}
\renewcommand\eqref[1]{(\ref{#1})} 
\begin{document}

\title[Unitary dual and matrix coefficients of compact nilpotent $p$-adic Lie groups with dimension $d\leq5$]
 {Unitary dual and matrix coefficients of compact nilpotent $p$-adic Lie groups with dimension $d\leq5$ }

\author{
  J.P. Velasquez-Rodriguez
}

\newcommand{\Addresses}{{
  \bigskip
  \footnotesize

 J.P. Velasquez-Rodriguez, \textsc{Departamento de Matematicas, Universidad del Valle, Cali-Colombia}\par\nopagebreak
  \textit{E-mail address:} \texttt{velasquez.juan@correounivalle.edu.co}}

}


\thanks{The author is supported by the grace of God, a.k.a. The Truth, and the magnanimity of his majesty Lord Bruce the First, ruler of The Lands Between }



\date{\today}
\begin{abstract}
Let $p>2$ be a prime number, and let $\mathbb{G}$ be a compact nilpotent $p$-adic Lie group with nilpotency class $\mathscr{N}<p$. In this note we calculate explicitly the unitary dual and the matrix coefficients of every compact nilpotent $p$-adic Lie group with dimension less or equal than $5$. As an application, we provide the corresponding spectral theorem for the Vladimirov sub-Laplacian, and show how this operator provides a non-trivial example of a globally hypoelliptic operator on compact nilpotent $p$-adic Lie groups.
\end{abstract}
\maketitle
\tableofcontents
\section{Introduction}
\subsection{Motivation} Noncommutative groups have been very important mathematical objects since their introduction, specially after physicist started viewing them as the appropriate language to talk about \emph{transformations and symmetries of physical systems}. For instance, Marius Sophus Lie initially introduced what we call now \emph{Lie groups} to study continuous symmetries related to differential equations. He aimed to understand how these symmetries could be represented, which led to significant advancements in both mathematics and physics. Lie groups allow for the analysis of transformations that preserve the structure of differential equations, enabling the classification and solution of these equations through the concept of symmetry. By applying the tools of differential calculus to groups, Lie could explore how these symmetries interact and how they could simplify complex problems. 
z
The archetypal example of the principle Lie discovered is the \emph{heat equation} on the interval $[0,1]$, which can be thought as a bar of a certain material with length one, and with periodic boundary conditions: $f(0) = f(1)$. This equation was fully solved by Joseph Fourier in his seminal work \textit{Théorie Analytique de la Chaleur} \cite{Fourier2009}, and it marks a pivotal moment in the history of both mathematics and physics, as this groundbreaking book presents the first comprehensive mathematical treatment of solutions to the heat diffusion equation
\[
\frac{\partial f}{\partial t} = \mathscr{L} f, \quad \mathscr{L} := \frac{d^2}{dx^2},
\]where the operator $\mathscr{L}$ is the one-dimensional Laplacian. This equation serves as a fundamental model for understanding how heat disperses through a given medium, and Fourier's analysis introduces a variety of special series and integrals to articulate its solutions. The implications of Fourier's work were profound; he established that real-valued functions, which can be interpreted as representing heat distributions across a physical object, could be decomposed into sums of simpler, more manageable functions: specifically, sines and cosines. These functions turn out to be a complete system of eigenfunctions of the operator $\mathscr{L}$, which means they can serve as the fundamental building blocks for constructing more complicated functions. Fourier's concept extends naturally into the realm of complex-valued functions, which can similarly be written as sums of complex exponentials $\mathscr{e}_k(\theta) = e^{2 \pi i k \theta}$, where $k$ is an integer number. This way of expressing functions is not merely an isolated case; it exemplifies a broader mathematical phenomenon that would not be fully understood until centuries later, particularly due to the advancements made in group theory and the study of topological vector spaces. 

From a more elaborated point of view, functions defined on $[0,1]$ with periodic boundary conditions can be identified with functions on the unit circle
\[
\mathbb{T} := \{ z \in \mathbb{C} : z z^* = 1 \},
\]where $z^*$ indicates the complex conjugate of $z$. This set is actually a topological group with a differential structure, i.e., a Lie group, and its algebraic structure is important because, by the Peter-Weyl theorem, the set 

\[
\widehat{\mathbb{T}} := \{\mathscr{e}_k : k \in \mathbb{Z}\},
\]
which is a complete collection of representatives of equivalence classes of unitary irreducible representations of $\mathbb{T}$, serves as an orthonormal basis for the space \(L^2(\mathbb{T}) \cong L^2 [0,1]\). In particular, if we only consider real-valued functions, we recover Fourier's original ideas. But if instead we choose to consider more general noncommutative spaces, the extensions and generalizations of Fourier's techniques are collectively referred to as \emph{Fourier analysis}. This field encompasses a great deal of examples where functions defined on certain topological spaces can be decomposed into sums of "simpler" functions. When one focuses on spaces endowed with additional algebraic structure, which are  spaces rich in symmetries, there are two distinct pathways for exploration:

\begin{enumerate}
    \item \textbf{Studying locally connected spaces:} A natural progression from the theory on the torus $\mathbb{T}$ involves investigating compact Lie groups. This topic is meticulously addressed in the work of \cite{Ruzhansky2010}, which lays out the foundational elements of a comprehensive theory of pseudo-differential operators applicable to compact Lie groups. Understanding the representation theory of these groups is a crucial aspect, though it often presents significant challenges. The complexity of the representation theory can vary greatly depending on the specific group in question. For example, the representation theory for $\mathrm{SU}(2)$ is explored in \cite{Ruzhansky2010}, while spin groups are discussed in \cite{Cerejeiras2023}.
    
    \item \textbf{Studying totally disconnected spaces:} In the domain of totally disconnected spaces, a rich and intricate theory has developed concerning Fourier analysis on non-Archimedean local fields and their corresponding rings of integers. The first major generalization in this context involves examining what are known as \emph{locally compact abelian Vilenkin groups}. These groups are characterized by specific sequences of compact open subgroups, and further details can be found in references such as \cite{pseudosvinlekinsaloff}. Much like the locally connected case, when transitioning to non-commutative groups, we must consider compact $p$-adic Lie groups. These groups have predominantly been analyzed from an algebraic perspective within the literature \cite{Boyarchenko2008, Howe1977KirillovTF}, as well as within the broader category of \emph{compact Vilenkin groups}.
\end{enumerate}

In the present paper, we choose to pursue the second pathway, initiating our exploration with compact nilpotent $p$-adic Lie groups. Over the past fifteen years, a comprehensive symbol calculus for compact (noncommutative) groups has been developed by researchers such as Ruzhansky, Turunen, and Wirth. Although their primary focus has been on compact Lie groups over the reals \cite{Ruzhansky2010}, this new framework serves as a noncommutative extension of the classical Kohn–Nirenberg quantization. This quantization is an essential tool for analyzing differential operators on groups and presents several advantages compared to Hörmander's principal calculus or other local methods. For a given Lie group \( G \), this innovative approach leverages the representation theory and the corresponding harmonic analysis to establish a global Fourier transform, which effectively diagonalizes differential operators, including the Laplace-Beltrami operator. This is, in essence, a far-reaching generalization of the original ideas proposed by Fourier.

Our motivation to extend the ideas known on real Lie groups to the t.d. case is deeply rooted in a principle attributed to Harish-Chandra, known as the "Lefschetz principle." This principle asserts that \emph{real groups, $p$-adic groups, and automorphic forms—corresponding to both archimedean and non-archimedean local fields, as well as number fields—should be regarded on equal footing, and that ideas and results emerging from one of these categories should be transferable to the other two} \cite{HarisLanglands}. Within this framework, we aim to extend certain techniques from harmonic analysis on real groups to the realm of compact $p$-adic Lie groups. These topological groups have primarily been examined in the literature through an algebraic lens. Our main focus will be on Fourier analysis, particularly investigating the Vladimirov-Taibleson operator, which has a significant parallel with the fractional Laplacian in real analysis, within the context of non-commutative groups defined over the $p$-adic integers.

\newpage

To further elucidate what we mean by the phrase \emph{the extension of some techniques from harmonic analysis on Lie groups}, it is crucial to understand that analysis on Lie groups fundamentally relies on the differential structure of the group as a smooth manifold. To date, to the best of the author’s knowledge, there exists no comprehensive theory of derivative, differential, or pseudo-differential operators specifically on $p$-adic manifolds, nor on the broader class of non-commutative locally profinite groups. Despite the rich theoretical landscape surrounding $p$-adic Lie groups, algebraic groups, and $\K$-analytic manifolds, significant gaps remain in our understanding. In recent years, however, researchers like P.E. Bradley and collaborators have begun the study of integral operators operators (which can be regarded as pseudo-differential) that remain invariant under finite group actions, heat equations on Mumford curves, and $p$-adic Laplacians \cite{Bradley3, Bradley2, bradley1}. Beyond this, it is worth noting that only Kochubei has examined equations defined on a $p$-adic ball \cite{Kochubei2018}, and a limited number of works have addressed general ultrametric spaces \cite{Bendikov2014}. A common thread that emerges within the $p$-adic side of the theory is the Vladimirov–Taibleson operator, which provides a definition of ``differentiability'' for functions defined on ultrametric spaces. Most pseudo-differential equations in the literature, especially those aimed at real-world applications, are expressed in terms of this operator or a similar one \cite{Khrennikov2018}. This operator is commonly regarded as a kind of fractional Laplacian for complex-valued functions on totally disconnected spaces \cite{Kochubei2023}, particularly because, in some cases, it actually coincides with the isotropic Laplacian associated with a certain ultrametric in general ultrametric spaces, as demonstrated by the works of A. Bendikov, A. Grigoryan, C. Pittet, and W. Woess.

\begin{defi}\normalfont\label{defidualcompact}
\,
\begin{itemize}
    \item We will use the symbols $a \lesssim b$ and $a \gtrsim b$ to indicate that the quantity $a$ is, respectively, less or equal, or greater or equal than a constant times the quantity $b$. We use the notation $a \asymp b$ to indicate that $a \lesssim b$ and $a \gtrsim b$ hold simultaneously.
    \item Let $(G, \mu)$ be a measure space. For a measurable set $A \subset G$, we will denote by $\mu(A)$ the measure of the set $A$, and the symbol $\1_A$ will be used to denote the indicator function on $A$. 
    \item Let $G$ be a compact topological group. We will denote by $\mathrm{Rep}(G)$ the collection of all (equivalent classes of) continuous, unitary, finite-dimensional representations of $G$. The symbol $\widehat{G}$ will be used to denote the \emph{unitary dual} of $G$, that is, the collection of all (equivalence classes of) continuous unitary irreducible representations of $G$. Remember how any element $[\pi]$ of $\mathrm{Rep}(G)$ can be written as a direct sum of finitely many elements of $\widehat{G}$.
\end{itemize}      
\end{defi}

\subsection{Our setting}
This paper focuses on the representation theory of compact nilpotent $p$-adic Lie groups with dimensions up to 5. In his seminal work \cite{Dixmier1958}, Dixmier established that any nilpotent Lie algebra with dimension $d\leq 5$ is isomorphic to one of the following:
\begin{itemize}
    \item \textbf{Dimension 1:} $\mathfrak{g}_1$.
    \item \textbf{Dimension 2:} $(\mathfrak{g}_1)^2$.
    \item \textbf{Dimension 3:} $(\mathfrak{g}_1)^3$, $\mathfrak{h}_1 $.
    \item \textbf{Dimension 4:} $(\mathfrak{g}_1)^4$, $\mathfrak{h}_{1} \times \mathfrak{g}_1$, $\mathfrak{b}_4$.
    \item \textbf{Dimension 5:} $(\mathfrak{g}_1)^5$, $\mathfrak{h}_{1} \times (\mathfrak{g}_1)^2$, $\mathfrak{b}_4 \times \mathfrak{g}_1$, $ \mathfrak{g}^{5,1}= \mathfrak{h}_{2 } $, $\mathfrak{g}^{5,2}$, $\mathfrak{g}^{5,3}$, $\mathfrak{g}^{5,4}$, $\mathfrak{g}^{5,5}$ ,$\mathfrak{g}^{5,6}$.
\end{itemize}
The above listed Lie algebras are defined by the following commutation relations: 
\begin{itemize}
    \item \textbf{Dimension 3:} $$\mathfrak{h}_1 \, : \quad [X_1 , X_2]=[X_3].$$
    \item \textbf{Dimension 4:} $$\mathfrak{b}_4 \, : \quad [X_1 , X_2]=X_3, \quad [X_1, X_3]=X_4.$$
    \item \textbf{Dimension 5:} $$\mathfrak{g}^{5,1}=\mathfrak{h}_2 \, : \quad [X_1 , X_3]=X_5, \quad [X_2, X_4]=X_5.$$
    $$\mathfrak{g}^{5,2} \, : [X_1,X_2] = X_4 , \quad [X_1 , X_3]= X_5.$$
    $$\mathfrak{g}^{5,3} \, : [X_1,X_2] = X_4 , \quad [X_1 , X_4]= X_5, \quad [X_2 , X_3]=X_5.$$
    $$\mathfrak{g}^{5,4} \, : [X_1,X_2] = X_3 , \quad [X_1 , X_3]= X_4, \quad [X_2 , X_3]=X_5.$$
    $$\mathfrak{g}^{5,5} \, : [X_1,X_2] = X_3 , \quad [X_1 , X_3]= X_4, \quad [X_1 , X_4]=X_5.$$
    $$\mathfrak{g}^{5,6} \, : [X_1,X_2] = X_3 , \quad [X_1 , X_3]= X_4, \quad [X_1 , X_4]=X_5, \quad [X_2, X_3] = X_5.$$
\end{itemize}
In his paper, Dixmier characterized the unitary dual of all Lie groups over $\R$ that can be represented as the exponential image of the aforementioned Lie algebras, see \cite{Dixmier1958}. Here in this paper, we aim to achieve a similar analysis for the compact $\Z_p$-Lie modules derived from any of the specified commutation relations, including the case of $\mathfrak{g}^{5,4}$, which is referred to as the \emph{Cartan algebra} by some authors. We have already begun this endeavor by explicitly calculating the unitary dual of the Heisenberg groups $\mathbb{H}_d(\Z_p)$, $d \in  \N$, and the Engel group $\mathcal{B}_4(\Z_p)$, which are the exponential images of the $\Z_p$-Lie modules $\mathfrak{h}_{d}$ and $\mathfrak{b}_4$, respectively. See \cite{velasquezrodriguez2024Engelspectrumvladimirovsublaplaciancompact, velasquezrodriguez2024spectrumvladimirovsublaplaciancompact} for the full calculations. In this paper, we focus exclusively on the $5$-dimensional algebras $\mathfrak{g}^{5,i}$, where $2 \leq i \leq 6$, and the unitary irreducible representations of the compact groups obtained as their exponential images. Our main objective is to derive explicit expressions for the unitary irreducible representations and the matrix coefficients of $5$-dimensional compact nilpotent $p$-adic Lie groups, which will later aid in the study of pseudo-differential operators. In particular, we will focus on a specific operator known as \emph{the Vladimirov sub-Laplacian}, which we introduced in \cite{velasquezrodriguez2024spectrumvladimirovsublaplaciancompact, velasquezrodriguez2024Engelspectrumvladimirovsublaplaciancompact}. This operator shares some similarities with the sub-Laplacian on graded Lie groups over $\R$, or it can also be considered close to a H{\"o}rmander's operator. 

Compact nilpotent $p$-adic Lie groups have been extensively studied from an algebraic standpoint, particularly within the framework of pro-$p$ groups or general profinite groups. In this work, we treat them as the unit ball $\Z_p^d$ of $\Q_p^d$, endowed with a non-commutative operation denoted by "$\star$", which gives $\Z_p^d$ the structure of a non-commutative (compact) topological group. While we explore their algebraic properties to some extent, our primary focus here is the analysis on these groups, specially the use of representation theory to study harmonic analysis and pseudo-differential operators. For instance, to illustrate our approach, consider the $(2d+1)$-dimensional Heisenberg group over $\mathbb{Z}_p$, denoted as $\mathbb{H}_d(\mathbb{Z}_p)$, which can be defined as $\Z_p^{2d+1}$ equipped with the non-commutative operation  $$(\mathbf{x}, \mathbf{y},z) \star (\mathbf{x}',\mathbf{y}', z') := (\textbf{x} + \textbf{x}' , \textbf{y} + \textbf{y}' , z + z' + \mathbf{x} \cdot \mathbf{y}').$$
For this group, Theorem \ref{TeoRepresentationsHd} was established in \cite{velasquezrodriguez2024spectrumvladimirovsublaplaciancompact}, which is the following characterization of the \emph{unitary dual}, that is, the collection of all (equivalence classes of) unitary irreducible representations of $\mathbb{H}_d$.  
 \begin{rem}
    In this paper, we will identify each equivalence class $\lambda$ in $\widehat{\Z}_p \cong \Q_p / \Z_p$, with its associated representative in the complete system of representatives $$\{1\} \cup \big\{ \sum_{k =1}^\infty \lambda_k p^{-k} \, : \, \, \text{only finitely many $\lambda_k$ are non-zero.} \big\}.$$Similarly, every time we consider an element of the quotients $\Q_p / p^{-n} \Z_p$ it will be chosen from the complete system of representatives   $$\{1\} \cup \big\{ \sum_{k =n+1}^\infty \lambda_k p^{-k} \, : \, \text{only finitely many $\lambda_k$ are non-zero} \big\}.$$Also, given any $p$-adic number $\lambda$, we will denote by $\vartheta(\lambda)$ the $p$-adic valuation of $\lambda \in \Q_p$.
\end{rem}

\begin{teo}\label{TeoRepresentationsHd}
Let $p>2$. Let $\mathbb{H}_d (\Z_p)$, or simply $\mathbb{H}_d$ for short, be the $(2d+1)$-dimensional Heisenberg group over the $p$-adic integers, and let us denote by $\vartheta:\Q_p \to \Z$ the $p$-adic valuation over $\Q_p$. Let us denote by $\widehat{\mathbb{H}}_d$ the unitary dual of $\mathbb{H}_d$, i.e., the collection of all equivalence classes of unitary irreducible representations of $\mathbb{H}_d$. Then we can identify $\widehat{\mathbb{H}}_d$ with the following subset of $\widehat{\Z}_p^{2d+1} \cong \Q_p^{2d+1}/\Z_p^{2d+1}$: $$\widehat{\mathbb{H}}_d = \{(\xi , \eta, \lambda ) \in  \widehat{\Z}_p^{2d+1} \, : (\xi , \eta) \in \Q_p^{2d} / p^{\vartheta(\lambda)} \Z_p^{2d} \}.$$Moreover, each non-trivial representation $[\pi_{(\xi, \eta, \lambda)}] \in \widehat{\mathbb{H}}_d $ can be realized in the finite dimensional sub-space $\mathcal{H}_\lambda$ of $L^2(\Z_p^d)$ defined as
$$\mathcal{H}_\lambda := \spn_\C \{ \varphi_h \, : \, h \in \Z_p^d / p^{-\vartheta(\lambda)}\Z_p^d  \}, \, \, \, \varphi_h (u) := |\lambda|_p^{d/2} \1_{h + {p^{-\vartheta(\lambda)}} \Z_p^d} (u), \, \, \dim_{\C}(\mathcal{H}_\lambda) =|\lambda|_p^d,$$
where the representation acts on functions $\varphi \in \mathcal{H}_\lambda$ according to the formula 
$$\pi_{(\xi , \eta , \lambda)}(\mathbf{x}, \mathbf{y} ,z) \varphi (u) := e^{2 \pi i \{\xi \cdot \mathbf{x} + \eta \cdot \mathbf{y} + \lambda (z +  u \cdot \mathbf{y}) \}_p} \varphi (u + \mathbf{x}), \, \, \, \, \varphi \in \mathcal{H}_\lambda.$$With this explicit realization, and by choosing the basis $\{ \varphi_h\,: \, h \in \Z_p^d / {p^{-\vartheta(\lambda)}}\Z_p^d\}$ for each representation space, the associated matrix coefficients are given by $$(\pi_{(\xi , \eta , \lambda)})_{h h'}(\mathbf{x},\mathbf{y}, z)=e^{ 2 \pi i \{ \xi \cdot \mathbf{x} +  \eta \cdot \mathbf{y} + \lambda(z + h' \cdot \mathbf{y} )  \}_p} \1_{h - h' + {p^{-\vartheta(\lambda)}}  \Z_p^d}(\mathbf{x}) .$$Sometimes we will use the notation $$\mathcal{V}_{(\xi , \eta , \lambda)}:= \mathrm{span}_\C \{ (\pi_{(\xi , \eta , \lambda)})_{h h'} \, : \, h,h' \in \Z_p^d / {p^{-\vartheta(\lambda)}} \Z_p^d \}.$$
\end{teo}
The above result, whose statement is included here for a matter of completeness, is the model for the theorems we will be proving on $\mathbb{G}^{5,i}$, $2 \leq i \leq 6$. Concretely, we are interested in showing how the unitary dual of any nilpotent $\Z_p$-Lie group, with $d=5$, can be identified with some subset of $\widehat{\Z}_p^5$, kind of like a pruned version of the three $\widehat{\Z}_p^5$ particular to each group. The guess is: the same holds true for higher dimensional groups, and one should be able to obtain the matrix coefficients with a similar procedure, but proving it is beyond the scope of this work. 

\subsection{Main goals}
In 1967, H{\"o}rmander proved that, under the assumption that the system of smooth vector fields $X_1, \ldots, X_\kappa$ generates the entire tangent space at any point, a condition nowadays called \emph{H{\"o}rmander's condition}, the operator
\begin{equation}
    L := \sum_{i=1}^\kappa X_i^2 + X_0 + c,
\end{equation}
which is called a \emph{H{\"o}rmander's operator}, has the following property: if $Lu = f$ and $f$ is smooth, we can conclude for sure that $u$ is also smooth. Operators satisfying this condition are called \emph{hypoelliptic operators}, and they play a very important role in the theory of partial differential operators. Important examples of H{\"o}rmander's operators are the sub-Laplacians on Lie groups, see \cite{Bramanti2014} and the references therein.

In general, proving that a given partial differential operator is hypoelliptic is a non-trivial problem, and there is a substantial body of literature dedicated to addressing this question. In this paper we consider the problem of the global hypoellipticity on compact nilpotent groups over $\Z_p$. Specifically, we will study the \emph{directional Vladimirov–Taibleson operators} introduced in \cite{velasquezrodriguez2024spectrumvladimirovsublaplaciancompact, velasquezrodriguez2024Engelspectrumvladimirovsublaplaciancompact}, which somehow reassemble directional derivatives. 

\begin{defi}\label{defiDirectionalK}\normalfont
Let $\K$ be a non-archimedean local field with ring of integers $\mathscr{O}_\K$, prime ideal $\mathfrak{p}= \textbf{p} \mathscr{O}_\K$, and residue field $\mathbb{F}_q = \mathscr{O}_\K/\textbf{p} \mathscr{O}_\K.$ Let $\mathfrak{g} = \spn_{\mathscr{O}_\K} \{ X_1,..,X_d \}$ be a nilpotent $\mathscr{O}_\K$-Lie algebra, and let $\mathbb{G}$ be the exponential image of $\mathfrak{g}$, so that $\mathbb{G}$ is a compact nilpotent $\K$-Lie group. We will use the symbol symbol $\partial_{X}^\alpha$ to denote the \emph{directional Vladimirov–Taibleson operator in the direction of $X \in \mathfrak{g}$}, or directional VT operator for short, which we define as $$\partial_{X}^\alpha f(\mathbf{x}) := \frac{1 - q^\alpha}{1-q^{-(\alpha +1)}} \int_{\mathscr{O}_\K} \frac{f(\mathbf{x} \star \textbf{exp}(tX)^{-1}) - f(\mathbf{x}) dt}{|t|_\K^{\alpha +1}}, \, \, \, \, f \in \mathcal{D}(\mathbb{G}).$$Here $\mathcal{D}(\mathbb{G})$ denotes the space of smooth functions on $\mathbb{G}$, i.e., the collection of locally constant functions with a fixed index of local constancy.   
\end{defi}

\begin{rem}
Notice how the above definition is independent of any coordinate system we choose for the group as a manifold. Even though we are using the exponential map in the definition, this does not mean we are using the exponential system of coordinates. First, $x \in \mathbb{G}$ is not written in coordinates, and second, we know there is a one to one correspondence between elements of the Lie algebra $\mathfrak{g}$ and one-parameter subgroups of $\mathbb{G}$. Here a one-parameter subgroup of $\mathbb{G}$ is an analytic homomorphism $\gamma_X:\mathscr{O}_\K \to \mathbb{G}$.  Using this fact, given any $X \in \mathfrak{g}$, we can define the directional VT operator in the direction of $X \in \mathfrak{g}$ via the formula  $$\partial_{X}^\alpha f(\mathbf{x}) := \frac{1 - q^\alpha}{1-q^{-(\alpha +1)}} \int_{\mathscr{O}_\K} \frac{f(\mathbf{x} \star \gamma_X(t)^{-1}) - f(\mathbf{x}) dt}{|t|_\K^{\alpha +1}}.$$
Since $\mathbb{G}$ is compact, we can identify it with a matrix group, where all one parameter subgroups have the form $\gamma_X (t) = e^{tX}$, which justifies our initial definition. More generally, we can write any analytic vector field on $\mathbb{G}$ as $$X(\mathbf{x}):= 
\sum_{j=1}^d c_j (\mathbf{x}) X_j,$$where the coefficient functions $c_j$ are analytic. For this vector field, the directional VT operator is defined in a similar way as $$\partial_{X}^\alpha f(\mathbf{x}) := \frac{1 - q^\alpha}{1-q^{-(\alpha +1)}} \int_{\mathscr{O}_\K} \frac{f(\mathbf{x} \star \mathbf{exp}(tX (\mathbf{x}))^{-1}) - f(\mathbf{x}) dt}{|t|_\K^{\alpha +1}}.$$However, in this work we wont consider the case where the coefficients are not constant, because such operators are not necessarily invariant.         
\end{rem}

By using Definition \ref{defiDirectionalK}, we can link an specific Vladimirov-type operator to each direction $X \in \mathfrak{g}$ and subsequently study the resulting operators. While this association does not preserve the Lie algebra structure, as seen in Lie groups over the real numbers, the resulting operators are nonetheless interesting and share similarities with partial differential operators on Lie groups. We want to study polynomials in these operators and, specifically, a distinguished one called here the \emph{Vladimirov sub-Laplacian}, which we can understand very well in the nilpotent case, $d:= dim_{\mathbb{Z}_p} (\mathbb{G}) \leq 5$,  due to our explicit knowledge of the representation theory. For this operator, which we define in principle for compact nilpotent groups, we would like to pose the following conjecture:

\begin{conjecture}\label{conjecture}\normalfont   
 Let $\K$ be a non-archimedean local field with ring of integers $\mathscr{O}_\K$, prime ideal $\mathfrak{p}= \textbf{p} \mathscr{O}_\K$, and residue field $\mathbb{F}_q = \mathscr{O}_\K/\textbf{p} \mathscr{O}_\K.$ Let $\mathfrak{g} = \spn_{\mathscr{O}_\K} \{ X_1,..,X_d \}$ be a nilpotent $\mathscr{O}_\K$-Lie algebra, and let $\mathbb{G}$ be the exponential image of $\mathfrak{g}$, so that $G$ is a compact nilpotent $\K$-Lie group. Let $X_1,...,X_{\kappa}$, $1 \leq \kappa \leq d$, be a basis for $\mathfrak{g}/ [\mathfrak{g} , \mathfrak{g}]$, so that when $\mathfrak{g}$ is graded or stratified $X_1,...,X_{\kappa}$ generates $\mathfrak{g}$, and let $\mathbb{G}$ be the exponential image of $\mathfrak{g}$. Then the \emph{Vladimirov sub-Laplacian of order $\alpha>0$} is a hypoelliptic operator on $G$, and it is invertible in the space of mean-zero functions. Here the Vladimirov sub-Laplacian is the pseudo-differential operator $\mathscr{L}^\alpha_{sub}$, defined on the space of smooth functions $\mathcal{D}(\mathbb{G})$ via the formula $$\mathscr{L}^\alpha_{sub} f(\mathbf{x}) := \sum_{k= 1}^\kappa 
\partial_{X_k}^\alpha f(\mathbf{x}),$$where $X_1 , ...,X_\kappa $ spans $\mathfrak{g}/[\mathfrak{g},\mathfrak{g}]$. 
\end{conjecture}

\begin{rem}\normalfont
    To the knowledge of the author, an operator like the Vladimirov sub-Laplacian has not appeared before in the mathematical literature. The closest thing is probably the study of the \emph{Vladimirov-Laplacian}, which the reader can find in the work of Rajkumar and Weisbart \cite{Rajkumar2023}, and the paper by Bendikov,  Grigoryan,  Pittet,  and Woess \cite{Bendikov2014} where the authors studied the operators $$\mathscr{L}^\alpha f(\mathbf{x}) := \sum_{k= 1}^d 
\partial_{X_k}^\alpha f(\mathbf{x}),$$as the generators of a Dirichlet form associated with a certain jump kernel. See \cite[Section 5]{Bendikov2014} for more details. From the perspective of \cite{Rajkumar2023}, when $X_k = \mathbf{e}_k$ are the canonical directions, the Vladimirov Laplacian is the infinitesimal generator of a certain Markov process called there \emph{the product process}. Here we want to to emphasize how it is not necessary to restrict ourselves to the canonical directions, and furthermore we can extend the same analysis to noncommutative groups.
\end{rem}

So, drawing an analogy between the real case and the non-archimedean case, we can think of the Vladimirov sub-Laplacian as an analog of the sub-Laplacian or H{\"o}rmander's sum of squares for the non-archimedean case. Consequently, it is reasonable to expect the global hypoellipticity for these operators, and the second main goal of this paper is to demonstrate that this is indeed the case for all the compact nilpotent $p$-adic Lie groups, assuming the dimension of the group is less or equal than $5$, and the nilpotency index of the group is strictly less than $p$. As the main results of this paper, we will provide a complete description of the unitary dual and the matrix coefficients of the representations of these groups, and as an application we will obtain the spectal theorem for the Vladimirov sub-Laplacian on any of these groups. To be more precise, our goal in this work is to establish the following result, which justify our conjecture about sub-Laplacians on more general stratified, or even just  nilpotent groups.

\begin{teo}\label{teogeneralresult}
Let $\mathbb{G}$ be a compact nilpotent $p$-adic Lie group with nilpotency class $\mathscr{N}$, and assume that $p>\mathscr{N}$, $d:=\mathrm{dim}_{\Z_p} (\mathbb{G}) \leq 5$. Let $\mathfrak{g}$ be its associated $\Z_p$-Lie module, and let let $\{ X_1 , ...,X_\kappa \}$ be a $\Z_p$-basis for $\mathfrak{g}/[\mathfrak{g},\mathfrak{g}]$. Then: 
\begin{enumerate}
    \item The unitary dual of $\mathbb{G}$ can be indexed by a certain $\Gamma  \subset \widehat{\Z}_p^d$, i.e. any $[\pi] \in \widehat{\mathbb{G}}$ is equivalent to some $\pi_\xi$, with $\xi \in \Gamma$, and the annihilators $$ \widehat{\mathbb{G}} \cap \mathbb{G} (p^n \Z_p)^\bot:=\{ [\pi] \in \widehat{\mathbb{G}} \, : \, \pi|_{\mathbb{G}(p^n \Z_p)} = I_{d_\pi} \},$$can be identified with the balls $$B_{\widehat{\mathbb{G}}} (n) := \{ \xi \in \Gamma \, : \, \| \xi \|_p \leq p^n \}.$$
    \item The Vladimirov sub-Laplacian, here defined as $$\mathcal{L}_{sub}^\alpha f (\mathbf{x}) := \sum_{k=1}^\kappa \partial_{X_k}^\alpha f (\mathbf{x}), \quad f \in \mathcal{D}(\mathbb{G}),$$is a globally hypoelliptic, self-adjoint operator on $L^2 (\mathbb{G})$, with an associated complete system of smooth generalized eigenfunctions. Moreover, it holds: $$ \|(\xi_1 , ..., \xi_\kappa) \|_p^\alpha \lesssim \| \sigma_{\mathcal{L}_{sub}^\alpha} (\xi) \|_{inf} , \quad \xi \in \Gamma,$$where $\sigma_{\mathcal{L}_{sub}^\alpha}$ denotes the symbol of the invariant operator $\mathcal{L}_{sub}^\alpha$.
\end{enumerate}
   
\end{teo}

\subsection{The concept of global hypoellipticity}
When talking about the real case, global hypoellipticty comes in two versions, one local and one global, because we can talk about local and global differentiability. If $M$ is a smooth manifold (over $\R$) and $L$ is some differential operator, these definitions are: 
\begin{itemize}
    \item \textbf{Local hypoellipcity :} If $Lf$ is differentiable \emph{around} $x_0 \in M$, then $f$ is differentiable around $x_0$.
    \item \textbf{Global hypoellipticity:} If $Lf$ is differentiable in \emph{the whole manifold}  $M$, then $f$ is differentiable in the whole $M$.
\end{itemize}
Our interest here is the global hypoellipticity, which on compact smooth manifolds can actually be put very nicely in terms of the Fourier analysis associated to some elliptic operator. In particular for compact Lie groups the Fourier analysis is given in terms of the representation theory, where the matrix entries of the representations are $C^\infty$-eigenfunctions of the Laplace-Beltrami operator, considering the group as a compact smooth manifold. In that setting, the space of smooth functions can be completely described in terms of the Fourier transform: $f \in C^\infty(G)$ if and only if $\| \widehat{f} (\xi) \|_{HS} \lesssim (1 + \mu_{[\xi]})^{-k/2}$, for any $k \geq 0$ and any $[\xi] \in \widehat{G}$, where $\mu_{[\xi]}$ denotes the eigenvalue of the Laplace-Beltrami operator associated to $[\xi] \in \widehat{G}$. Here $\widehat{G}$ denotes the unitary dual of $G$, and $\widehat{f}(\xi)$ is defined as: $$\widehat{f}(\xi):= \int_G f(x) \xi^* (x)dx, \quad \|f(\xi) \|_{HS}:=Tr[\widehat{f}(\xi)\widehat{f}(\xi)^*]^{1/2},$$and $\| \cdot \|_{HS}$ denotes the \emph{Hilbert-Schmidt} norm of a linear operator. As a direct consequence, see for instance \cite{Kirilov2020}, an invariant pseudo-differential operator $$T_\sigma f(x) := \sum_{[\xi] \in \widehat{G}} d_\xi Tr[\xi (x)\sigma (\xi)\widehat{f}(\xi) ], \quad f \in C^\infty (G),$$defines a globally hypoelliptic operator if and only if, for all except finitely many $[\xi] \in \widehat{G}$  and some $C>0$, we have $$C (1 + \mu_{[\xi]})^{s/2} \leq \| \sigma(\xi)\|_{inf}:=\inf \{ \| \sigma (\xi)v\|_{\mathcal{H}_\xi} \, : \, \| v \|_{\mathcal{H}_\xi} = 1 \},$$where $s \in \R$, $\mu_{[\xi]}$ denotes the eigenvalue of the Laplace-Beltrami operator associated to $[\xi] \in \widehat{G}$, and $\mathcal{H}_\xi$ is the corresponding representation space.  

In the case of metrizable profinite groups, and even more general ultrametric spaces, the definition of differentiability is replaced by the Vladimirov-Taibleson operator as it is discussed in \cite{Dragovich2017}, which serves as an analogue to the Laplace-Beltrami operator from the real case. See for instance \cite{Bendikov2014}. In this paper we are interested specifically on nilpotent Lie groups over $\Z_p$, which we will always think of as $\Z_p^d$ endowed with some non-commutative operation "$\star$". In this kind of groups the Vladimirov-Taibleson operator on $\mathbb{G}$ has the particular form \[
D^\alpha f(\mathbf{x
}) := \frac{1 - p^\alpha}{1 - p^{- (\alpha + d)}} \int_{\mathbb{G}} \frac{f (\mathbf{x} \star \mathbf{y}^{-1}) - f(\mathbf{x})}{\|\mathbf{y} \|_p^{ \alpha + d}} d\mathbf{y}.
\]Here $dy$ denotes the unique normalized Haar measure on $\mathbb{G}$. The right analogue of a $C^\infty$-function is a \emph{Schwartz function}, which is a square integrable function $f$ with the property that $$D^\alpha f \in L^2 (\mathbb{G}), \quad \text{for all} \, \, \alpha >0.$$ Actually, on compact nilpotent $p$-adic Lie groups the right definition of Sobolev spaces is given in terms of the invariant operator operator \[
\mathbb{D}^\alpha f(\mathbf{x}) :=\frac{1-p^{-d}}{1-p^{-(\alpha +d)}}f(\mathbf{x}) + \frac{1 - p^\alpha}{1 - p^{- (\alpha + d)}} \int_{\mathbb{G}} \frac{f (\mathbf{x} \star \textbf{y}^{-1}) - f(\mathbf{x})}{\|\textbf{y} \|_p^{ \alpha + d}} d \mathbf{y},
\]where$$W^{k,p}(\mathbb{G}):= \{ f \in L^r (\mathbb{G}) \, : \, \mathbb{D}^\alpha f \in  L^r (\mathbb{G})\}, \quad \mathcal{S}(\mathbb{G}):= \bigcap_{k>0} W^{k,p}(\mathbb{G}), \quad 1 \leq p < \infty .$$It is not hard to see how the matrix coefficients of a compact nilpotent $p$-adic Lie group are actually eigenfunctions of $\mathbb{D}^\alpha$, and for a given $[\xi] \in \widehat{G}$, all the functions $\xi_{ij}$ share the same eigenvalues. We will denote by $\langle \xi \rangle_{\mathbb{G}}$ the eigenvalue of $\mathbb{D}^\alpha$ corresponding to $[\xi] \in \widehat{\mathbb{G}}$. The problem of global hypoellipticity on a compact nilpotent $p$-adic Lie groups has two different versions, due to the existence of two different notions of smooth functions:
\begin{itemize}
    \item What we call \emph{smooth functions} in this context, is the collection of locally constant functions on $\mathbb{G}$ with a fixed index of local constancy, which we denote along the paper by $\mathcal{D}(\mathbb{G})$. That is $$\quad \quad \quad \mathcal{D}(\mathbb{G}) = \{ f \in L^2 (\mathbb{G}) \, : \, \text{There is an} \, \, n_f \in \N_0 \, \, \text{such that} \,  \, f(\mathbf{x} \star \textbf{y}) = f(\mathbf{x}), \quad \text{for} \,\, \|\textbf{y}\|_p \leq p^{-n_f}  \}.$$This is going to be the initial domain of all of our operators, and its finite dimensional sub-spaces are very important for our analysis. Via Fourier transform, $f \in \mathcal{D}(\mathbb{G})$ if and only if $\widehat{f}(\xi) =0$ for all but finitely many $[\xi] \in \widehat{\mathbb{G}}$.
    \item The second relevant space is the \emph{Schwartz space}, which via Fourier transform can be written as $$\quad \quad \quad \mathcal{S}(\mathbb{G}) = \{ f \in L^2 (\mathbb{G}) \, : \, \| \widehat{f}(\xi)\|_{HS} \leq C \langle \xi \rangle_{\mathbb{G}}^{-k}, \, \text{for all but finitely many} \, \, [\xi] \in \widehat{G}\,\, \text{and all} \, \, k >0 \}.$$
    This is going to be the relevant space for our discussion, so that our definition of global hypoellipticity is the following: 
\end{itemize}
\begin{defi}\normalfont\label{defiGH}
    Let $\mathbb{G}$ be a compact nilpotent $p$-adic Lie group. We say that an invariant pseudo-differential operator $T_\sigma$ is a \emph{globally hypoelliptic operator}, if and only if the condition $T_\sigma f \in \mathcal{S}(\mathbb{G})$ implies that $f \in \mathcal{S}(\mathbb{G})$. 
\end{defi}

In a similar way as it is proven in \cite{Kirilov2020}, an invariant operator with associated symbol $\sigma$ is a globally hypoelliptic operator if and only if $$C \langle \xi \rangle_{\mathbb{G}}^\alpha \leq \|\sigma (\xi) \|_{inf}, \quad \text{for some} \, \, C>0, \, \, \alpha \in \R, \, \, \text{and all but finitely many} \, [\xi] \in \widehat{\mathbb{G}}.$$According to Theorem \ref{teogeneralresult}, if $\mathbb{G}$ has dimension $d \leq 5$ then the elements of $\widehat{\mathbb{G}}$ are indexed by some subset of $\widehat{\Z}_p^d \cong \Q_p^d / \Z_p^d$, and it also holds that \[\langle \xi \rangle_{\mathbb{G}}^\alpha = \begin{cases}
     \frac{1 - p^{-d}}{1 - p^{-(\alpha + d)}} , & \, \text{if} \, \| \xi \|_p = 1,\\
    \|\xi \|_p^\alpha I_{d_\xi}. \, & \, \text{in other case,}
\end{cases}\]for any $\alpha \in \R$. So, actually the condition for the global hypoellipticity in our setting is going to be $$C \| \xi \|_p^\alpha  \leq \| \widehat{f}(\xi) \|_{inf}, \quad \text{for all but finitely many} \, \, [\xi] \in \widehat{\mathbb{G}}, \,\, \text{and some} \, \alpha \in \R.$$Our goal now is to present some interesting operators for which the previous condition holds. For instance, according to Theorem \ref{TeoRepresentationsHd}, the Fourier series of a square integrable function $f \in L^2 (\mathbb{H}_d)$ on the Heisenberg group $\mathbb{H}_d$ has the form $$f(\mathbf{x},\mathbf{y},z) = \sum_{\lambda \in \widehat{\Z}_p} \sum_{(\xi , \eta) \in\Q_p^{2d} / {p^{\vartheta(\lambda)}}\Z_p^d} |\lambda|_p^d Tr [\pi_{(\xi ,\eta , \lambda)} (\mathbf{x},\mathbf{y}, z) \widehat{f} (\xi, \eta ,\lambda)],$$so that an invariant operator looks like $$T_\sigma f(\mathbf{x},\mathbf{y},z) = \sum_{\lambda \in \widehat{\Z}_p} \sum_{(\xi , \eta) \in\Q_p^{2d} / p^{\vartheta(\lambda)}\Z_p^d} |\lambda|_p^d Tr [\pi_{(\xi ,\eta , \lambda)} (\mathbf{x},\mathbf{y}, z) \sigma(\xi , \eta , \lambda) \widehat{f} (\xi, \eta ,\lambda)],$$where $\sigma$ is a mapping $$\sigma: \widehat{\mathbb{H}}_d \to \bigcup_{(\xi, \eta, \lambda) \in \widehat{\mathbb{H}}_d } \C^{|\lambda|_p^d \times |\lambda|_p^d}, \quad \sigma(\xi , \eta , \lambda) \in \C^{|\lambda|_p^d \times |\lambda|_p^d}.$$In this way, observing how the Vladimirov sub-Laplacian is an invariant operator, we can prove the following spectral theorem on $\mathbb{H}_d$:

\begin{teo}\label{TeoSpectrumSublaplacianHd}
Let $p>2$. Let $\mathfrak{h}_d$ be the $(2d+1)$-dimensional Heisenberg Lie algebra, with generators $$\{ X_1,...,X_d,Y_1,...,Y_d,Z\}, \, \, \, [X_i, Y_j] = \delta_{ij} Z,$$ and let $\textbf{V} = \{V_1,...,V_d \} \subset \mathrm{span}_{\Z_p} \{ X_1,...,X_d\}, \,\, \textbf{W} =  \{W_1,...,W_d \} \subset \mathrm{span}_{\Z_p} \{ Y_1,...,Y_d\}$ be collections of linearly independent vectors. The Vladimirov sub-Laplacian associated to this collection  $$T_{\textbf{V}, \textbf{W}}^{\alpha } : = \sum_{k=1}^\varkappa \partial_{V_k}^{\alpha} + \partial_{W_k}^{\alpha} ,$$defines a left-invariant, self-adjoint, globally hypoelliptic operator on $\mathbb{H}_d$. The spectrum of this operator is purely punctual, and its associated eigenfunctions form an orthonormal basis of $L^2(\mathbb{H}_d)$. Furthermore, the symbol of $T_{\textbf{V}, \textbf{W}}^{\alpha }$ acts on each representation space as a $p$-adic Schr{\"o}dinger operator, and the space $L^2(\mathbb{H}_d)$ can be written as the direct sum $$L^2(\mathbb{H}_d) = \overline{\bigoplus_{(\xi,\eta,\lambda) \in \widehat{\mathbb{H}}_d} \bigoplus_{h' \in  \Z_p^d / p^{-\vartheta(\lambda)} \Z_p^d} \mathcal{V}_{(\xi , \eta, \lambda)}^{h'}}, \, \,\, \mathcal{V}_{(\xi , \eta , \lambda)} = \bigoplus_{h' \in  \Z_p^d / p^{-\vartheta(\lambda)} \Z_p^d} \mathcal{V}_{(\xi , \eta, \lambda)}^{h'}, $$where each finite-dimensional sub-space$$\mathcal{V}_{(\xi , \eta, \lambda)}^{h'}:= \mathrm{span}_\C \{ (\pi_{(\xi , \eta , \lambda)})_{hh'} \, : \, h \in \Z_p^d / p^{-\vartheta(\lambda)} \Z_p^d \},$$is an invariant sub-space where $T_{\textbf{V}, \textbf{W}}^{\alpha }$ acts like the Schr{\"o}dinger-type operator operator $$ \sum_{k=1}^d \partial_{V_k}^{\alpha} + |W_k \cdot ( \lambda  h' + \eta ) |_p^{\alpha} - \frac{1 - p^{-1}}{1 - p^{-(\alpha +1)}}.$$
Consequently, the spectrum of $T_{\textbf{V}, \textbf{W}}^{\alpha }$ restricted to $\mathcal{V}_{(\xi , \eta, \lambda)}^{h'}$ is given by $$Spec(T_{\textbf{V}, \textbf{W}}^{\alpha }|_{\mathcal{V}_{(\xi , \eta, \lambda)}^{h}})= \Big\{ \sum_{k=1}^d |V_k \cdot(\tau + \xi)|_p^\alpha + |W_k \cdot ( \lambda  h' + \eta ) |_p^{\alpha} - 2\frac{1 - p^{-1}}{1 - p^{-(\alpha +1)}} \, : \, 1 \leq\| \tau \|_p \leq |\lambda|_p \Big\},$$
so that $Spec(T_{\textbf{V}, \textbf{W}}^{\alpha })$ is going to be the collection of real numbers \begin{align*}
    \sum_{k=1}^d |V_k \cdot(\tau + \xi)|_p^\alpha + |W_k \cdot ( \lambda  h' + \eta ) |_p^{\alpha} - 2\frac{1 - p^{-1}}{1 - p^{-(\alpha +1)}},
\end{align*}where $(\xi, \eta, \lambda) \in \widehat{\mathbb{H}}_d, \, h' \in \Z_p^d / p^{-\vartheta(\lambda)} \Z_p^d , \,  1 \leq \| \tau \|_p \leq |\lambda|_p$, and the corresponding eigenfunctions are given by  $$\mathscr{e}_{(\xi, \eta, \lambda) , h', \tau}(\mathbf{x}) = e^{ 2 \pi i \{  (\xi+\tau) \cdot \mathbf{x} +  \eta \cdot \mathbf{y} + \lambda(z +  h' \cdot \mathbf{y} )  \}_p}, \, \, \, \,(\xi, \eta, \lambda) \in \widehat{\mathbb{H}}_d, \, h' \in \Z_p^d / p^{-\vartheta(\lambda)} \Z_p^d , \,  1 \leq \| \tau \|_p \leq |\lambda|_p.$$
\end{teo} 
\begin{rem}
The restriction of the Vladimirov sub-Laplacian to each representation space acts in a similar way as a the $p$-adic Schrodinger operators studied in \cite{vladiBook}. However, in that reference the authors take a rather different approach, calculating the eigenspaces associated to a certain eigenvalue and estimating the associated multiplicities. Here we found a very convenient form to decompose the operator $D^\alpha + V$, but this is not very useful to find out some desirable information like the multiplicity of each eigenvalue. Unfortunately, investigating this kind of question is beyond the scope of this paper.       
\end{rem}

\subsection{Organization of the paper}In order to prove our results, we will organize our ideas in the following way: 

\begin{itemize}
    \item In Section 2 we recall the basic facts about $p$-adic analysis we will be using. A key detail in this section will be the introduction of $p$-adic Gaussian integrals studied in \cite{vladiBook}, which are essential for the analysis of certain characters of representations. Apart from that, we will recall in this section the basic elemets of the analysis on $\mathbb{H}_d$ and $\mathcal{B}_4$. With these two groups we have covered already dimensions $3$ and $4$, and one of the groups of dimension $5$: the $5$-dimensional Heisenberg group $\mathbb{H}_2 = \mathbf{exp}(\mathfrak{g}^{5,1})$. This serve us as a preparation to continue proving Theorem \ref{teogeneralresult} on the compact nilpotent $p$-adic Lie groups with dimension $d=5$, $\mathbb{G}^{5,i}= \mathbf{exp}(\mathfrak{g}^{5,i})$, $2 \leq i \leq 6$. Once again, we emphasize the importance of stating Theorems \ref{TeoRepresentationsHd} and \ref{TeoRepresentationsB4}, as the descriptions of the unitary duals of $\mathbb{H}_d$ and $\mathcal{B}_4$ will be essential for analyzing the duals of some of our $5$-dimensional groups, due to the fact that certain representations of $5$-dimensional groups descend to representations of lower-dimensional groups. For detailed proofs and the complete calculations on $\mathbb{H}_d$ and $\mathcal{B}_4$, we refer the reader to \cite{velasquezrodriguez2024spectrumvladimirovsublaplaciancompact, velasquezrodriguez2024Engelspectrumvladimirovsublaplaciancompact}.
     \item In Sections 3 to 7 we explain the proof of Theorem \ref{teogeneralresult} in the particular cases of $\mathbb{G}^{5,i}$, $1 \leq i \leq 6$. The arguments for each group follow the same lines as for $\mathbb{H}_d$, but sometimes it is necessary to introduce some extra details, like $p$-adic Gaussian integrals or Hensel's lemma. hopefully, this will provide some light about the more general case of nilpotent $p$-adic Lie groups over $\Z_p$.
    \item In Section 8 we discuss briefly our results, and some future challenges. 
\end{itemize}

\section{Preliminaries}
\subsection{$p$-adic numbers and Gaussian integrals}
Throughout this section $p>2$ will denote a fixed prime number. The field of $p$-adic numbers, usually denoted by $\Q_p$, can be defined as the completion of the field of rational numbers $\Q$ with respect to the $p$-adic norm $|\cdot|_p$, defined as \[|u|_p := \begin{cases}
0 & \esp \text{if} \esp u=0, \\ p^{-l} & \esp \text{if} \esp u= p^{l} \frac{a}{b},
\end{cases}\]where $a$ and $b$ are integers coprime with $p$. The integer $l:= \vartheta(u)$, with $\vartheta(0) := + \infty$, is called the \emph{$p$-adic valuation} of $u$. The unit ball of $\Q_p^d$ with the $p$-adic norm $$\|u \|_p:=\max_{1 \leq j \leq d} |u_j|_p,$$ is called the group of $p$-adic integers, and it will be denoted by $\Z_p^d$. Any $p$-adic number $u \neq 0$ has a unique expansion of the form $$u = p^{\vartheta(u)} \sum_{j=0}^{\infty} u_j p^j,$$where $u_j \in \{0,1,...,p-1\}$ and $u_0 \neq 0$. With this expansion we define the fractional part of $u \in \Q_p$, denoted by $\{u\}_p$, as the rational number\[\{u\}_p := \begin{cases}
0 & \esp \text{if} \esp u=0 \esp \text{or} \esp \vartheta(u) \geq 0, \\ p^{\vartheta(u)} \sum_{j=0}^{-\vartheta(u)-1} u_j p^j,& \esp \text{if} \esp ord(u) <0.
\end{cases}\] 
\begin{defi}\normalfont\label{defivaluation}
Given a vector $u \in \Q_p^d$, we define its associated \emph{$p$-adic valuation} as the minimum among the valuations $\vartheta(u_j)$, $1 \leq j \leq d$, that is $$\vartheta(u):= \min_{1 \leq j \leq d} \vartheta(u_j).$$ 
\end{defi}
The ring $\Z_p^d$ is  compact, totally disconected, i.e. profinite, and abelian. Its dual group in the sense of Pontryagin, the collection of characters of $\Z_p^d$, will be denoted by $\widehat{\Z}_p^d$. The dual group of the $p$-adic integers is known to be the Pr{\"u}fer group $\Z (p^{\infty})$,  the unique $p$-group in which every element has $p$ different $p$-th roots. The Pr{\"u}fer group may be identified with the quotient group $\Q_p/\Z_p$. In this way, each character of the group $\Z_p^d$ may be written as $$\chi_p (\tau  u) := e^{2 \pi i \{\tau \cdot u \}_p}, \esp \esp u \in \Z_p^d, \esp \tau \in \widehat{\Z}_p^d 
\cong \Q_p^d / \Z_p^d ,$$where $\tau$ is some fixed element of $\widehat{\Z}_p^d.$

By the Peter–Weyl theorem, the elements of $\widehat{\Z}_p^d$ constitute an orthonormal basis for the Hilbert space $L^2 (\Z_p^d)$, which provide us a Fourier analysis for suitable functions defined on $\Z_p$ in such a way that the formula $$\varphi(u) = \sum_{\tau \in \widehat{\Z}_p^d} \widehat{\varphi}(\tau) \chi_p (\tau  u),$$holds almost everywhere in $\Z_p$. Here $\mathcal{F}_{\Z_p^d}[\varphi]$ denotes the Fourier transform of $f$, in turn defined as $$\mathcal{F}_{\Z_p^d}[\varphi](\tau):= \int_{\Z_p^d} \varphi(u) \overline{\chi_p (\tau  u)}du,$$where $du$ is the normalised Haar measure on $\Z_p^d$. 

Another important tool from $p$-adic analysis that we are going to be using, is the calculation of the $p$-adic Gaussian integral on the disk, see \cite[p.p. 65]{vladiBook}. 
\begin{lema}\label{lemagaussian}
    Let $p \neq 2$ be a prime number, and take any $\gamma \in \Z$. Then, for any $a,b \in \Q_p$, we have \[ \int_{p^\gamma \Z_p} e^{2 \pi i \{ a u^2 + b u\}_p} du =  \begin{cases}
        p^{-\gamma} \1_{p^{-\gamma} \Z_p} (b), & \, \, \text{if} \, \, \, |a|_p \leq  p^{2 \gamma}, \\ \lambda_p(a) |a|_p^{-1/2} e^{2 \pi i \{ \frac{- b^2}{4a}  \}_p }\1_{p^{\gamma}\Z_p} (b/a)& \, \, \, \text{if} \, \, \, |a|_p >  p^{2 \gamma}.
    \end{cases}\]Here $\lambda_p : \Q_p \to \C$ is the function defined as \[\lambda_p (a) := \begin{cases}
        1, & \, \, \text{if} \, \, ord(u) \, \text{is even,} \\
         \big( \frac{a_0}{p} \big), & \, \, \text{if} \, \, ord(u) \, \text{is odd and} \, \, p \equiv 1 (\mathrm{mod} \, 4), \\ 
         i\big( \frac{a_0}{p} \big), & \, \, \text{if} \, \, ord(u) \, \text{is odd and} \, \, p \equiv 3 (\mathrm{mod} \, 4),
    \end{cases}\]
where $(\frac{a_0}{p})$ denotes the Legendre symbol of $a_0$. Here, to simplify our notation, we are going to define $$\Lambda (a,b):= \lambda_p(a) |a|_p^{-1/2} e^{2 \pi i \{ \frac{- b^2}{4a}  \}_p }, $$ so that $|\Lambda (a,b)|= |a|_p^{-1/2}$, and in this way $$\int_{p^\gamma \Z_p} e^{2 \pi i \{ a u^2 + b u\}_p} du =  p^{-\gamma} \1_{p^{-2\gamma} \Z_p}(a) \cdot \1_{p^{-\gamma} \Z_p} (b)    + \1_{\Q_p \setminus p^{-2\gamma} \Z_p}(a) \cdot \Lambda (a,b) \cdot \1_{p^{\gamma}\Z_p} (b/a) .  $$
\end{lema}

With this tool we can prove the following auxiliary lemma, which is simply the calculation of the $L^2$-norm of a function given in terms of $p$-adic Gaussian integral: 
\begin{lema}\label{lemaaux}
    Let $(x_2, x_3) \in \Z_p^2$, and take $\xi_3 , \xi_4 \in \widehat{\Z}_p$. Then $$\int_{\Z_p^2} \Big| \int_{ \Z_p}e^{2 \pi i \{(\xi_3   x_2 +  \xi_4 x_3)u +  \xi_4 x_2 \frac{u^2}{2} \}_p} du \Big|^2 dx_2 dx_3 = \max \{|\xi_3|_p , |\xi_4|_p \}^{-1} = \|(\xi_3 , \xi_4) \|_p^{-1}.$$
\end{lema}
\begin{proof}
    Let us define the auxiliary function $$f(x_2 , x_3):=\int_{ \Z_p}e^{2 \pi i \{(\xi_3   x_2 +  \xi_4 x_3)u +  \xi_4 x_2 \frac{u^2}{2} \}_p} du. $$Then, as a function on $\Z_p^2$, we can compute its Fourier transform  as \begin{align*}
        \mathcal{F}_{\Z_p^2}[f](\alpha , \beta)&= \int_{\Z_p^2} \int_{ \Z_p}e^{2 \pi i \{(\xi_3   x_2 +  \xi_4 x_3)u +  \xi_4 x_2 \frac{u^2}{2} - \alpha x_2 - \beta x_3 \}_p} du dx_2 dx_3\\&=\int_{ \Z_p}\int_{\Z_p^2} e^{2 \pi i \{(\xi_3 u  +  \xi_4 \frac{u^2}{2} - \alpha )x_2 - ( \xi_4 u - \beta)x_3 \}_p}  dx_2 dx_3du\\&= \int_{\Z_p} \1_{\Z_p} (\xi_3 u  +  \xi_4 \frac{u^2}{2} - \alpha) \1_{\Z_p} (\xi_4 u - \beta) du.  
    \end{align*}Notice how $|\xi_4 u - \beta|_p \leq 1 $ implies that $\xi_4 u =\beta$ in $\Q_p / \Z_p$. In this way, we have the following equality modulo $\Z_p$:  $$\xi_3 u  +  \xi_4 \frac{u^2}{2} = \frac{u}{2}\big( 2 \xi_3   +  \xi_4 u\big) = \frac{u}{2}\big( 2 \xi_3   +  \beta \big)  \in \Q_p / \Z_p,$$therefore \begin{align*}
        \mathcal{F}_{\Z_p^2}[f](\alpha , \beta)&= \int_{\Z_p} \1_{\Z_p} (u(  \xi_3   +  \beta/2) - \alpha) \1_{\Z_p} (\xi_4 u - \beta) du \\ &=\int_{\Z_p} \1_{\frac{\alpha}{\xi_3   +  \beta/2} + p^{-\vartheta(\xi_3   +  \beta/2)} \Z_p} (u) \1_{\frac{\beta}{\xi_4} + p^{-\vartheta(\xi_4)} \Z_p }(u)du \\ &= \mu_{\Z_p}\Big(  \big[ \frac{\alpha}{\xi_3   +  \beta/2} + p^{-\vartheta(\xi_3   +  \beta/2)} \Z_p \big] \cap \big[\frac{\beta}{\xi_4} + p^{-\vartheta(\xi_4)} \Z_p \big] \Big),
    \end{align*}where $\mu_{\Z_p^d}$ denotes the normalized Haar measure of $\Z_p^d$. To continue we need to see how, since we are integrating over $u\in \Z_p$, in order for the above integral to be different from zero we need $$\big| \frac{\alpha}{\xi_3   +  \beta/2}\big|_p \leq 1, \, \, \text{and} \, \, \, \big| \frac{\beta}{\xi_4}\big|_p \leq 1,$$so that only finitely many terms  $\mathcal{F}_{\Z_p^2}[f](\alpha , \beta)$ can be non-zero. Also \[ \mathcal{F}_{\Z_p^2}[f](\alpha , \beta) = \begin{cases}
    0, & \, \, \text{if the balls are disjoint}, \\ 
    \max\{|\xi_4|_p ,  |\xi_3   +  \beta/2|_p\}^{-1}, &\, \, \text{if the balls are not disjoint}.
\end{cases}\] 
Finally, we use this Fourier transform to calculate the $L^2$-norm of $f$ in two different cases. When $|\xi_3|_p >|\xi_4|_p$, we have $|\xi_3   +  \beta/2|_p = |\xi_3|_p$, so that  
\begin{align*}
    \| f \|_{L^2 (\Z_p^2)}^2 & = \sum_{(\alpha , \beta) \in \widehat{\Z}_p^2} \mu_{\Z_p}\Big( \big[\frac{\alpha}{\xi_3   +  \beta/2} + p^{-\vartheta(\xi_3   +  \beta/2)} \Z_p\big] \cap \big[ \frac{\beta}{\xi_4} + p^{-\vartheta(\xi_4)} \Z_p \big] \Big)^2 \\ &=\sum_{(\alpha , \beta) \in \widehat{\Z}_p^2} \mu_{\Z_p} \Big( \big[\frac{\alpha}{\xi_3   +  \beta/2} + p^{-\vartheta(\xi_3)} \Z_p \big] \cap \big[ \frac{\beta}{\xi_4} + p^{-\vartheta(\xi_4)} \Z_p \big] \Big)^2 \\ &=\sum_{\beta \in \widehat{\Z}_p}|\xi_3|_p^{-1} \sum_{\alpha \in \widehat{\Z}_p} \mu_{\Z_p}\Big( \big[ \frac{\alpha}{\xi_3   +  \beta/2} + p^{-\vartheta(\xi_3)} \Z_p \big] \cap \big[ \frac{\beta}{\xi_4} + p^{-\vartheta(\xi_4)} \Z_p \big] \Big) \\ &=|\xi_3|_p^{-1} \sum_{\beta \in \widehat{\Z}_p} \mu_{\Z_p} \Big( \frac{\beta}{\xi_4} + p^{-\vartheta(\xi_4)} \Z_p \Big) = |\xi_3|_p^{-1}.
\end{align*}
When $|\xi_3|_p \leq |\xi_4|_p,$we get $|\xi_3   +  \beta/2|_p \leq |\xi_4|_p$, so we conclude in a similar way

\begin{align*}
    \| f \|_{L^2 (\Z_p^2)}^2 & = \sum_{(\alpha , \beta) \in \widehat{\Z}_p^2} \mu_{\Z_p}\Big( \big[ \frac{\alpha}{\xi_3   +  \beta/2} + p^{-\vartheta(\xi_3   +  \beta/2)} \Z_p \big] \cap \big[ \frac{\beta}{\xi_4} + p^{-\vartheta(\xi_4)} \Z_p \big] \Big)^2 \\ &= \sum_{\alpha \in \widehat{\Z}_p} \sum_{\beta \in \widehat{\Z}_p} |\xi_4|_p^{-1} \mu_{\Z_p} \Big( \big[ \frac{\alpha}{\xi_3   +  \beta/2} + p^{-\vartheta(\xi_3  +  \beta/2)} \Z_p \big] \cap \big[ \frac{\beta}{\xi_4} + p^{-\vartheta(\xi_4)} \Z_p \big] \Big) \\ &= |\xi_4|^{-1}_p  \sum_{\alpha \in \widehat{\Z}_p} \mu_{\Z_p} \Big(\frac{\alpha}{\xi_3   +  \beta/2} + p^{-\vartheta(\xi_3 +  \beta/2)} \Z_p \Big) = |\xi_4|_p^{-1}.
\end{align*}
This concludes the proof. 
\end{proof}

A more general version of this lemma, which we need to handle a more complicated Gaussian integral, is the following: 

\begin{lema}\label{lemaauxG54}
Let $\xi \in \widehat{\Z}_p^3$ and $(x_1,x_2) \in \Z_p^2$. Consider the function $$f(x_1,x_2) = \int_{\Z_p}e^{2 \pi i \{ \xi_1 x_1 u^2 +(\xi_2 x_1 + \xi_3 x_2) \}_p} du.$$Then $\| f\|_{L^2(\Z_p^2)}^2 = \| \xi \|_p^{-1}.$
\end{lema}

\begin{proof}
We only need to see how \begin{align*}
    \mathcal{F}_{\Z_p^2}[f](\gamma_1, \gamma_2)&= \int_{\Z_p} \1_{\Z_p}(\xi_1u^2 + \xi_2  - \gamma_1 u)\1_{\Z_p}(\xi_3 u - \gamma_2) du \\& = \mu_{\Z_p}(\{u:\, \xi_1 u^2 + \xi_2 u - \gamma_1 \in \Z_p \} \cap \{ u: \, \xi_3u - \gamma_2 \in \Z_p \} ),
\end{align*}where $\mu_{\Z_p}$ denotes the unique normalized Haar measure on $\Z_p$. Now we consider cases: 
\begin{itemize}
    \item We know that $\xi_1 u (u+\xi_2/\xi_1) - \gamma_1 \in \Z_p$ is possible only for $|\gamma_1|_p \leq \|(\xi_1, \xi_2)\|_p$. Also, it should be clear that $$\sum_{|\gamma_1|_p \leq \|(\xi_1,\xi_2)\|_p} \mu_{\Z_p}(\{u:\, \xi_1 u^2 + \xi_2 u - \gamma_1 \in \Z_p \}) =1 .$$Since it holds $$h_1 (u)= \xi_1 u^2 + \xi_2 u, \quad h_1(u+v)=h_1(u), \quad \text{for} \,\, |v|_p \leq \| (\xi_1, \xi_2)\|_p^{-1},$$each set $\{u:\, \xi_1 u^2 + \xi_2 u - \gamma_1 \in \Z_p \}$ should me expressible as a characteristic function of a $p$-adic ball of radius $\| (\xi_1 , \xi_2)\|_p^{-1}$.
    \item Similarly, $\xi_3u - \gamma_2 \in \Z_p$ is possible only for $|\gamma_2|_p \leq |\xi_3|_p$. Also, $$\sum_{|\gamma_2|_p \leq |\xi_3|_p} \mu_{\Z_p}(\{u:\, \xi_3 u^2  - \gamma_2 \in \Z_p \}) =1 .$$Since it holds $$h_2 (u)= \xi_3 u , \quad h_2(u+v)=h_2(u), \quad \text{for} \,\, |v|_p \leq | \xi_3|_p^{-1},$$each set $\{u:\, \xi_3 u - \gamma_2 \in \Z_p \}$ should be expressible as a characteristic function of a $p$-adic ball of radius $| \xi_3|_p^{-1}$.
\end{itemize}
In conclusion, each set $$\{u:\, \xi_1 u^2 + \xi_2 u - \gamma_1 \in \Z_p \} \cap \{ u: \, \xi_3u - \gamma_2 \in \Z_p \}$$is the intersection of two balls, at least one of them with measure $\| \xi \|_p^{-1}$, so that we can write $$\widehat{f}(\gamma_1 , \gamma_2) = \mathbf{N}_\xi (\gamma_1 , \gamma_2) \| \xi \|_p^{-1},$$ where $\mathbf{N}_\xi (\gamma_1 , \gamma_2)$ is always zero or one, depending on whether the intersections of the balls are empty or not. In conclusion:
\begin{align*}
    \sum_{(\gamma_1 , \gamma_2) \in \widehat{\Z}_p^2} |\mathcal{F}_{\Z_p^2}[f](\gamma_1, \gamma_2)|^2 &=\sum_{(\gamma_1 , \gamma_2) \in \widehat{\Z}_p^2} \mathbf{N}_\xi^2 (\gamma_1 , \gamma_2) \|\xi\|_p^{-2} \\ &=\|\xi\|_p^{-1}\sum_{(\gamma_1 , \gamma_2) \in \widehat{\Z}_p^2} \mathbf{N}_\xi (\gamma_1 , \gamma_2) \|\xi\|_p^{-1}=\|\xi\|_p^{-1},
\end{align*}where we used the fact that $\mathbf{N}_\xi^2 = \mathbf{N}_\xi$. The proof is concluded. 
\end{proof}

\subsection{The Heisenberg group over $\Z_p$}
Let $p>2$ be a prime number. Let us denote by $\mathbb{H}_d (\Z_p)$ the $(2d+1)$-dimensional Heisenberg group over $\Z_p$, or simply $\mathbb{H}_d $ for short, here defined as \[
\mathbb{H}_{d}(\Z_p)= \left\{
  \begin{bmatrix}
    1 & \mathbf{x}^t & z \\
    0 & I_{d} & \mathbf{y} \\
    0 & 0 & 1 
  \end{bmatrix}\in \mathrm{GL}_{d+2}(\Z_p) \, : \, \mathbf{x} , \mathbf{y} \in \Z_p^d, \, z \in \Z_p \right\}. 
\]
Clearly $\mathbb{H}_{d}(\Z_p)$ is a compact analytic $d$-dimensional manifold, which is homeomorphic to $\Z_p^{2d+1}$. Moreover, the operations on $\mathbb{H}_{d}(\Z_p)$ are analytic functions, making $\mathbb{H}_{d}(\Z_p)$ a $p$-adic Lie group. Let us denote by $\mathfrak{h}_{d}(\Z_p)$ its associated $\Z_p$-Lie algebra. We can write explicitly \[
\mathfrak{h}_{d}(\Z_p)= \left\{
  \begin{bmatrix}
    0 & \textbf{a}^t & c - \frac{\textbf{a} \cdot \textbf{b}}{2} \\
    0 & 0_{d} & \textbf{b} \\
    0 & 0 & 0 
  \end{bmatrix}\in \mathcal{M}_{d+2}(\Z_p) \, : \, \textbf{a} , \textbf{b} \in \Z_p^d, c \in \Z_p \right\}. 
\]Recall how for an element of the Lie algebra \[u:= \begin{bmatrix}
    0 & \textbf{a}^t & c - \frac{\textbf{a} \cdot \textbf{b}}{2}\\
    0 & 0_{d} & \textbf{b} \\
    0 & 0 & 0 
  \end{bmatrix},\]the exponential map evaluates to \[\textbf{exp} (u) = \begin{bmatrix}
    1 & \textbf{a}^t & c  \\
    0 & I_{d} & \textbf{b} \\
    0 & 0 & 1 
  \end{bmatrix}.\]The exponential map transform sub-ideals of the Lie algebra $\mathfrak{h}_{d} $ to subgroups of $\mathbb{H}_{d} $. Actually, we can turn the exponential map into a group homomorphism by using the Baker–Campbell–Hausdorff formula. Let us define the operation ``$\star$" on $\mathfrak{h}_{d}$ by $$X \star Y:= X + Y + \frac{1}{2} [X,Y].$$ Then clearly $(\mathfrak{h}_{d}  , \star) \cong \mathbb{H}_d $ is a profinite topological group, and it can be endowed with the sequence of subgroups $J_n := ( \mathfrak{h}_{d}(p^n\Z_p),\star)$, where $$ \mathfrak{h}_{d}(p^n \Z_p)= p^n \Z_p X_1 +...+ p^n \Z_p X_d + p^n \Z_p Y_1 + ...+ p^n\Z_p Y_{d} + p^{n} \Z_p Z,$$so $\mathbb{H}_d$ is a compact Vilenkin group, together with the sequence of compact open subgroups $$G_n := \mathbb{H}_d (p^n \Z_p)= \textbf{exp}( \mathfrak{h}_{d}(p^n \Z_p)), \,\,\, n \in \N_0.$$
  Notice how the sequence $\mathscr{G}=\{G_n\}_{n \in \N_0}$ forms a basis of neighbourhoods at the identity, so the group is metrizable, and we can endow it with the natural ultrametric \[ |(\mathbf{x},\mathbf{y},z)\star(\mathbf{x}',\mathbf{y}',z')^{-1}|_{\mathscr{G}} :=\begin{cases} 0 & \, \, \text{if} \, (\mathbf{x},\mathbf{y},z)=(\mathbf{x}',\mathbf{y}',z'), \\ |G_n| = p^{-n(2d+1)}  & \, \, \text{if} \, (\mathbf{x},\mathbf{y},z) \star(\mathbf{x}',\mathbf{y}',z')^{-1} \in G_n \setminus G_{n+1}.\end{cases}\]   
Nevertheless, instead of this ultrametric we will use the $p$-adic norm $$\| (\mathbf{x},\mathbf{y},z) \|_p:= \max \{\|\mathbf{x}\|_p, \| \mathbf{y} \|_p , |z|_p  \}.$$ 
\begin{rem}\label{remp-adicnormonG}
Notice that how $\|(\mathbf{x},\mathbf{y},z) \|_p^{2d+1} = |(\mathbf{x},\mathbf{y},z)|_{\mathscr{G}},$ for any $(\mathbf{x},\mathbf{y},z) \in \mathbb{H}_d,$ and we also have that \begin{align*}
     \|(\mathbf{x},\mathbf{y},z)\star(\mathbf{x}',\mathbf{y}',z')^{-1} \|_p= \| (\mathbf{x} - \mathbf{x'}, \mathbf{y} - \mathbf{y'}, z - z' + \mathbf{x}' \cdot \mathbf{y}'- \mathbf{x} \cdot \mathbf{y'} )\|_p.
\end{align*}From the above it is clear how $$\|(\mathbf{x},\mathbf{y},z)\star(\mathbf{x}',\mathbf{y}',z')^{-1} \|_p = 0 \iff (\mathbf{x},\mathbf{y},z) = (\mathbf{x}',\mathbf{y}',z'),$$proving that $\| \cdot \|_p$ is a valid choice for an alternative distance function on $\mathbb{H}_d$. The reason to use the $p$-adic norm instead of the Vilenkin distance function is to make the Vladimirov-Taibleson operator to look the same as for the abelian case  $\mathbb{G} = \Z_p^d$.     
\end{rem}

\subsection{The Engel group over $\Z_p$}
Let $p>2$ be a prime number. Let us denote by $\mathfrak{b}_4$ the $\Z_p$-Lie algebra generated by $X_1,..,X_4$, with the commutation relations $$[X_1 , X_2] = X_3, \,\,\, [X_1 , X_3 ] = X_4, \,\,\, \mathfrak{b}_4 = \bigoplus_{j=1}^3 \mathcal{V}_j, $$where $$\mathcal{V}_1 = \mathrm{span}_{\Z_p} \{ X_1 , X_2 \}, \,\,\, \mathcal{V}_1 = \mathrm{span}_{\Z_p} \{ X_3 \}, \,\,\, \mathcal{V}_3 = \mathrm{span}_{\Z_p} \{ X_4 \}.$$We call the $\Z_p$-Lie algebra $\mathfrak{b}_4$ the $4$-dimensional \emph{Engel algebra}, and its exponential image, which we denote here by $\mathcal{B}_4$, is called the \emph{Engel group} over the $p$-adic integers.    
Let us consider the realization of $\mathfrak{b}_4$ as the matrix algebra  \[
\mathfrak{b}_{4}(\Z_p)= \left\{
  \begin{bmatrix}
    0 & x_1 & 0 & x_4 -\frac{1}{2}x_1 (x_3  - \frac{1}{2}x_1 x_2) - \frac{1}{6}x_1^2 x_2 \\
    0 & 0 & x_1 & x_3 - \frac{1}{2}x_1 x_2 \\
    0 & 0 & 0 & x_2  \\
    0 & 0 & 0 & 0 \\
  \end{bmatrix}\in \mathcal{M}_{4}(\Z_p) \, : \, x \in \Z_p^4 \right\}. 
\]
With this realization, and using the usual matrix exponential map, we can think on $\mathcal{B}(\Z_p)$ as the matrix group \[
\mathcal{B}_{4}(\Z_p)= \left\{
  \begin{bmatrix}
    1 & x_1 & \frac{1}{2}x_1^2 & x_4  \\
    0 & 1 & x_1 & x_3  \\
    0 & 0 & 1 & x_2  \\
    0 & 0 & 0 & 1 \\
  \end{bmatrix}\in \mathcal{M}_{4}(\Z_p) \, : \, x \in \Z_p^4 \right\}, 
\]
which is analytically isomorphic to $\Z_p^4$ with the operation $$\mathbf{x}\star \mathbf{y}:=(x_1 + y_1, x_2 + y_2, x_3 + y_3 + x_1 y_2, x_4 + y_4 + x_1 y_3 + \frac{1}{2} x_1^2 y_2).$$

The exponential map transforms sub-ideals of the Lie algebra $\mathfrak{b}_{4} $ to subgroups of $\mathcal{B}_4 \cong (\mathfrak{b}_4 , *)$, which can be endowed with the sequence of subgroups $J_n := (  \mathfrak{b}_{4}(p^n\Z_p),*)$, where $$ \mathfrak{b}_{4}(p^n\Z_p)= p^n \Z_p X_1 + p^n \Z_p X_2 + p^n \Z_p X_3 + p^n\Z_p X_{4},$$so $\mathcal{B}_4$ is a compact Vilenkin group, together with the sequence of compact open subgroups $$G_n := \mathcal{B}_4 (p^n \Z_p)= \textbf{exp}( \mathfrak{b}_{4}(p^n\Z_p)), \,\,\, n \in \N_0.$$
  Notice how the sequence $\mathscr{G}=\{G_n\}_{n \in \N_0}$ forms a basis of neighbourhoods at the identity, so the group is metrizable, and we can endow it with the natural ultrametric \[ |\mathbf{x} \star \mathbf{y}^{-1}|_{\mathscr{G}} :=\begin{cases} 0 & \, \, \text{if} \, \mathbf{x}=\mathbf{y}, \\ |G_n| = p^{-4n}  & \, \, \text{if} \, \mathbf{x}\star \mathbf{y}^{-1} \in G_n \setminus G_{n+1}.\end{cases}\]   
Nevertheless, instead of this ultrametric we will use the $p$-adic norm $$\| \mathbf{x} \|_p:= \max_{1 \leq j \leq 4} \|x_j\|_p, .$$ Notice that how $\|\mathbf{x} \|_p^{4} = |\mathbf{x}|_{\mathscr{G}},$ for any $\mathbf{x} \in \mathcal{B}_4$

The group $\mathcal{B}_4$ is another important example where a version of Theorem \ref{teogeneralresult} holds. In dimension $5$, given that nilpotent groups of nilpotency index \(\mathscr{N}\) are extensions of nilpotent groups with lower indices and dimensions, the representation theory of $5$-dimensional groups often reduces to that of dimensions $3$ and $4$. Dimension $3$ corresponds to the \emph{compact Heisenberg group} $\mathbb{H}_1(\mathbb{Z}_p)$, while dimension $4$ is represented by $\mathcal{B}_4$, so that we need to include the representation theory of $\mathcal{B}_4$ too: 
\begin{teo}\label{TeoRepresentationsB4}
Let $\mathcal{B}_4(\Z_p)$, or simply $\mathcal{B}_4$ for short, be the $4$-dimensional Engel group over the $p$-adic integers. Let us denote by $\widehat{\mathcal{B}}_4$ the unitary dual of $\mathcal{B}_4$, i.e., the collection of all unitary irreducible representations of $\mathcal{B}_4$. Then we can identify $\widehat{\mathcal{B}}_4$ with the following subset of $\widehat{\Z}_p^{4} \cong \Q_p^{4}/\Z_p^{4}$:  
$$\widehat{\mathcal{B}}_4:= \{ \xi  \in \widehat{\Z}_p^4 \, : \, 1 \leq |\xi_4|_p <|\xi_3|_p \, \wedge \,  (\xi_1 , \xi_2 , \xi_3) \in \widehat{\mathbb{H}}_1 , \, \text{or} \,  |\xi_3|_p = 1 \,  \wedge  \,  \xi_1 \in \Q_p / p^{\vartheta(\xi_4)} \Z_p \}.$$
Moreover, each non-trivial representation $[\pi_\xi]$ is equivalent to one of the representation $[ \pi_\xi ] \in \widehat{\mathcal{B}}_4 $ which can be realized in the following finite dimensional sub-space $\mathcal{H}_\xi$ of $L^2(\Z_p)$:
\[ \mathcal{H}_\xi :=  \mathrm{span}_\C \{ \|(\xi_3 ,\xi_4) \|_p^{1/2} \1_{h + p^{-\vartheta(\xi_3 , \xi_4)} \Z_p } \,\, : \, h \in \Z_p / p^{-\vartheta(\xi_3 , \xi_4)}\Z_p \},    \]
where $d_\xi:= dim_\C (\mathcal{H}_\xi) = \max\{ |\xi_3|_p , |\xi_4|_p \}$, and the representation acts on functions $\varphi \in \mathcal{H}_{\xi}$ according to the formula 
\[\pi_{\xi}(\mathbf{x}) \varphi (u) := 
    e^{2 \pi i \{\xi_1 x_1 + \xi_2 x_2 + \xi_3 (x_3 +  u x_2)  + \xi_4 (x_4 +  u x_3 + \frac{u^2}{2} x_2) \}_p} \varphi (u + x_1)  . \]
With this explicit realization of $[\pi_\xi]  \in \widehat{\mathcal{B}}_4 $, and with the natural choice of basis for each representation space, the associated matrix coefficients $(\pi_{\xi})_{h h'}$ are given by \[(\pi_{\xi}(\mathbf{x}))_{hh'} = e^{2 \pi i \{\xi \cdot \mathbf{x} +(\xi_3   x_2 +  \xi_4 x_3)(h') +  \xi_4 x_2 \frac{(h')^2}{2} \}_p} \1_{h' - h + p^{\vartheta(\xi_3 , \xi_4)} \Z_p} (x_1) . \]
Sometimes we will use the notation $$\mathcal{V}_{\xi}:= \mathrm{span}_\C \{ (\pi_{\xi})_{h h'} \, : \, h,h' \in \Z_p/p^{-\vartheta(\xi_3 , \xi_4)} \Z_p \}.$$See again Definition \ref{defivaluation} for the definition of the $p$-adic valuation $\vartheta$.
\end{teo}
Theorem \ref{TeoRepresentationsB4} provides us with the following Fourier series representation of a function $f \in L^2 (\mathcal{B}_4)$: \begin{align*}
    f(\mathbf{x}) &= \sum_{\, \|(\xi_3 , \xi_4) \|_p = 1 } \widehat{f}(\xi) e^{2 \pi i \{ \xi \cdot \mathbf{x} \}_p} + \sum_{ \xi_4 \in \widehat{\Z}_p } \sum_{  |\xi_4|_p < |\xi_3|_p}  \sum_{(\xi_1 , \xi_2) \in \Q_p^2 / p^{\vartheta(\xi_3)}\Z_p^2} |\xi_3|_p Tr[\pi_\xi (\mathbf{x}) \widehat{f}(\xi)] \\ & \quad \quad +\sum_{\,|\xi_4|_p > |\xi_3|_p=1 } \sum_{\xi_2 \in \widehat{\Z}_p} \sum_{\xi_1 \in \Q_p/ p^{\vartheta(\xi_4)}\Z_p}|\xi_4|_p Tr[\pi_\xi (\mathbf{x}) \widehat{f}(\xi)],
\end{align*}
and in a similar way as for the Heisenberg group, we want to use this series representation to study linear invariant operators. For instance, with the description of the unitary dual given in Theorem \ref{TeoRepresentationsB4}, we can prove the following spectral theorem for the Vladimirov sub-Laplacian on $\mathcal{B}_4$:

\begin{teo}\label{TeoSpectrumSublaplacianB4}
Let $p>3$. The Vladimirov sub-Laplacian associated to the basis $\{X_1 ,X_2\}$ of the first stratum of $\mathfrak{b}_4$ $$\mathscr{L}_{sub}^{\alpha } : =  \partial_{X_1}^{\alpha} + \partial_{X_2}^{\alpha} ,$$defines a left-invariant, self-adjoint, sub-elliptic operator on $\mathcal{B}_4$. The spectrum of this operator is purely punctual, and its associated eigenfunctions form an orthonormal basis of $L^2(\mathcal{B}_4)$. Furthermore, the symbol of $\mathscr{L}_{sub}^{\alpha }$ acts on each representation space as a $p$-adic Schr{\"o}dinger type operator, and the space $L^2(\mathcal{B}_4)$ can be written as the direct sum $$L^2(\mathcal{B}_4) = \overline{\bigoplus_{[\xi] \in \widehat{\mathcal{B}}_4} \bigoplus_{h' \in  \Z_p / p^{-\vartheta(\xi_3 , \xi_4)} \Z_p} \mathcal{V}_{\xi}^{h'}}, \, \,\, \mathcal{V}_{\xi} = \bigoplus_{{h'} \in  \Z_p / p^{-\vartheta(\xi_3 , \xi_4)} \Z_p} \mathcal{V}_{\xi}^{h'}, $$where each finite-dimensional sub-space$$\mathcal{V}_{\xi}^{h'}:= \mathrm{span}_\C \{ (\pi_{\xi})_{hh'} \, : \, h \in \Z_p / p^{-\vartheta(\xi_3 , \xi_4)} \Z_p \},$$is an invariant sub-space where $\mathscr{L}_{sub}^{\alpha }$ acts like the Schr{\"o}dinger-type operator operator $$  \partial_{x_1}^{\alpha} + Q(\xi , h') - \frac{1 - p^{-1}}{1 - p^{-(\alpha +1)}}.$$
Consequently, the spectrum of $\mathscr{L}_{sub}^{\alpha }$ restricted to $\mathcal{V}_{\xi}^{h'}$ is purely punctual, and it is given by \[\begin{cases} |\xi_1|_p^\alpha + |\xi_2|^\alpha_p - 2\frac{1 - p^{-1}}{1 - p^{-(\alpha +1)}} \quad & \text{if} \quad |\xi_3|_p = |\xi_4 |_p =1,   \\ |\xi_1 + \tau|_p^{\alpha} +| \xi_2  + \xi_3    h'  |_p^\alpha - 2\frac{1 - p^{-1}}{1 - p^{-(\alpha +1)}} \quad & \text{if} \quad 1= |\xi_4|_p <|\xi_3|_p,  \\|\xi_1 + \tau|_p^{\alpha} +| \xi_2  + \xi_3    h' +  \xi_4  \frac{(h')^2}{2} |_p^\alpha - 2\frac{1 - p^{-1}}{1 - p^{-(\alpha +1)}} \quad & \text{if} \quad 1< |\xi_4|_p <|\xi_3|_p, \\ |\xi_1 + \tau|_p^{\alpha} +| \xi_2  + \xi_4  \frac{(h')^2}{2} |_p^\alpha - 2\frac{1 - p^{-1}}{1 - p^{-(\alpha +1)}} \quad & \text{if} \quad 1= |\xi_3|_p <|\xi_4|_p,
\end{cases} \]where $\tau \in \Z_p / p^{-\vartheta(\xi_3 , \xi_4)}\Z_p$, and the corresponding eigenfunctions are given by  $$\mathscr{e}_{\xi, h', \tau}(\mathbf{x}): = e^{ 2 \pi i \{ \xi x  + \tau x_1 +(\xi_3   x_2 +  \xi_4 x_3)(h') +  \xi_4 x_2 \frac{(h')^2}{2} \}_p},$$ 
where $$\xi \in \widehat{\mathcal{B}}_4, \, h '\in \Z_p / p^{-\vartheta(\xi_3 , \xi_4)} \Z_p, \,  1 \leq | \tau |_p \leq \| (\xi_3 , \xi_4)\|_p .$$
\end{teo}

\subsection{Notation}
Throughout this section $\mathbb{G}$ will denote a compact nilpotent $p$-adic Lie group with dimension $d \leq 5$. In order to study the behavior of the directional VT operators, we need to introduce first several important definitions and notation. First, we have the notation about representation theory of compact Vilenkin groups.

\begin{defi}\normalfont\label{definotrep}
\,
\begin{itemize}
        \item The symbol $\mathrm{Rep}(\mathbb{G})$ will denote the collection of all (equivalence classes of) unitary finite-dimensional representations of $\mathbb{G}$. We will denote by $\widehat{\mathbb{G}}$ the \emph{unitary dual of $\mathbb{G}$}, that is, the collection of all equivalence classes of $\mathrm{Rep}(\mathbb{G})$.
        \item Let $K$ be a normal sub-group of $\mathbb{G}$. We denote by $K^\bot$ the \emph{anihilator of $K$}, here defined as $$K^\bot:= \{ [\pi] \in \mathrm{Rep}(\mathbb{G}) \, : \, \pi|_{K}=I_{d_\pi} \}.$$Also, we will use the notation $$B_{\widehat{\mathbb{G}}}(n):=\widehat{\mathbb{G}} \cap \mathbb{G}(p^n\Z_p)^\bot,$$and $\widehat{\mathbb{G}}(n):= B_{\widehat{\mathbb{G}}}(n) \setminus B_{\widehat{\mathbb{G}}}(n-1).$
\end{itemize}
\end{defi}
\begin{rem}\label{remball}
Recall how $\widehat{\Z}_p^d \cong \Q_p^d / \Z_p^d$ is actually an ultrametric space with the $p$-adic norm, and the balls of $\widehat{\Z}_p^d$ are precisely the balls of $\Q_p^d$ modulo $\Z_p^d$. For $\mathbb{H}_d$ and $\mathcal{B}_4$, their unitary duals can be identified with a certain subset of $\widehat{\Z}_p^{2d+1}$ and $\widehat{\Z}_p^4$ respectively, and this implies the same for any compact nilpotent $p$-adic Lie group with dimension $d \leq 4$. So, at least for $d \leq 4$, we already know how $$B_{\widehat{\mathbb{G}}}(n) = \{ \pi_\xi \, : \, \, \| \xi \|_p \leq p^n \} = \widehat{\mathbb{G}} \cap B_{\widehat{\Z}_p^d}(n),$$where $$B_{\widehat{\Z}_p^d}(n) := \{ \xi \in \Q_p^d / \Z_p^d \, \, : \, \, \| \xi \|_p \leq p^n \} = p^{-n}\Z_p / \Z_p, \, \, \, n \in \N.$$This highlights the motivation for our notation, which is the fact that for $d \leq 4$ the anihilators of balls in $\mathbb{G}$ are balls in the unitary dual $\mathbb{G}$. Similarly, the sets $\mathbb{G}(n)$ are spheres. Here in this paper we will prove the same for compact nilpotent groups with dimension $d=5$.
\end{rem}
A particularity of compact $p$-adic Lie groups, in contrast with their real counterparts, is how their continuous representations are actually \emph{smooth functions}. We say that a function $f:\mathbb{G} \to \C$ is \emph{smooth} if $f$ is a locally constant function with a fixed index of local constancy, i.e., there is an $n_f \in \N_0$, which we always choose to be the minimum possible, such that $$f(\mathbf{x}\star\textbf{y}) = f(\mathbf{x}), \, \, \,\text{ for all} \,\,\, \textbf{y} \in \mathbb{G}(p^{n_f} \Z_p).$$Here $\star$ is the operation on $\mathbb{G}$. The reason why any $[\pi] \in \mathrm{Rep}(\mathbb{G})$ must be smooth is that, due to their continuity and the lack of small subgroups for the general group $\mathrm{GL}(\mathcal{H}_\pi)$, any unitary representation $[\pi]: \mathbb{G} \to  \mathrm{GL} (\mathcal{H}_\pi)$ must map a certain ball $\mathbb{G}(p^n \Z_p), n \in \N,$ to the identity element. In other words, any continuous representation of $\mathbb{G}$ must have a non-trivial kernel, which is a compact open subgroup of $\mathbb{G}$. This is the reason why Definition \ref{definotrep} makes sense. We will denote by $\mathcal{D}(\mathbb{G})$ the collection of all smooth complex-valued functions on $\mathbb{G}$, and $\mathcal{D}_n (\mathbb{G})$ will denote the collection of smooth functions with index of local constancy equal to $n \in \N_0$. 

By the Peter-Weyl theorem, one has a definition of Fourier series on general compact groups, which is given in terms of the representation theory of noncommutative groups. The matrix entries of the unitary irreducible representations span a dense subspace of $L^2 (\mathbb{G})$, and we can write any $f \in L^2 (\mathbb{G})$ as $$f(\mathbf{x}):= \sum_{[\xi] \in \widehat{\mathbb{G}}} d_\xi Tr[ \xi(\mathbf{x}) \widehat{f} (\xi)].$$A powerful implication of this series the representations is that, if a given linear operator turns out to be invariant, let's say it commutes with left translations, then it is diagonalized by the matrix entries of the representations. And even if the operator is not that compatible with the group structure, we can still exploit the Fourier series of the group to write it down in a way that simplifies many problems. See for instance \cite{Ruzhansky2010}, where it is shown how a densely defined linear operator on a compact group has the following symbolic representation:    $$Tf(\mathbf{x}):= \sum_{[\xi] \in \widehat{\mathbb{G}}} d_\xi Tr[ \xi(\mathbf{x}) \sigma_T (\mathbf{x}, \xi) \widehat{f} (\xi)].$$
Here the symbol $\sigma_T$ is a matrix valued map which we define formally below. 

\begin{defi}\normalfont\label{defisymbol}
\, 
\begin{itemize}
    \item We will use the notation $\mathcal{D}'(\mathbb{G})$ for the space of \emph{smooth distributions} on $\mathbb{G}$, that is, the collection of all continuous linear functionals $F: \mathcal{D}(\mathbb{G}) \to \C$. Here $\mathcal{D}(\mathbb{G})$ is considered a topological vectors space with the topology of convergence over finite-dimensional sub-spaces.  
    \item A \emph{symbol $\sigma$} is a mapping $$\sigma: \mathbb{G} \times \mathrm{Rep}(\mathbb{G}) \to \bigcup_{[\pi] \in \mathrm{Rep}(G)} \mathcal{L}(\mathcal{H}_{\pi}), \,\,\,\, (\mathbf{x},[\pi]) \mapsto \sigma(\mathbf{x}, \pi) \in \mathcal{L}(\mathcal{H}_\pi).$$Given a symbol on $\mathbb{G}$, we define its associated pseudo-differential operator as the linear operator $T_\sigma$ acting on $\mathcal{D}(\mathbb{G})$ via the formula $$T_\sigma f(\mathbf{x}):= \sum_{[\xi] \in \widehat{\mathbb{G}}} d_\xi Tr[ \xi(\mathbf{x}) \sigma (\mathbf{x},\xi) \widehat{f} (\xi)].$$
    \item Conversely, given a densely defined linear operator $T : \mathcal{D} (\mathbb{G}) \subset D(T) \to \mathcal{D}'(\mathbb{G})$, we define its associated symbol via the formula $$\sigma(\mathbf{x}, [\pi]) = \pi^* (\mathbf{x}) T \pi (\mathbf{x}), \quad [\pi] \in \mathrm{Rep}(\mathbb{G}).$$ 
\end{itemize}
\end{defi}
\begin{rem}\label{remindependencesymbol}
The symbol $\sigma_T$ of a linear operator is well defined. To see it, take $\pi, \pi' \in [\xi] \in \widehat{\mathbb{G}}$. Then there is some unitary linear mapping $M$ such that $M^{*} \pi M = \pi'$, and for that reason $$\widehat{f}(\pi') = \int_{\mathbb{G}} f(\mathbf{x}) (\pi')^*(\mathbf{x}) d \mathbf{x}= \int_{\mathbb{G}} f(\mathbf{x}) M^* \pi^*(\mathbf{x}) M d \mathbf{x} =  M^* \widehat{f}(\pi) M.$$Hence, the Fourier series of a function are well defined and independent of the representative chosen for each class, because $$Tr(\pi'(\mathbf{x}) \widehat{f}(\pi')) = Tr(M^* \pi(\mathbf{x}) M M^*  \widehat{f}(\pi)M)=Tr(M^* \pi(\mathbf{x})  \widehat{f}(\pi)M) = Tr(\pi(\mathbf{x}) \widehat{f}(\pi)).$$Similarly $$\sigma_T (\mathbf{x}, \pi') = M^*  \pi^* (\mathbf{x}) T \pi (\mathbf{x}) M = M^* \sigma(\mathbf{x}, \pi) M,$$so we can see that $$Tr(\pi'(\mathbf{x}) \sigma_T (\mathbf{x}, \pi') \widehat{f}(\pi')) =Tr(M^* \pi(\mathbf{x}) \sigma_T (\mathbf{x}, \pi)  \widehat{f}(\pi)M) = Tr(\pi(\mathbf{x}) \sigma_T (\mathbf{x}, \pi) \widehat{f}(\pi)).$$
\end{rem}

To conclude with this subsection, we need to define formally the concept of global hypoellipticity we will be using in the following sections. As we anticipated in the introduction, there is actually two kinds of "smooth" functions on $\mathbb{G}$, but here we are interested only in \emph{Schwartz functions} and the concept of global hypoellipticity attached to their definition. This property will be a great example of a wide range of features of operators which are encoded in the behavior of their symbol.

\begin{defi}\label{defihypo}\normalfont
Let $\mathbb{G}$ be a compact $p$-adic Lie group, and let $\mathrm{Rep}(\mathbb{G})$ be the collection of all unitary continuous finite-dimensional representations of $\mathbb{G}$.
\begin{itemize}
    \item The \emph{Schwartz} space $\mathcal{S}(\mathbb{G})$ is defined as the collection of all $L^2$-functions such that $$\| \widehat{f}(\xi) \|_{HS} \lesssim  \langle \xi \rangle_{\mathbb{G}}^{-k},\,\,\, \text{for all} \, k \in \N_0,$$where $\widehat{f}(\xi) := \int_{\mathbb{G}} f(x) \xi^* (x) dx,$ and $\langle \xi \rangle_{\mathbb{G}}^k$ denotes the eigenvalue of the Vladimirov-Taibleson operator $\mathbb{D}^k$ defined in Definition \ref{defiVToperator}, associated to the class $[\xi] \in \widehat{\mathbb{G}}$.
    \item Let $T_\sigma : \mathcal{D} (\mathbb{G}) \subset D(T) \to \mathcal{D}'(\mathbb{G})$ be a densely defined linear operator. We say that $T_\sigma$ is \emph{globally hypoelliptic} if the condition $T_\sigma f = g$ with $f \in \mathcal{D}'(\mathbb{G})$ and $g \in \mathcal{S} (\mathbb{G})$ implies that $f \in \mathcal{S} (\mathbb{G})$.
\end{itemize}
\end{defi}
\begin{rem}
 It is well known, see for instance \cite{Kirilov2020}, how the global hypoellipticity of an invariant operator is equivalent to the condition $$\langle \pi \rangle_{\mathbb{G}}^m \lesssim \| \sigma (\pi) \|_{inf}, \, \, \, m \in \R, \, \,  \quad \text{for all} \, \, [\pi] \in \widehat{\mathbb{G}}.$$Here we are using the notation  $$\| \sigma (\pi)  \|_{inf} : = \inf \{ \|\sigma (\pi) v \|_{\mathcal{H}_\pi}  \, :  \, \| v\|_{\mathcal{H}_\pi} = 1 \}. $$Here $\| v\|_{\mathcal{H}_\pi}$ denotes the Hilbert space norm a vector $v$ in the representation space. 
\end{rem}

\subsection{VT operators}
One important idea from the theory of differential and pseudo-differential operators on Lie groups, is the correspondence between directional derivatives and elements of the Lie algebra. However, in the $p$-adic case, there are plenty of non-trivial locally constant functions, due to the fact that $p$-adic numbers are totally disconnected. This means that the usual notion of derivative does not apply, and therefore we need to find an alternative kind of operators to talk about differentiability on these groups. A first approach to this problem can be the \emph{Vladimitov-Taibleson operator} \cite{Dragovich2023, Dragovich2017}, which we define for general compact $\K$-Lie groups as follows:    

\begin{defi}\label{defiVToperator}\normalfont
Let $\K$ be a non-archimedean local field with ring of integers $\mathscr{O}_\K$, prime ideal $\mathfrak{p}=\textbf{p} \mathscr{O}_\K$ and residue field $\mathbb{F}_q \cong \mathscr{O}_\K/\textbf{p} \mathscr{O}_\K$. Let  $\mathbb{G} \leq \mathrm{GL}_m (\mathscr{O}_\K)$ be a compact $d$-dimensional $\K$-Lie group. We define the \emph{Vladimirov–Taibleson operator} on $\mathbb{G}$  via the formula \[
D^\alpha f(\mathbf{x}) := \frac{1 - q^\alpha}{1 - q^{- (\alpha + d)}} \int_{\mathbb{G}} \frac{f (\mathbf{x} \star \mathbf{y}^{-1}) - f(\mathbf{x})}{\|\mathbf{y} \|_\K^{ \alpha + d}} d\mathbf{y},
\]where \[\| \mathbf{y} \|_\K := \begin{cases}
    1, \, & \, \, \text{if} \, \, \mathbf{y} \in \mathbb{G} \setminus \mathrm{GL}_m (\textbf{p}\mathscr{O}_\K), \\ q^{-n}, \, & \, \, \text{if} \, \, \mathbf{y} \in \mathrm{GL}_m (\textbf{p}^n\mathscr{O}_\K) \setminus \mathrm{GL}_m (\textbf{p}^{n+1}\mathscr{O}_\K). 
\end{cases}\]Here $d\mathbf{y}$ denotes the unique normalized Haar measure on $\mathbb{G}$. Sometimes it will be convenient to consider the operator \[
\mathbb{D}^\alpha f(\mathbf{x}) :=\frac{1-q^{-d}}{1-q^{-(\alpha +d)}}f(\mathbf{x}) + \frac{1 - q^\alpha}{1 - q^{- (\alpha + d)}} \int_{\mathbb{G}} \frac{f (\mathbf{x} \star \mathbf{y}^{-1}) - f(\mathbf{x})}{\|\mathbf{y} \|_\K^{ \alpha + d}} d\mathbf{y}.
\]
\end{defi}
As for any profinite group, the unitary irreducible representations of a compact $p$-adic Lie group $\mathbb{G}$ must be smooth functions, and therefore they all have a non-trivial kernel which contains a certain subgroup $\mathbb{G}(p^l \Z_p)$, where we can assume that $l=l(\xi)$ to be the minimum possible. Thus, as a left-inavriant operator, the VT-operator $D^\alpha$ has the associated symbol: 

\begin{align*}
    \sigma_{D^\alpha}(\xi) &= \xi^* (\mathbf{x}) D^\alpha \xi(x) = \frac{1 - p^\alpha}{1 - p^{- (\alpha + d)}} \int_{\mathbb{G}} \frac{ \xi ( \mathbf{y}^{-1}) - I_{d_\xi}}{\|\mathbf{y} \|_\K^{ \alpha + d}} d\mathbf{y} \\ &= \frac{1 - p^\alpha}{1 - p^{- (\alpha + d)}} \int_{\mathbb{G} \setminus \mathbb{G}(p^l \Z_p) } \frac{ \xi ( \mathbf{y}^{-1}) - I_{d_\xi}}{\|\mathbf{y} \|_\K^{ \alpha + d}} d\mathbf{y} \\  &= \frac{1 - p^\alpha}{1 - p^{- (\alpha + d)}} \Big( -p^{(l-1)(\alpha + d)} p^{-ld} I_{d_\xi}-  \int_{\mathbb{G} \setminus \mathbb{G}(p^l \Z_p) } \frac{  I_{d_\xi}}{\|\mathbf{y} \|_\K^{ \alpha + d}} d\mathbf{y} \Big) \\ &= \frac{1 - p^\alpha}{1 - p^{- (\alpha + d)}} \Big( -p^{l\alpha} p^{-(\alpha + d)} I_{d_\xi}-  \sum_{k=0}^{l-1} p^{k (\alpha + d)} p^{-k d} (1-p^{-d})I_{d_\xi} \Big) \\ &= \frac{1 - p^\alpha}{1 - p^{- (\alpha + d)}} p^{l\alpha} \Big( - p^{-(\alpha + d)} -  p^{-\alpha}(1-p^{-d})\sum_{k=0}^{l-1} p^{-k \alpha }   \Big)I_{d_\xi} \\ &= \frac{1 - p^\alpha}{1 - p^{- (\alpha + d)}} p^{l\alpha} \Big( - p^{-(\alpha + d)} -  p^{-\alpha}(1-p^{-d})\frac{1 - p^{-l\alpha}}{1-p^{-\alpha}}   \Big)I_{d_\xi} \\ &= \frac{1 - p^\alpha}{1 - p^{- (\alpha + d)}} p^{l\alpha} \Big( - p^{-(\alpha + d)} + (1-p^{-d})\frac{1 - p^{-l\alpha}}{1-p^{\alpha}}   \Big)I_{d_\xi} \\ &= \Big( p^{l \alpha} -  \frac{1 - p^{-d}}{1 - p^{- (\alpha + d)}}\Big)I_{d_\xi} .
\end{align*}
In particular, using the fact we can identify the unitary dual of our groups with a certain subset of $\widehat{\Z}_p^d \cong \Q_p^d / \Z_p^d$, we get $$\sigma_{D^\alpha}(\xi) =\Big( \| \xi \|_p^\alpha -  \frac{1 - p^{-d}}{1 - p^{- (\alpha + d)}}\Big)I_{d_\xi}.$$

The Vladimirov–Taibleson operator can be considered as a fractional Laplacian for functions on totally disconnected spaces, and it provides a first notion of differentiability. However, for functions of several variables it is natural to consider the differentiability of the function in each variable, or in a certain given direction. For that reason, we introduce the following definition.

\begin{defi}\normalfont
Let $\K$ be a non-archimedean local field with ring of integers $\mathscr{O}_\K$, prime ideal $\mathfrak{p}= \textbf{p} \mathscr{O}_\K$, and residue field $\mathbb{F}_q = \mathscr{O}_\K/\textbf{p} \mathscr{O}_\K.$ Let $\mathfrak{g} = \spn_{\mathscr{O}_\K} \{ X_1,..,X_d \}$ be a nilpotent $\mathscr{O}_\K$-Lie algebra, and let $\mathbb{G}$ be the exponential image of $\mathfrak{g}$, so that $\mathbb{G}$ is a compact nilpotent $\K$-Lie group. We will use the symbol symbol $\partial_{X}^\alpha$ to denote the \emph{directional Vladimirov–Taibleson operator in the direction of $X \in \mathfrak{g}$}, or directional VT operator for short, which we define as $$\partial_{X}^\alpha f(\mathbf{x}) := \frac{1 - q^\alpha}{1-q^{-(\alpha +1)}} \int_{\mathscr{O}_\K} \frac{f(\mathbf{x} \star \textbf{exp}(tX)^{-1}) - f(\mathbf{x}) dt}{|t|_\K^{\alpha +1}}, \, \, \, \, f \in \mathcal{D}(\mathbb{G}).$$Here $\mathcal{D}(\mathbb{G})$ denotes the space of smooth functions on $\mathbb{G}$, i.e., the collection of locally constant functions with a fixed index of local constancy.   
\end{defi}
 Directional VT operators are interesting because they associate a certain pseudo-differential operator to each element of the Lie algebra. Nonetheless, it is important to remark how this association does not follow the same patter as in the locally connected case, where the correspondence between vectors and operators preserves the Lie algebra structure. However, these operator bear some resemblance to partial derivatives, as we can easily see for the abelian groups $\mathbb{G} = \Z_p^d$. In such context, any analytic vector field $V: \Z_p^d \to \Z_p^d$ has an associated symbol $\partial_V^\alpha$:
\begin{align*}
    \sigma_{\partial_V^\alpha}(x, \xi) & = e^{-2 \pi i \{ \xi \cdot x \}_p} \partial_V^\alpha (e^{2 \pi i \{ \xi \cdot x \}_p}) \\ &= e^{-2 \pi i \{ \xi \cdot x \}_p} \int_{\Z_p} \frac{e^{2 \pi i \{ \xi \cdot  (x - t V(x)) \}_p} - e^{2 \pi i \{ \xi \cdot  x  \}_p}}{|t|_p^{\alpha + 1}} dt \\ &=\int_{\Z_p} \frac{e^{2 \pi i \{ - t (\xi \cdot  V(x)) \}_p} -1}{|t|_p^{\alpha + 1}} dt = | V(x)\cdot \xi|_p^\alpha - \frac{1 - p^{-1}}{1-p^{- (\alpha + 1)}}.
\end{align*}In particular, if we define $\partial_{x_i}^\alpha := \partial_{\mathbf{e}_i}^\alpha $, where $\mathbf{e}_i$, $1 \leq i \leq d$ are the canonical vectors of $\Q_p^d$, then \[\sigma_{\partial_{x_i}^\alpha}(\xi) = \begin{cases}
    0, \, & \, \, \text{if} \, \, |\xi_i|_p=1,\\| \xi_i|_p^\alpha - \frac{1 - p^{-1}}{1 - p^{- (\alpha + 1)}}  & \, \, \text{if} \, \, |\xi_i|_p>1,
\end{cases}
 \]which resembles the symbol of the usual partial derivatives on $\R^d$, justifying that way our choice of notation. However, we want to be emphatic about the fact that these directional VT operators are not derivatives, but rather some special kind of pseudo-differential operators which we will study with the help of the Fourier analysis on compact groups.

\section{The group $\mathbb{G}^{5,2}$}
Let $p>2$ be a prime number. We will start our work by studying the group $\mathbb{G}^{5,2}(\Z_p)$, or just $\mathbb{G}^{5,2}$ for simplicity, defined here as $\Z_p^{5}$ together with the non-commutative operation$$\mathbf{x} \star \textbf{y} := (x_1 + y_1, x_2 + y_2, x_3 + y_3, x_4 + y_4 +x_1 y_2 , x_5 + y_5 + x_1 y_3 ),$$where the inverse element "$\mathbf{x}^{-1}$" of $\mathbf{x}$ in the above operation is given by $$\mathbf{x}^{-1} = (-x_1 , -x_2,-x_3, -x_4 + x_1x_2,-x_5+x_1x_3).$$ We can identify this group with the exponential image of the $\Z_p$-Lie algebra $\mathfrak{g}^{5,2}$ defined by the commutation relations$$[X_1,X_2] = X_4 , \quad [X_1 , X_3]= [X_5].$$

In the following theorem, the full representation theory of this group is given, together with the essentials on the arguments we will be using along this paper. The idea is simple: while in the real case one only has a single kind on relevant representations, corresponding to the generic orbits of the co-adjoint action, here we need to consider all kind of representations. Indeed, we have those representations which in the real case would correspond to the generic orbits, but here the full dual is going to contain also the representations of $\mathbb{G}/ K$, for any closed subgroup $K$, in particular $K = \mathcal{Z}(\mathbb{G})$, where $\mathcal{Z}(\mathbb{G})$ denotes the center of $\mathbb{G}$. Since any nilpotent group can be considered as a sequence of consecutive central extensions, we will always have two kinds of representations: the "old ones" corresponding to the previous central extensions and the "new" representations, corresponding to the particular central extension that we are considering. Let us illustrate this with the proof of our results on $\mathbb{G}^{5,2}$. 
\begin{rem}
    In this paper, we will identify each equivalence class $\lambda$ in $\widehat{\Z}_p \cong \Q_p / \Z_p$, with its associated representative in the complete system of representatives $$\{1\} \cup \big\{ \sum_{k =1}^\infty \lambda_k p^{-k} \, : \, \, \text{only finitely many $\lambda_k$ are non-zero.} \big\}.$$Similarly, every time we consider an element of the quotients $\Q_p / p^{-n} \Z_p$ it will be chosen from the complete system of representatives   $$\{1\} \cup \big\{ \sum_{k =n+1}^\infty \lambda_k p^{-k} \, : \, \text{only finitely many $\lambda_k$ are non-zero} \big\}.$$Also, given any $p$-adic number $u$, we will denote by $\vartheta(u)$ the $p$-adic valuation of $u \in \Q_p$, and for $u \in \Q_p^d$ we write $$\vartheta (u) := \min_{1 \leq j \leq d} \vartheta(u_j).$$
\end{rem}
\begin{teo}\label{teoRepG22}
The unitary dual $\widehat{\mathbb{G}}^{5,2}$ of $\mathbb{G}^{5,2}$ can be identified with following subset of $\widehat{\Z}_p^{5}$:
$$\widehat{\mathbb{G}}^{5,2} = \{ \xi \in \widehat{\Z}_p^5 \, : |\xi_4|_p>|\xi_5|_p, \, (\xi_1,\xi_2) \in \Q_p^2 / p^{\vartheta(\xi_4)}\Z_p, \, \, \, \text{or}, \,\,\,\,  |\xi_4|_p\leq |\xi_5|_p ,\, (\xi_1,\xi_3) \in \Q_p^2 / p^{\vartheta(\xi_5)}\Z_p \},$$and we can write $\widehat{\mathbb{G}}^{5,2} = A_1 \cup A_2 \cup A_3$, where
\begin{itemize}
    \item $A_1 := \{ \xi\in \widehat{\Z}_p^5 \, : \, \|(\xi_4 , \xi_5) \|_p = 1\},$
    \item $A_2:= \{\xi\in \widehat{\Z}_p^5 \, : \, \|(\xi_4 , \xi_5) \|_p > 1, \, \, |\xi_4|_p > |\xi_5|_p, \, \, (\xi_1 , \xi_2) \in \Q_p^2 / p^{\vartheta(\xi_4 )} \Z_p^2 \},$
    \item $A_3:= \{\xi\in \widehat{\Z}_p^5 \, : \, \|(\xi_4 , \xi_5) \|_p > 1, \, \, |\xi_4|_p \leq  |\xi_5|_p, \, \, (\xi_1 , \xi_3) \in \Q_p^2 / p^{\vartheta(\xi_5)} \Z_p^2 \}.$
\end{itemize}Moreover, each unitary irreducible representation can be realized in the finite dimensional Hilbert space  $$\mathcal{H}_\xi := \mathrm{span}_\C \{\|(\xi_4 , \xi_5)\|_p^{1/2}\1_{h + p^{\vartheta(\xi_4 , \xi_5)}\Z_p} \, : \, h \in \Z_p / p^{-\vartheta(\xi_4 , \xi_5)}\Z_p\} , \, \, d_\xi:= dim_\C(\mathcal{H}_\xi) =\|(\xi_4 , \xi_5) \|_p,$$acting according to the formula $$\pi_\xi(\mathbf{x}) \varphi (u) :=e^{2 \pi i \{ \xi \cdot \mathbf{x} + (\xi_4 x_2  + \xi_5 x_3)u \}_p} \varphi (u + x_1).$$With this realization and the natural choice of basis for $\mathcal{H}_\xi$, the associated matrix coefficients are going to be \begin{align*}
    (\pi_\xi)_{hh'} (\mathbf{x}) = e^{2 \pi i \{ \xi \cdot \mathbf{x} + (\xi_4 x_2  + \xi_5 x_3)(h') \}_p } \1_{h' - h + p^{-\vartheta(\xi_4 , \xi_5)}\Z_p } (x_1) , \end{align*}and the associated characters are $$\chi_{\pi_\xi} (\mathbf{x}) = \|(\xi_4, \xi_5) \|_p  e^{2 \pi i \{ \xi \cdot \mathbf{x} \}_p } \1_{ p^{-\vartheta(\xi_4 , \xi_5)}\Z_p } (x_1)\1_{ \Z_p } (\xi_4 x_2  + \xi_5 x_3).$$
    Sometimes we will use the notation $$\mathcal{V}_\xi := \mathrm{span}_\C \{ (\pi_\xi)_{hh'} (\mathbf{x}) \, : \, h,h' \in \Z_p / p^{-\vartheta(\xi_4 , \xi_5)}\Z_p \} \subset L^2 (\mathbb{G}^{5,2}).$$
\end{teo}

\begin{proof}
Notice how the center of $\mathbb{G}^{5,2}$, here denoted by $\mathcal{Z}(\mathbb{G}^{5,2})$, is isomorphic to $\Z_p^2$. If a finite-dimensional continuous representation $\pi$ of $\mathbb{G}^{5,2}$ is trivial on $\mathcal{Z}(\mathbb{G}^{5,2})$, then it must descend to a representation of the abelian group $\mathbb{G}^{5,2}/\mathcal{Z}(\mathbb{G}^{5,2})\cong \Z_p^3$. In this way, we obtain the one dimensional representations $$\pi_\xi(x) = \chi_{(\xi_1 , \xi_2 , \xi_3 , 1 , 1)} (\mathbf{x}):= e^{2 \pi i \{ \xi \cdot \mathbf{x} \}_p}, \, \, \, (\xi_1, \xi_2 , \xi_3) \in \widehat{\Z}_p^3,$$and these can be indexed by the set $$A_1 :=\{\xi \in \widehat{\Z}_p^5 \, : \, \| (\xi_4 , \xi_5) \|_p = 1 \}.$$   
On the other hand, if a representation is not trivial on the center, i.e. if $\| (\xi_4 , \xi_5)\|_p>1$, we need to consider the non-commutative representations $$\pi_\xi(\mathbf{x}) \varphi (u) :=e^{2 \pi i \{ \xi \cdot \mathbf{x} + (\xi_4 x_2  + \xi_5 x_3)u \}_p} \varphi (u + x_1),$$ realized in the finite-dimensional sub-spaces of $L^2 (\Z_p)$ $$\mathcal{H}_\xi := \mathrm{span}_\C \{\1_{h + p^{-\vartheta(\xi_4 , \xi_5)}\Z_p}(u) \, : \, h \in \Z_p / p^{-\vartheta(\xi_4 , \xi_5)}\Z_p\}, \, \, d_\xi :=\mathrm{dim}_\C (\mathcal{H}_\xi) = \|(\xi_4 , \xi_5) \|_p.$$
Notice how: 

\begin{itemize}
    \item The representation space $\mathcal{H}_\xi$ is invariant under the action of $\pi_\xi$, since \begin{align*}
        \pi_\xi(\mathbf{x}) \varphi (u + t) &=e^{2 \pi i \{ \xi \cdot \mathbf{x} + (\xi_4 x_2  + \xi_5 x_3)u+ (\xi_4 x_2  + \xi_5 x_3)t  \}_p} \varphi (u + x_1 + t) \\ &=\pi_\xi(\mathbf{x}) \varphi (u ), 
    \end{align*}for any $t \in \Z_p$ with $|t|_p \leq \|(\xi_4 , \xi_5) \|_p^{-1}.$\
    \item The linear operator $\pi_\xi$ is indeed a representation of $\mathbb{G}^{5,2}$, since \begin{align*}
        \quad \quad \quad \pi_\xi (\mathbf{x}) \pi_\xi (\textbf{y}) \varphi (u) & = \pi_\xi (\mathbf{x})(e^{2 \pi i \{ (\xi_4 y_2 + \xi_5 y_3 )u \}_p} \varphi(u + y_1)) \\ &=e^{2 \pi i \{ (\xi_4 x_2 + \xi_5 x_3 )u \}_p} e^{2 \pi i \{ (\xi_4 y_2 + \xi_5 y_3 )(u + x_1) \}_p} \varphi(u + y_1 + x_1) \\ &=e^{2 \pi i \{ \xi_1(x_1 + y_1) + \xi_2(x_2 + y_2)) + \xi_3(x_3 + y_3) + \xi_4(x_4+ y_4 + y_1 x_2) + \xi_5(x_5 + y_5 + y_1 x_3) + (\xi_4 (x_2 + y_2)  + \xi_5 (x_3+y_3))u \}_p}\\& \quad \quad \times \varphi (u + y_1 +x_1 ) \\ &= \pi_\xi (\mathbf{x} \star \textbf{y})\varphi(u).
    \end{align*}
    Clearly, $\pi_\xi$ also defines an unitary operator since $$\pi^*_\xi (\mathbf{x}) \varphi (u) = e^{2 \pi i \{ - \xi \cdot \mathbf{x} - (\xi_4 x_2  + \xi_5 x_3)(u - x_1) \}_p} \varphi (u - x_1) = \pi_\xi^{-1} (\mathbf{x}) \varphi (u) =\pi_\xi (\mathbf{x}^{-1}) \varphi (u)  .$$
\end{itemize}Let us check these are indeed all the desired representations. To this end we write $$\varphi_h (u):= \| (\xi_4 , \xi_5)\|_p^{1/2}\1_{h + p^{-\vartheta(\xi_4 , \xi_5)}\Z_p}(u), \quad \, \, h \in \Z_p / p^{-\vartheta(\xi_4 , \xi_5)}\Z_p.$$
    This collection of functions defines an orthonormal basis of $\mathcal{H}_\xi$, and the associated matrix coefficients are:
    \begin{align*}
    (\pi_\xi)_{hh'} (\mathbf{x}) &:= (\pi_\xi \varphi_h , \varphi_{h'})_{L^2(\Z_p)} \\ &=\| (\xi_4 , \xi_5) \|_pe^{2 \pi i \{ \xi \cdot \mathbf{x} \}_p }  \int_{\Z_p} e^{2 \pi i \{ (\xi_4 x_2  + \xi_5 x_3)u  \}_p}\varphi_h (u + x_1) \overline{\varphi}_{h'} (u)du \\&= \| (\xi_4 , \xi_5) \|_pe^{2 \pi i \{ \xi \cdot \mathbf{x} \}_p } \1_{h' - h + p^{-\vartheta(\xi_4 , \xi_5)}\Z_p } (x_1) \int_{h' + p^{-\vartheta(\xi_4 , \xi_5)}\Z_p} e^{2 \pi i \{ (\xi_4 x_2  + \xi_5 x_3)u \}_p}du \\&= e^{2 \pi i \{ \xi \cdot \mathbf{x} + (\xi_4 x_2  + \xi_5 x_3)(h') \}_p}\1_{h' - h + p^{-\vartheta(\xi_4 , \xi_5)}\Z_p } (x_1)  .
    \end{align*}
    With this we can compute \begin{align*}
        \chi_{\pi_\xi} (\mathbf{x}) & = \sum_{h \in \Z_p / p^{-\vartheta(\xi_4 , \xi_5)}\Z_p }e^{2 \pi i \{ \xi \cdot \mathbf{x} \}_p } \1_{ p^{-\vartheta(\xi_4 , \xi_5)}\Z_p } (x_1) e^{2 \pi i \{ (\xi_4 x_2  + \xi_5 x_3)h \}_p} \\ &= \|(\xi_4, \xi_5) \|_p  e^{2 \pi i \{ \xi \cdot \mathbf{x} \}_p } \1_{ p^{-\vartheta(\xi_4 , \xi_5)}\Z_p } (x_1)\1_{ \Z_p } (\xi_4 x_2  + \xi_5 x_3),
    \end{align*}and from the above it is clear that \begin{align*}
        \int_{\mathbb{G}^{5,2}} |\chi_{\pi_\xi} (\mathbf{x})|^2 d\mathbf{x} &= \int_{\Z_p^5} \|(\xi_4 , \xi_5) \|_p^2 \1_{ p^{-\vartheta(\xi_4 , \xi_5)}\Z_p } (x_1)\1_{ \Z_p } (\xi_4 x_2  + \xi_5 x_3) dx \\ & = \|(\xi_4 , \xi_5) \|_p \int_{\Z_p^2} \1_{ \Z_p } (\xi_4 x_2  + \xi_5 x_3) dx_2 dx_3.
    \end{align*}
    Now we consider cases. First, if $|\xi_4|_p > |\xi_5|_p,$ we get that $$\1_{ \Z_p } (\xi_4 x_2  + \xi_5 x_3) = \1_{ p^{-\vartheta(\xi_4 )}\Z_p } (x_2  + \frac{\xi_5}{\xi_4} x_3) = \1_{ \frac{\xi_5}{\xi_4} x_3 + p^{-\vartheta(\xi_4 )}\Z_p } (x_2  ), $$so that\begin{align*}
        \int_{\Z_p} \int_{\Z_p} \1_{ \Z_p } (\xi_4 x_2  + \xi_5 x_3) dx_2 dx_3 & =\int_{\Z_p} \int_{\Z_p} \1_{ \frac{\xi_5}{\xi_4} x_3 + p^{-\vartheta(\xi_4 )}\Z_p } (x_2  ) dx_2 dx_3  \\ &= \int_{\Z_p} |\xi_4|_p^{-1} dx_3 = |\xi_4|_p^{-1}.
    \end{align*}
    Similarly, if $|\xi_5|_p \geq |\xi_4|$ then $$\1_{ \Z_p } (\xi_4 x_2  + \xi_5 x_3) = \1_{ p^{-\vartheta(\xi_5)}\Z_p } (x_3  + \frac{\xi_4}{\xi_5} x_2) = \1_{ \frac{\xi_4}{\xi_5} x_2 + p^{-\vartheta(\xi_5 )}\Z_p } (x_2 ), $$so that\begin{align*}
        \int_{\Z_p} \int_{\Z_p} \1_{ \Z_p } (\xi_4 x_2  + \xi_5 x_3) dx_2 dx_3 & =\int_{\Z_p} \int_{\Z_p} \1_{ \frac{\xi_4}{\xi_5} x_2 + p^{-\vartheta(\xi_5)}\Z_p } (x_2 )dx_2 dx_3  \\ &= \int_{\Z_p} |\xi_5|_p^{-1} dx_2 = |\xi_5|_p^{-1}.
    \end{align*}
    In any case, $\int_{\mathbb{G}^{5,2}} |\chi_{\pi_\xi} (\mathbf{x})|^2 d\mathbf{x} =1$, proving how each $\pi_\xi$ defines an unitary irreducible representation. Next, we need to use the obtained characters to count the number of different equivalence classes of representations, by counting the number of different characters obtained. Once again we need to consider two cases. First, if $|\xi_4|_p > |\xi_5|_p,$ then the condition $\xi_4 x_2  + \xi_5 x_3 = a \in  \Z_p$ implies that $$ x_2 = \frac{a}{\xi_4} - \frac{\xi_5}{\xi_4} x_3, \, \implies \, \, \xi_2 x_2 = \frac{\xi_2}{\xi_4}a - \frac{\xi_5 \xi_2}{\xi_4} x_3.$$This shows how $$e^{2 \pi i \{\xi_1 x_1 + \xi_2 x_2 +\xi_3 x_3\}_p} = e^{2 \pi i \{\xi_1 x_1 + (\xi_3 +\frac{\xi_5 \xi_2}{\xi_4}) x_3\}_p}, \quad \text{for} \, \, |\xi_2|_p \leq |\xi_4|_p,$$so that we only need to consider $\xi_2 \in \Q_p/ p^{\vartheta(\xi_4 )}\Z_p$, and thus, these representations are indexed by the set $$A_2 := \{\xi\in \widehat{\Z}_p^5 \, : \, \|(\xi_4 , \xi_5) \|_p > 1, \, \, |\xi_4|_p > |\xi_5|_p, \, \, (\xi_1 , \xi_2) \in \Q_p^2 / p^{\vartheta(\xi_4 )} \Z_p^2 \}.$$ 
In a very similar way, if $|\xi_4|_p \leq |\xi_5|_p,$ then the condition $\xi_4 x_2  + \xi_5 x_3 = a \in  \Z_p$ implies that $$ x_3 = \frac{a}{\xi_5} - \frac{\xi_4}{\xi_5} x_2, \, \implies \, \, \xi_3 x_3 = \frac{\xi_3}{\xi_5}a - \frac{\xi_4 \xi_2}{\xi_5} x_2,$$showing how $$e^{2 \pi i \{\xi_1 x_1 + \xi_2 x_2 +\xi_3 x_3\}_p} = e^{2 \pi i \{\xi_1 x_1 + (\xi_2 -\frac{\xi_4 \xi_2}{\xi_5} )x_2\}_p}, \quad \text{for} \, \, |\xi_2|_p \leq |\xi_4|_p,$$so that we only need to count $\xi_3 \in \Q_p/ |\xi_5|_p^{1}\Z_p$, and thus these representations are indexed by the set $$A_3 := \{\xi\in \widehat{\Z}_p^5 \, : \, \|(\xi_4 , \xi_5) \|_p > 1, \, \, |\xi_5|_p \geq |\xi_4|_p, \, \, (\xi_1 , \xi_3) \in \Q_p^2 / p^{\vartheta(\xi_5 )} \Z_p^2 \}.$$ 
Finally, in order to prove these are indeed all the desired representations, let us define
\begin{itemize}
    \item $A_1 (n) := \{\| \xi \|_p \leq p^n\, : \, \|(\xi_4 , \xi_5) \|_p = 1 \},$
    \item $A_2 (n):= \{\| \xi \|_p \leq p^n\, : \, \|(\xi_4 , \xi_5) \|_p > 1, \, \, |\xi_4|_p > |\xi_5|_p, \, \, (\xi_1 , \xi_2) \in \Q_p^2 / p^{\vartheta(\xi_4)} \Z_p^2 \}$,
    \item $A_3(n) := \{\| \xi \|_p \leq p^n \, : \, \|(\xi_4 , \xi_5) \|_p > 1, \, \, |\xi_5|_p \geq |\xi_4|_p, \, \, (\xi_1 , \xi_3) \in \Q_p^2 / p^{\vartheta(\xi_5)} \Z_p^2 \},$ 
    \item $B_{\widehat{\mathbb{G}}^{5,2}} (n)= A_1 (n) \cup A_2(n) \cup A_3(n) \subset p^{-n} \Z_p^5 / \Z_p^5 ,$
\end{itemize}
where each of the sets $A_j (n)$ , $j=1,2,3,$ are subsets of the $p$-adic ball $B_{\widehat{\Z_p^5}} (n)$. With this notation 
\begin{align*}
    \sum_{\xi \in  B_{\widehat{\mathbb{G}}^{5,2}} (n)} d_\xi^2 & = \sum_{\xi \in A_1(n)} 1^2 + \sum_{\xi \in A_2(n)} |\xi_4|_p^2 + \sum_{\xi \in A_3(n)} |\xi_5|_p^2 \\ & = \sum_{\| \xi \|_p \leq p^n, \, \| (\xi_4 , \xi_5)\|_p =1} 1^2 + \sum_{1< \| (\xi_4 , \xi_5)\|_p \leq p^n, \, \, |\xi_5|_p< |\xi_4|_p}  \sum_{(\xi_1,\xi_2) \in \Q_p^2 / p^{\vartheta(\xi_4,)} \Z_p^2} \sum_{1 \leq |\xi_3|_p \leq p^n}   | \xi_4|_p^2  \\& + \sum_{1< \| (\xi_4 , \xi_5)\|_p \leq p^n, \, \, |\xi_4|_p \leq|\xi_5|_p}  \sum_{(\xi_1,\xi_3) \in \Q_p^2 / p^{\vartheta(\xi_5,)} \Z_p^2} \sum_{1 \leq |\xi_2|_p \leq p^n}   | \xi_4|_p^2\\ \, \, \, & + \sum_{\| (\xi_4 , \xi_5)\|_p >1, \,\, 1 \leq |\xi_4|_p \leq |\xi_5|_p }  \sum_{(\xi_1 , \xi_3) \in \Q_p^2 / |\xi_5|_p^{-1} \Z_p} \sum_{1 \leq |\xi_2| \leq p^n}|\xi_5|_p^2  \\ &= p^{3n } +\sum_{1< \| (\xi_4 , \xi_5)\|_p \leq p^n, \, \, |\xi_5|_p< |\xi_4|_p}  (\frac{p^{2n}}{|\xi_4|_p^2}) (p^n) |\xi_4|_p^2 + \sum_{1< \| (\xi_4 , \xi_5)\|_p \leq p^n, \, \, |\xi_4|_p \leq |\xi_5|_p}  (\frac{p^{2n}}{|\xi_5|_p^2}) (p^n)|\xi_5|_p^2 \\ &= p^{3n}+ p^{3n} \sum_{1 <\|(\xi_4 , \xi_5) \|_p \leq p^n} 1 =p^{3n} + p^{3n}(p^{2n} -1 )\\ &= p^{5n} = |\mathbb{G}^{5,2} / \mathbb{G}^{5,2}(p^n\Z_p) |.   
\end{align*}
This concludes the proof.
\end{proof}
With the above theorem, functions $f \in L^2 (\mathbb{G}^{5,2})$ have the Fourier series representation \begin{align*}
    f(\mathbf{x}) &= \sum_{ \| (\xi_4, \xi_5)\|_p =1 } \widehat{f}(\xi) e^{2 \pi i \{ \xi \cdot \mathbf{x}\}_p} \\ &\quad \quad + \sum_{1< \| (\xi_4 , \xi_5)\|_p , \, \, |\xi_5|_p< |\xi_4|_p}  \sum_{(\xi_1,\xi_2) \in \Q_p^2 / p^{\vartheta(\xi_4)} \Z_p^2} \sum_{ \xi_3 \in \widehat{\Z}_p} | \xi_4 |_p Tr[ \pi_\xi (\mathbf{x}) \widehat{f}(\xi)] \\ &\quad \quad + \sum_{1< \| (\xi_4 , \xi_5)\|_p , \, \, |\xi_4|_p \leq |\xi_5|_p}  \sum_{(\xi_1,\xi_3) \in \Q_p^2 / p^{\vartheta(\xi_5)} \Z_p^2} \sum_{ \xi_2 \in \widehat{\Z}_p} | \xi_5 |_p Tr[ \pi_\xi (\mathbf{x}) \widehat{f}(\xi)]  \\ &= \sum_{\xi \in \widehat{\mathbb{G}}^{5,2}} \| (\xi_4 , \xi_5) \|_p Tr[\pi_\xi (\mathbf{x}) \widehat{f}(\xi)],
\end{align*}where the Fourier transform $\widehat{f}(\xi)$ is the linear operator defined as \begin{align*}
    \widehat{f}(\xi) \varphi (u)&:= \int_{\mathbb{G}^{5,2}} f(\mathbf{x}) \pi_\xi^* (\mathbf{x}) \varphi (u) d \mathbf{x} \\ &= \int_{\Z_p^5} f(\mathbf{x}) e^{-2 \pi i \{ \xi \cdot \mathbf{x} + (\xi_4 x_2  + \xi_5 x_3)(u - x_1) \}_p} \varphi (u - x_1) d\mathbf{x}, \quad \varphi \in \mathcal{H}_\xi.
\end{align*}Alternatively, by using our expressions for the matrix coefficients, we can see this operator as the matrix $\widehat{f}(\xi) \in \C^{\|(\xi_4, \xi_5)\|_p \times \|(\xi_4, \xi_5)\|_p}$ defined entry wise by the expression \begin{align*}
    \widehat{f}(\xi)_{h h'}&:= \int_{\mathbb{G}^{5,2}} f(\mathbf{x}) \pi_\xi^* (\mathbf{x})_{hh'} d \mathbf{x} \\ &= \int_{\Z_p^5} f(\mathbf{x}) e^{-2 \pi i \{ \xi \cdot \mathbf{x} +2 \pi i \{ (\xi_4 x_2  + \xi_5 x_3)(h) \}_p } \1_{h - h' + p^{-\vartheta(\xi_4 , \xi_5)}\Z_p } (x_1) d\mathbf{x}, \quad \varphi \in \mathcal{H}_\xi, 
\end{align*}which in terms of the $\Z_p^5$-Fourier transform is expressed as \begin{align*}
    \widehat{f}(\xi)_{h h'} &= \mathcal{F}_{\Z_p^5} [\1_{h - h' + p^{-\vartheta(\xi_4 , \xi_5)}\Z_p } (x_1) f] (\xi_1 ,\xi_2 + h \xi_4, \xi_3 + \xi_5 h, \xi_4 , \xi_5 )\\ &=\mathcal{F}_{\Z_p^5} [\1_{h - h' + p^{-\vartheta(\xi_4 , \xi_5)}\Z_p } (x_1)]*_{\widehat{\Z}_p^5}\mathcal{F}_{\Z_p^5}[  f](\xi_1 ,\xi_2 + h \xi_4, \xi_3 + \xi_5 h, \xi_4 , \xi_5 ) .
\end{align*}
With this Fourier series representation, a linear invariant operator $T$ can be written as a pseudo-differential operator \begin{align*}
    T f(\mathbf{x}) &= \sum_{\xi \in \widehat{\mathbb{G}}^{5,2}} \| (\xi_4 , \xi_5) \|_p Tr[\pi_\xi (\mathbf{x})\sigma_T (\xi) \widehat{f}(\xi)]  ,
\end{align*}where the symbol is defined as the mapping $$\sigma_T (\xi) := \pi_\xi^*(\mathbf{x}) T \pi_\xi (\mathbf{x}) \in \C^{\|(\xi_4 ,\xi_4)\|_p \times \|(\xi_4 ,\xi_4)\|_p}, \quad \xi \in \widehat{\mathbb{G}}^{5,2},$$and we can use this association of symbols to study the particular example of the Vladimirov sub-Laplacian, which we can easily check defines an left invariant operator. We know how $\mathbb{G}^{5,2}$ is a stratified group since $\mathfrak{g}^{5,2} = V_1 \oplus V_2$, where $$V_1 := \mathrm{span}_{\Z_p} \{X_1, X_2, X_2 \}, \quad V_2:= \mathrm{span}_{\Z_p} \{X_3 , X_4 \},$$so that, given any collection $\textbf{W}:=\{W_1 , W_2, W_3$\}  with $$W_i = w_i^1 X_1 + w_i^2 X_2 + w_i^3 X_3,$$which spans $\mathfrak{g}^{5,2}/ [\mathfrak{g}^{5,2},\mathfrak{g}^{5,2}]$, we can define its associated Vladimirov sub-Laplacian as $$\mathcal{L}_{sub, \textbf{W}}^\alpha f (x) : = (\partial^\alpha_{W_1} + \partial^\alpha_{W_2} + \partial^\alpha_{W_3}) f (x), \,\,\, f \in \mathcal{D}(\mathbb{G}^{5,2}).$$For $\alpha>0$ the following theorem proves how Conjecture \ref{conjecture} holds true for this operator on $\mathbb{G}^{5,2}$, and give a particular version of the first par of Theorem \ref{teogeneralresult}. 

\begin{rem}
    In every section of this work we are dedicated to the proof of two results: one about the representation theory of a certain compact group, like Theorems \ref{TeoRepresentationsHd}, \ref{TeoRepresentationsB4} and \ref{teoRepG22}, and a second, which actually is a corollary of the first, about the spectrum of the Vladimirov sub-Laplacian. For this reason, we will often be reusing arguments and the reader will notice a lot of similarities from section to section, because we are essentially proving the same results again and again, each time on a different group. Sill, the author decided to repeat some of our arguments in each section aiming to make possible to read them independently. The author sincerely hopes this is not bothersome to the reader.
\end{rem}

\begin{teo}\label{teosubLapG22}
The Vladimirov sub-Laplacian $\mathcal{L}_{sub, \textbf{W}}^\alpha$ assocated with the collection $\textbf{W}$ is a globally hypoelliptic operator, which is invertible in the space of mean zero functions. Moreover, the space $L^2(\mathbb{G}^{5,2})$ can be written as the direct sum $$L^2(\mathbb{G}^{5,2}) = \overline{\bigoplus_{\xi \in \widehat{\mathbb{G}}^{5,2}} \bigoplus_{h' \in  \Z_p / p^{-\vartheta( \xi_4 , \xi_5)} \Z_p} \mathcal{V}_{\xi}^{h'}}, \, \,\, \mathcal{V}_{\xi} = \bigoplus_{h' \in  \Z_p / p^{-\vartheta( \xi_4 , \xi_5)} \Z_p} \mathcal{V}_{\xi}^{h'}, $$where each finite-dimensional sub-space$$\mathcal{V}_{\xi}^{h'}:= \mathrm{span}_\C \{ (\pi_{\xi})_{hh'} \, : \, h \in \Z_p / p^{-\vartheta( \xi_4 , \xi_5)} \Z_p \},$$is an invariant sub-space of $\mathcal{L}_{sub, \textbf{W}}^\alpha$, and its spectrum restricted to $\mathcal{V}_{\xi }^{h'}$ is given by $$Spec(\mathcal{L}_{sub, \textbf{W}}^\alpha|_{\mathcal{V}_{\xi}^{h'}})= \Big\{ \sum_{k=1}^3 |(\xi_1 , \xi_2 + h' \xi_4, \xi_3 + h' \xi_5)\cdot W_k|_p^\alpha - 3\frac{1 - p^{-1}}{1 - p^{-(\alpha +1)}} \, : \, 1 \leq| \tau |_p \leq \| (\xi_4 , \xi_5)\|_p \Big\},$$
so that $Spec(\mathcal{L}_{sub, \textbf{W}}^\alpha)$ is going to be the collection of real numbers \begin{align*}
     \sum_{k=1}^3 |(\xi_1 , \xi_2 + h' \xi_4, \xi_3 + h' \xi_5) \cdot W_k|_p^\alpha - 3\frac{1 - p^{-1}}{1 - p^{-(\alpha +1)}} ,
\end{align*}where $\xi \in \widehat{\mathbb{G}}^{5,2}, \, h' \in \Z_p / p^{-\vartheta( \xi_4 , \xi_5)} \Z_p , \,  1 \leq | \tau |_p \leq \|(\xi_4 , \xi_5)\|_p$, and the corresponding eigenfunctions are given by  $$\mathscr{e}_{\xi , h' , \tau} (\mathbf{x}) := e^{2 \pi i \{ \xi \cdot \mathbf{\mathbf{x}} + (\xi_4 x_2  + \xi_5 x_3)(h') + \tau x_1 \}_p}, \, \, \, \,\xi \in \widehat{\mathbb{G}}^{5,2}, \, h' \in \Z_p / p^{-\vartheta( \xi_4 , \xi_5)}  \Z_p , \,  1 \leq | \tau |_p \leq \| (\xi_4 , \xi_5) \|_p.$$
\end{teo}
\begin{proof}
To prove our theorem, and this will be a recurrent argument along this paper, we simply need to consider the subspaces $$\mathcal{V}_{\xi}^{h'}:= \mathrm{span}_\C \{ (\pi_{\xi})_{hh'} \, : \, h \in \Z_p / p^{-\vartheta( \xi_4 , \xi_5)} \Z_p \}.$$
An alternative description of this space could be \begin{align*}
    \mathcal{V}_{\xi}^{h'} &= e^{2 \pi i \{ \xi \cdot \mathbf{x} + (\xi_4 x_2  + \xi_5 x_3)(h') \}_p} \mathcal{D}_{-\vartheta(\xi_4 , \xi_5) } (\Z_p) \\ & =\{  e^{2 \pi i \{ \xi \cdot \mathbf{x} + (\xi_4 x_2  + \xi_5 x_3)(h') \}_p} f \, : \, f(x_1 + u ) = f(x_1), \, \, \, \text{for} \, \, \, | u|_p \leq \|(\xi_4 , \xi_5) \|_p^{-1} \},
\end{align*}
and in this way it is not hard to see how the following set defines an orthonormal basis for the space: $$\mathscr{e}_{\xi , h' , \tau} (\mathbf{x}) := e^{2 \pi i \{ \xi \cdot \mathbf{x} + (\xi_4 x_2  + \xi_5 x_3)(h') + \tau x_1 \}_p}, \quad 1 \leq |\tau|_p \leq \|(\xi_4 , \xi_5) \|_p.$$
With this we can check how \begin{align*}
    \partial_{W_i}^\alpha \mathscr{e}_{\xi , h' , \tau} (\mathbf{x}) &= \Big(\int_{\Z_p} \frac{e^{2 \pi i \{ -t(\xi_1 w_i^1 + (\xi_2 + h' \xi_4) w_i^2 + (\xi_3+ h' \xi_5) w_i^3)  \}_p } - 1}{|t|_p^{\alpha + 1}} dt \Big) \mathscr{e}_{\xi , h' , \tau} (\mathbf{x}),
\end{align*}
so that $\partial_{W_i}^\alpha \mathscr{e}_{\xi , h' , \tau} (\mathbf{x}) =$ \[
\begin{cases}
 \Big(| (\xi_1 , \xi_2 + h' \xi_4, \xi_3 + h' \xi_5) \cdot W_i|_p^\alpha - \frac{1 - p^{-1}}{1 - p^{-(\alpha + 1)}}\Big)\mathscr{e}_{\xi , h' , \tau} (\mathbf{x}), & \, \text{if} \,  |(\xi_1 , \xi_2 + h' \xi_4, \xi_3 + h' \xi_5) \cdot W_i|_p>1 , \\ 0, & \, \text{if} \,  (\xi_1 , \xi_2 + h' \xi_4, \xi_3 + h' \xi_5) \cdot W_i \in \Z_p .
\end{cases}
\]Thus we can write $$\mathcal{L}_{sub, \textbf{W}}^\alpha\mathscr{e}_{\xi , h' , \tau} (\mathbf{x}) = \Big(\sum_{i=1}^3| (\xi_1 , \xi_2 + h' \xi_4, \xi_3 + h' \xi_5) \cdot W_i|_p^\alpha - 3\frac{1 - p^{-1}}{1 - p^{-(\alpha + 1)}}\Big)\mathscr{e}_{\xi , h' , \tau} (\mathbf{x}).$$
Finally, in order to prove the global hypoellipticity, let us recall for the reader the arguments in \cite{Kirilov2020}. Given any compact group, if a linear operator $T$ is a translation invariant operator, then it can be written as $$Tf(x) = \sum_{\xi \in \widehat{G}} d_\xi Tr[ \xi(x) \sigma_T (\xi) \widehat{f}(\xi)].$$
Using this symbolic representation we can see how, for the particular case of a compact Lie group $G$, the global hypoellipticity becomes equivalent to two conditions: first, the invertibility of the symbol $\sigma_T(\xi)$, for all but finitely many $\xi \in  \widehat{G}$, and second, the operator norm of the inverse of the symbol should be polynomially bounded. The reason is how smooth functions are characterized via their Fourier coefficients as the functions with rapidly decreasing Fourier transform. In our case the same is true, and we have the direct sum decomposition for each reprsentation space
$$\mathcal{V}_{\xi} = \bigoplus_{h' \in  \Z_p / p^{-\vartheta( \xi_4 , \xi_5)} \Z_p} \mathcal{V}_{\xi}^{h'},$$
so that for the particula case of $\mathcal{L}_{sub, \textbf{W}}^\alpha$, it should be clear that \begin{align*}
    \| \sigma_{\mathcal{L}_{sub, \textbf{W}}^\alpha}(\xi) \|_{op} &= \| \mathcal{L}_{sub, \textbf{W}}^\alpha|_{\mathcal{V}_{\xi}} \|_{op}  = \max_{h'  \in \Z_p / p^{-\vartheta(\xi_4 , \xi_5)} \Z_p }  \| \mathcal{L}_{sub, \textbf{W}}^\alpha \|_{\mathcal{V}_{\xi}^{h'}} \|_{op},
\end{align*}
and $$ \| \sigma_{\mathcal{L}_{sub, \textbf{W}}^\alpha}(\xi)^{-1} \|_{op}^{-1} =  \| \sigma_{\mathcal{L}_{sub, \textbf{W}}^\alpha}(\xi) \|_{inf} = \| \mathcal{L}_{sub, \textbf{W}}^\alpha |_{\mathcal{V}_{\xi}} \|_{inf}  = \min_{h'  \in \Z_p / p^{-\vartheta(\xi_4 , \xi_5)} \Z_p }  \| \mathcal{L}_{sub, \textbf{W}}^\alpha |_{\mathcal{V}_{\xi}^{h'}} \|_{inf},$$so, we can conclude that $$\| \sigma_{\mathcal{L}_{sub, \textbf{W}}^\alpha}(\xi) \|_{op} \leq \| \xi\|_p^\alpha, \, \, \text{and} \, \, \, \| (\xi_1 , \xi_2, \xi_3) \|_p^\alpha \leq \| \sigma_{\mathcal{L}_{sub, \textbf{W}}^\alpha}(\xi) \|_{inf}.$$
This concludes the proof. 
\end{proof}

\section{The group $\mathbb{G}^{5,3}$}
Let $p>3$ be a prime number. In this section we consider the group $\mathbb{G}^{5,3}(\Z_p)$, or just $\mathbb{G}^{5,3}$ for simplicity, defined here as $\Z_p^{5}$ together with the non-commutative operation$$\mathbf{x} \star \textbf{y} := (x_1 + y_1, x_2 + y_2, x_3 + y_3, x_4 + y_4 +x_1 y_2 , x_5 + y_5 + x_2 y_3 + x_1y_4 + \frac{1}{2} x_1^2 y_2 ),$$and inverse element $$\mathbf{x}^{-1}:= (-x_1, -x_2, -x_3, -x_4 + x_1 x_2, - x_5 + x_1 x_4 + x_2 x_3 - \frac{1}{2} x_1 x_2^2).$$We can identify this group with the exponential image of the $\Z_p$-Lie algebra $\mathfrak{g}^{5,3}$ defined by the commutation relations$$[X_1,X_2] = X_4 , \quad [X_1 , X_4]= X_5, \quad [X_2 , X_3] = X_5.$$
In Theorem \ref{teorepG53}, in a similar way as for the previous groups, we will provide an explicit description of the unitary dual of $\mathbb{G}^{5,3}$. This time, before stating the theorem and proceeding with its proof, we will establish the following auxiliary lemma. 

\begin{lema}\label{lemaauxG53}
Let $(\xi_4 , \xi_5) \in \widehat{\Z}_p^5$. Then $$\int_{\Z_p^2} \Big| \int_{\Z_p }e^{2 \pi i \{(\xi_4 x_2 + \xi_5 x_4) u\}_p}du  \Big|^2 dx_2 dx_4 = \|(\xi_4, \xi_5) \|_p^{-1}.$$
\end{lema}
\begin{proof}
Define the auxiliary function $$f_{\xi_4 , \xi_5} (x_2, x_4) := \int_{\Z_p }e^{2 \pi i \{(\xi_4 x_2 + \xi_5 x_4) u\}_p}du.$$
Then, using its Fourier transform to calculate its $L^2$-norm, we get \begin{align*}
    \mathcal{F}_{\Z_p^2} [f_{\xi_4 , \xi_5}] (\alpha , \beta) &= \int_{\Z_p} \1_{\Z_p} (\xi_4 u - \alpha)  \1_{\Z_p}(\xi_5 u - \beta) du \\ & =\int_{\Z_p} \1_{ \frac{\alpha}{ \xi_4} + p^{-\vartheta(\xi_4)}\Z_p}  ( u)  \1_{\frac{\beta}{\xi_5} + p^{-\vartheta(\xi_5)}\Z_p}( u ) du \\ &= \mu_{\Z_p}\Big( \big[  \frac{\alpha}{ \xi_4} + p^{-\vartheta(\xi_4)}\Z_p \big] \cap \big[ \frac{\beta}{\xi_5} + p^{-\vartheta(\xi_5)}\Z_p \big] \Big),
\end{align*}where $\mu_{\Z_p}$ denotes the normalized Haar measure on $\Z_p$. In this way its is clear how \begin{align*}
\| f_{\xi_4 , \xi_5}\|_{L^2 (\Z_p^2)}^2 &= \sum_{(\alpha, \beta) \in \widehat{\Z}_p^2}      \mu_{\Z_p} \Big(  \big[\frac{\alpha}{ \xi_4} + p^{-\vartheta(\xi_4)}\Z_p \big] \cap \big[ \frac{\beta}{\xi_5} + p^{-\vartheta(\xi_5)}\Z_p \big] \Big)^2 \\ &= \max\{ |\xi_4|_p, |\xi_5|_p \}^{-1} \sum_{(\alpha, \beta) \in \widehat{\Z}_p^2} \mu_{\Z_p} \Big( \big[ \frac{\alpha}{ \xi_4} + p^{-\vartheta(\xi_4)} \Z_p \big] \cap \big[ \frac{\beta}{\xi_5} + p^{-\vartheta(\xi_5)} \Z_p \big] \big) \\ &= \max\{ |\xi_4|_p, |\xi_5|_p \}^{-1} \mu_{\Z_p}(\Z_p) = \| (\xi_4 , \xi_5) \|_p.
\end{align*}
This concludes the proof. 
\end{proof}
\begin{rem}
    There is an alternative proof for the previous lemma, which is observing how $$\int_{\Z_p }e^{2 \pi i \{(\xi_4 x_2 + \xi_5 x_4) u\}_p}du = \1_{\Z_p}(\xi_4 x_2 + \xi_5 x_4),$$but we prefer here the previous proof because it introduce the argument we will use later for more complicated integrals, like $p$-adic Gaussians.  
\end{rem}

Now we are ready to prove the main result of this section. 
\begin{teo}\label{teorepG53}
The unitary dual $\widehat{\mathbb{G}}^{5,3}$ of $\mathbb{G}^{5,3}$ can be identified with the  following subset of $\widehat{\Z}_p^{5}$: 
$$\mathbb{G}^{5,3}=\{\xi \in \widehat{\Z}_p^5 \, : \, \, (\xi_3, \xi_4) \in \Q_p^2/p^{-\vartheta(\xi_5)}\Z_p^2, \, \, (\xi_1, \xi_2) \in \Q_p^2 / p^{-\vartheta(\xi_4 , \xi_5)} \Z_p^2 \}.$$Moreover, each unitary irreducible representation can be realized in the finite dimensional Hilbert space \[\mathcal{H}_\xi := 
    \mathrm{span}_\C \{\varphi_h:= \| (\xi_4, \xi_5)\|_p^{1/2} |\xi_5|_p^{1/2} \1_{h + p^{-\vartheta(\xi_4 , \xi_5)} \Z_p \times p^{\vartheta(\xi_5)}\Z_p}(u) \, : \, h \in \Z_p^2 / p^{\vartheta(\xi_4 , \xi_5)} \Z_p \times p^{-\vartheta(\xi_5)}\Z_p\},\]acting according to the formula \[\pi_\xi(\mathbf{x}) \varphi (u) := e^{2 \pi i \{ \xi \cdot \mathbf{x} + \xi_4 (u_1 x_2)+ \xi_5(-x_4 u_1   + \frac{1}{2} x_2 u_1^2+x_3 u_2) \}_p} \varphi (u + (x_1 , x_2)) ,\]where \[d_\xi := \mathrm{dim}_\C(\mathcal{H}_\xi) =  \|(\xi_4 , \xi_5) \|_p |\xi_5|_p= \begin{cases}
    |\xi_4|_p |\xi_5|_p,   \, \, & \,\, \text{if} \, \, |\xi_4|_p > |\xi_5|_p>1, \\
    |\xi_5|_p^2, \, \, & \,\, \text{if} \, \, |\xi_5|_p>1, \, \, |\xi_4|_p =1.
\end{cases}\]With this realization and the natural choice of basis for $\mathcal{H}_\xi$, the associated matrix coefficients are going to be \[(\pi_\xi)_{hh'} (\mathbf{x}) = 
     e^{2 \pi i \{ \xi \cdot \mathbf{x}+\xi_4 x_2 (h_1')+ \xi_5 x_3(h'_2) + \xi_5(x_4 h_1'   + \frac{1}{2} x_2 (h_1')^2) \}_p } \1_{h' - h + p^{-\vartheta(\xi_4, \xi_5)} \Z_p \times p^{-\vartheta(\xi_5)}\Z_p } (x_1,x_2),\]and the associated characters are $\chi_{\pi_\xi} (\mathbf{x})= $ \[\begin{cases}
    |\xi_4|_p e^{2 \pi i \{ \xi \cdot \mathbf{x} \}_p } \1_{ p^{-\vartheta(\xi_4)}  \Z_p^2 } (x_1,x_2) ,   \, \, & \,\, \text{if} \, \,  |\xi_5|_p=1,\\
     |\xi_5|_p^2 e^{2 \pi i \{ \xi \cdot \mathbf{x} \}_p } \1_{ p^{-\vartheta(\xi_5)} \Z_p^4 } (x_1,x_2,x_3,x_4), \, \, & \,\, \text{if} \, \, |\xi_5|_p>1, \, \, |\xi_4|_p = 1,\\ |\xi_4|_p |\xi_5|_p  e^{2 \pi i \{\xi  \cdot \mathbf{x} \}_p}  \1_{p^{-\vartheta(\xi_4)} \Z_p} (x_1) \1_{p^{-\vartheta(\xi_5)} \Z_p} (x_2)\1_{p^{-\vartheta(\xi_5)} \Z_p}(x_3) \1_{\Z_p} ( \xi_4  x_2+ \xi_5 x_4)\,\, \, & \,\, \text{if} \, \, 1<|\xi_5|_p<|\xi_4|_p.
\end{cases}\]Sometimes we will use the notation $$\mathcal{V}_\xi := \mathrm{span}_\C \{ (\pi_\xi)_{hh'} (\mathbf{x}) \, : \, h,h' \in I_\xi \},$$where \[I_\xi :=\Z_p^2 / p^{-\vartheta(\xi_4, \xi_5)} \Z_p \times p^{-3\vartheta(\xi_5)}\Z_p= \begin{cases}
    \Z_p / p^{-\vartheta(\xi_4)}\Z_p,   \, \, & \,\, \text{if} \, \, |\xi_5|_p=1, \\
    \Z_p^2/p^{-\vartheta(\xi_4, \xi_5)} \Z_p \times p^{-\vartheta(\xi_5)}\Z_p, \, \, & \,\, \text{if} \, \, |\xi_5|_p>1.
\end{cases}\]
\end{teo}

\begin{proof}
Let us denote by $\mathcal{Z}(\mathbb{G}^{5,3})$ the center of $\mathbb{G}^{5,3}$, which in this case is isomorphic to $\Z_p$. If a finite-dimensional continuous representation $\pi$ of $\mathbb{G}^{5,3}$ is trivial on $\mathcal{Z}(\mathbb{G}^{5,3})$, then it must descend to a representation of $\mathbb{G}^{5,3}/\mathcal{Z}(\mathbb{G}^{5,3}) \cong \Z_p \times \mathbb{H}_1$, and these where already described in Theorem \ref{TeoRepresentationsHd}, where it is proven that $$\widehat{ \Z_p \times \mathbb{H}_1  } = \widehat{\Z_p} \times \widehat{\mathbb{H}}_1   \cong \{(\xi_1 , \xi_2, \xi_3, \xi_4) \in \widehat{\Z}_p^4 \, : \, (\xi_1,\xi_2) \in \Q_p^2/|\xi_4|_p^{-1}\Z_p^2 \},$$and the associated representations are given by$$\pi_{ \xi}(\mathbf{x}) \varphi (u) := e^{2 \pi i \{ \xi \cdot \mathbf{x} +  \xi_4 (u x_2) \}_p} \varphi (u + x_1), \, \, \, \, \varphi \in \mathcal{H}_{\xi},$$where $$\mathcal{H}_{\xi} := \spn_\C \{ \varphi_h \, : \, h \in \Z_p / p^{-\vartheta(\xi_4)} \Z_p  \}, \, \, \, \varphi_h (u) := |\xi_4|_p^{ 1/2} \1_{h + p^{-\vartheta(\xi_4)}} (u), \, \, \dim_\C(\mathcal{H}_{\xi}) =|\xi_4|_p.$$The characters associated to these representations have the form $$\chi_{\pi_{\xi_3}}(\mathbf{x}) = \sum_{h \in  \Z_p/p^{-\vartheta(\xi_4)}\Z_p} e^{ 2 \pi i \{\xi \cdot \mathbf{x} + {\xi_4}(h x_2 ) \}_p} \1_{p^{-\vartheta(\xi_4)} \Z_p}(x_1) = |\xi_4|_p e^{ 2 \pi i \{ \xi \cdot \mathbf{x}  \}_p} \1_{p^{-\vartheta(\xi_4)} \Z_p}(x_1) \1_{p^{-\vartheta(\xi_4)} \Z_p}(x_2), $$thus the irreducibility of $[\pi_{\xi}]$ is proved by the condition  $$\int_{\mathbb{G}^{5,3}} |\chi_{\pi_{\xi}}(\mathbf{x})|^2 d\mathbf{x} \,  = |\xi_4|_p^{2} \int_{\Z_p^5} |\1_{p^{-\vartheta(\xi_4)} \Z_p}(x_1) \1_{p^{-\vartheta(\xi_4)} \Z_p}(x_2)|^2 d \mathbf{x} \,  = 1.$$

Alternatively, if a representation is not trivial on the center, then we must have $|\xi_5|_p>1$, and we can realize these in the finite dimensional Hilbert spaces  $$\mathcal{H}_\xi := \mathrm{span}_\C \{ \varphi_h := \|(\xi_4 , \xi_5) \|_p^{1/2}|\xi_5|_p^{1/2}\1_{h + \Z_p^2 / p^{-\vartheta(\xi_4 , \xi_5)}  \Z_p \times p^{-\vartheta(\xi_5)}\Z_p}(u) \, : \, h \in I_\xi\},$$where $$d_\xi :=\mathrm{dim}_\C (\mathcal{H}_\xi) = \| (\xi_4 , \xi_5)\|_p |\xi_5 |_p, \quad I_\xi:= \Z_p^2 / p^{-\vartheta(\xi_4, \xi_5)} \Z_p \times p^{-\vartheta(\xi_5)}\Z_p,$$acting according to the formula $$\pi_\xi(\mathbf{x}) \varphi (u) :=e^{2 \pi i \{ \xi \cdot \mathbf{x} + \xi_4 u_1 x_2 + \xi_5(x_4 u_1   + \frac{1}{2} x_2 u_1^2+x_3 u_2) \}_p} \varphi (u + (x_1 , x_2)).$$To see how this is indeed a representation, let us define the polynomial $$P_\xi (\mathbf{x} , u) = \xi \cdot \mathbf{x} + \xi_4 u_1 x_2 + \xi_5(x_4 u_1   + \frac{1}{2} x_2 u_1^2+x_3 u_2),$$so that \begin{align*}
    P_\xi (\mathbf{y} , u + (x_1 , x_2)) &= \xi \cdot \mathbf{y} + \xi_4 x_1 y_2  + \xi_5( x_1y_4 +\frac{1}{2} x_1^2 y_2+ x_2 y_3 )\\ & \quad  + \xi_4 y_2  u_1 + \xi_5(( y_4 +x_1 y_2)u_1  + \frac{1}{2} u_1^2 y_2 + y_3 u_2)   
\end{align*}proving how $$P_\xi(\mathbf{x} , u ) + P_\xi (\mathbf{y} , u + (x_1, x_2)) = P_\xi (\mathbf{x} \star \mathbf{y}, u), \quad \text{and} \quad \pi_\xi (\mathbf{x}) \pi_\xi (\mathbf{y}) \varphi (u) = \pi_\xi (\mathbf{x} \star \mathbf{y}) \varphi(u).$$Let us simplify our notation by writing $K_\xi := p^{-\vartheta( \xi_4, \xi_5)} \Z_p \times p^{-\vartheta(\xi_5)} \Z_p.$ We can compute \begin{align*}
    (\pi_{\xi } (\mathbf{x}))_{h h'}&= (\pi_{\xi} (\mathbf{x}) \varphi_h , \varphi_{h'})_{L^2 (\Z_p^2)} \\ &= \|(\xi_4 , \xi_5) \|_p |\xi_5|_p \int_{\Z_p^2} e^{2 \pi i \{ \xi \cdot \mathbf{x}+\xi_4 u_1 x_2 + \xi_5(x_4 u_1   + \frac{1}{2} x_2 u_1^2+x_3 u_2) \}_p} \1_h (u +(x_1,x_2)) \1_{h'} (u) du \\ &= \| (\xi_4 , \xi_5) \|_p |\xi_5|_p  e^{2 \pi i \{\xi \cdot \mathbf{x} \}_p} \1_{h - h' + K_\xi} (x_1,x_2) \int_{h' + K_\xi}e^{2 \pi i \{\xi_4 u_1 x_2+ \xi_5(x_4 u_1   + \frac{1}{2} x_2 u_1^2+x_3 u_2) \}_p} du \\ &= e^{2 \pi i \{ \xi \cdot \mathbf{x}+  \xi_4 x_2 (h_1')  +  \xi_5(x_4 h_1'   + \frac{1}{2} x_2 (h_1')^2 + x_2 (h_1') + x_3(h'_2)) \}_p } \1_{h' - h + \Z_p^2 / p^{-\vartheta(\xi_4, \xi_5)} \Z_p \times p^{-\vartheta(\xi_5)}\Z_p } (x_1,x_2), 
\end{align*}and from this it readily follows that \begin{align*}
    \chi_{\pi_\xi}(\mathbf{x}) &= \sum_{h \in I_\xi  } \| (\xi_4 , \xi_5)\|_p |\xi_5|_p  e^{2 \pi i \{\xi \cdot \mathbf{x} \}_p} \1_{ K_\xi} (x_1,x_2) \int_{h' + K_\xi}e^{2 \pi i \{ \xi_4 u_1 x_2 + \xi_5(x_4 u_1   + \frac{1}{2} x_2 u_1^2+x_3 u_2) \}_p} du \\ &= \| (\xi_4 , \xi_5)\|_p |\xi_5|_p  e^{2 \pi i \{\xi \cdot \mathbf{x} \}_p} \1_{K_\xi} (x_1,x_2) \int_{\Z_p^2}e^{2 \pi i \{ \xi_4 u_1 x_2+ \xi_5(x_4 u_1   + \frac{1}{2} x_2 u_1^2+x_3 u_2) \}_p} du \\&= \| (\xi_4 , \xi_5)\|_p |\xi_5|_p  e^{2 \pi i \{\xi \cdot \mathbf{x} \}_p} \1_{p^{-\vartheta(\xi_4, \xi_5)} \Z_p \times p^{-\vartheta(\xi_5)}\Z_p} (x_1,x_2)\1_{p^{-\vartheta(\xi_5)}\Z_p}(x_3)  \int_{\Z_p} e^{2 \pi i \{ \xi_4 u_1 x_2+ \xi_5(x_4 u_1   + \frac{1}{2} x_2 u_1^2) \}_p} du_1 \\ &=\| (\xi_4 , \xi_5)\|_p |\xi_5|_p  e^{2 \pi i \{\xi \cdot \mathbf{x} \}_p} \1_{K_\xi} (x_1,x_2)\1_{p^{-\vartheta(\xi_5)} \Z_p}(x_3)  \int_{\Z_p} e^{2 \pi i \{ (\xi_4  x_2+ \xi_5 x_4) u_1 \}_p} du_1.
\end{align*}
Finally we have two cases. If $| \xi_4 |_p \leq |\xi_5|_p$ then
\begin{align*}
    \chi_{\pi_\xi} (\mathbf{x}) & =  |\xi_5|_p^2  e^{2 \pi i \{\xi \cdot \mathbf{x} \}_p} \1_{p^{-\vartheta(\xi_5)} \Z_p^2} (x_1,x_2)\1_{p^{-\vartheta(\xi_5)} \Z_p}(x_3)  \int_{\Z_p} e^{2 \pi i \{ \xi_5 x_4 u_1 \}_p} du_1 \\ &= |\xi_5|_p^2  e^{2 \pi i \{\xi_1 x_1 + \xi_2 x_2 + \xi_3 x_3 \}_p} \1_{p^{-\vartheta(\xi_5)} \Z_p^4} (x_1,x_2, x_3 , x_4), 
\end{align*}
so that $\xi_4$ disappears and we see how we only need to count those characters indexed by $$(\xi_1, \xi_2, \xi_3) \in \Q_p^3 / p^{\vartheta(\xi_5)}\Z_p .$$Notice how these are all irreducible since $$\int_{\mathbb{G}^{5,3}} |\chi_{\pi_\xi} (\mathbf{x})|^2 d\mathbf{x}  = \int_{\Z_p^5} |\xi_5|_p^4 \1_{p^{-\vartheta(\xi_5)} \Z_p^4} (x_1,x_2, x_3 , x_4) d \mathbf{x} =1. $$ 
In the case when $|\xi_4|_p > |\xi_5|_p$,
\begin{align*}
    \chi_{\pi_\xi} (\mathbf{x}) & =  |\xi_4|_p |\xi_5|_p  e^{2 \pi i \{\xi \cdot \mathbf{x} \}_p} \1_{p^{-\vartheta(\xi_4)} \Z_p} (x_1) \1_{p^{-\vartheta(\xi_5)} \Z_p} (x_2)\1_{p^{-\vartheta(\xi_5)}}(x_3)  \int_{\Z_p} e^{2 \pi i \{ (\xi_4  x_2+ \xi_5 x_4) u_1 \}_p} du_1 \\ &= |\xi_4|_p |\xi_5|_p  e^{2 \pi i \{\xi  \cdot \mathbf{x} \}_p}  \1_{|p^{-\vartheta(\xi_4)} \Z_p} (x_1) \1_{p^{-\vartheta(\xi_5)} \Z_p} (x_2)\1_{p^{-\vartheta(\xi_5)} \Z_p}(x_3) \1_{\Z_p} ( \xi_4  x_2+ \xi_5 x_4),
\end{align*}
so the desired representations are indexed by $$(\xi_1, \xi_2) \in \Q_p^2 / p^{\vartheta(\xi_4)} \Z_p^2, \, \,  \xi_3 \in \Q_p / p^{\vartheta(\xi_5)} \Z_p, $$and these are irreducible because \begin{align*}
 \int_{\mathbb{G}^{5,3}}  | \chi_{\pi_\xi} (\mathbf{x})|^2 d \mathbf{x} & =  \int_{\Z_p^5} |\xi_4|_p^2 |\xi_5|_p^2    \1_{p^{-\vartheta(\xi_4)} \Z_p} (x_1) \1_{p^{-\vartheta(\xi_5)}\Z_p} (x_2)\1_{p^{-\vartheta(\xi_5)}\Z_p}(x_3) \Big|  \int_{\Z_p} e^{2 \pi i \{ (\xi_4  x_2+ \xi_5 x_4) u_1 \}_p} du_1 \Big|^2 d\mathbf{x}\\ &=  \int_{\Z_p^5} |\xi_4|_p^2 |\xi_5|_p    \1_{p^{-\vartheta(\xi_4)} \Z_p} (x_1) \1_{p^{-\vartheta(\xi_5)} \Z_p}(x_3) \1_{\Z_p} ( \xi_4  x_2+ \xi_5 x_4) d\mathbf{x} \\ &=  \int_{\Z_p^5} |\xi_4|_p    \1_{\Z_p} ( \xi_4  x_2+ \xi_5 x_4) dx_2 dx_4  = 1 ,
\end{align*}where the last equality comes as an application of Lemma \ref{lemaauxG53}. To conclude our proof, we define the sets \begin{itemize}
    \item $$A_1(n):= \{ \| \xi \|_p \leq p^n \, : \, |\xi_5|_p=1, \, \xi_3 \in \widehat{\Z}_p , \,\,\, (\xi_1, \xi_2 ) \in \Q_p^2 / p^{\vartheta(\xi_4)} \Z_p\}.$$
    \item $A_2(n):=$ $$\{\| \xi \|_p \leq p^n \, : \, 1<  |\xi_5|_p,\, \,  \,  (\xi_1, \xi_2) \in \Q_p^2/p^{\vartheta(\xi_4)} \Z_p^2,\,\, (\xi_3, \xi_4) \in \Q_p^2/p^{\vartheta(\xi_5)}\Z_p^2, \,\, |\xi_4|_p >1\}.$$
    \item $$A_3(n):=\{\| \xi \|_p \leq p^n \, : \,1=|\xi_4|_p<|\xi_5|_p , \,\,\,  (\xi_1, \xi_2, \xi_3) \in \Q_p^3/ p^{\vartheta(\xi_5)} \Z_p^3\}.$$
\end{itemize} 
The representations obtained are all the desired representations since \begin{align*}
    \sum_{[\pi_\xi] \in \widehat{\mathbb{G}}^{5,3} \cap (\mathbb{G}^{5,3}(p^n\Z_p))^{\bot}} &d_\xi^2 =  \sum_{\xi \in A_1(n) \cup A_2 (n) \cup A_3(n)} d_\xi^2= \sum_{\xi \in  A_1(n)  }d_\xi^2 + \sum_{\xi \in A_2(n)  }d_\xi^2 + \sum_{\xi \in  A_3(n) } d_\xi^2 \\ &= p^{4n} + \sum_{1<|\xi_5|_p \leq p^n} \sum_{(\xi_1,\xi_2, \xi_3) \in \Q_p^3 / p^{\vartheta(\xi_5)} \Z_p^{3}} (|\xi_5|_p^2)^2 \\ & \quad + \sum_{1<|\xi_5|_p \leq p^n}   \sum_{ (\xi_3, \xi_4) \in \Q_p^2 / p^{\vartheta(\xi_5)} \Z_p^2 , \,\, |\xi_4|_p > |\xi_5|_p } \sum_{(\xi_1, \xi_2) \in \Q_p^2 / p^{\vartheta(\xi_4)}\Z_p^2} (|\xi_4|_p|\xi_5|_p)^2 \\ &= p^{4n} + p^{3n} \sum_{1<|\xi_5|_p \leq p^n} |\xi_5|_p +  \sum_{1<|\xi_5|_p \leq p^n} \sum_{\xi_4 \in \Q_p / p^{\vartheta(\xi_5)}\Z_p, \,\, |\xi_4|_p > |\xi_5|_p } |\xi_5|_p \\&=p^{4n} + p^{3n} \sum_{1<|\xi_5|_p \leq p^n} |\xi_5|_p +  p^{3n}\sum_{1<|\xi_5|_p \leq p^n} \frac{p^n - |\xi_5|_p}{|\xi_5|_p } |\xi_5|_p  \\&=p^{5n}= |\mathbb{G}^{5,3}/\mathbb{G}^{5,3}(p^n \Z_p)|. 
\end{align*} 
With this, the proof is concluded. 
\end{proof}
According to the above theorem, functions $f \in L^2 (\mathbb{G}^{5,3})$ have the Fourier series representation \begin{align*}
    f(\mathbf{x}) &= \sum_{ \| (\xi_4, \xi_5)\|_p =1 } \widehat{f}(\xi) e^{2 \pi i \{ \xi \cdot \mathbf{x}\}_p} + \sum_{\, |\xi_5|_p >1 } \sum_{(\xi_3 , \xi_4) \in \Q_p^2 / p^{\vartheta(\xi_5)}\Z_p^{-1}} \sum_{(\xi_1, \xi_2) \in \Q_p^2 / p^{ \vartheta(\xi_3 , \xi_4)} \Z_p^2} \| (\xi_4, \xi_5) \|_p |\xi_5|_p Tr[ \pi_\xi (\mathbf{x}) \widehat{f}(\xi)]  ,
\end{align*}where the Fourier transform $\widehat{f}(\xi)$ is the linear operator defined as \begin{align*}
    \widehat{f}(\xi) \varphi (u)&:= \int_{\mathbb{G}^{5,3}} f(\mathbf{x}) \pi_\xi^* (\mathbf{x}) \varphi(u) d \mathbf{x} \\ &= \int_{\Z_p^5} f(\mathbf{x}) e^{-2 \pi i \{ \xi \cdot \mathbf{x} + \xi_4 ((u_1 - x_1) x_2)+ \xi_5(x_4 (u_1 -x_1)   + \frac{1}{2} x_2 (u_1 - x_1)^2+x_3 (u_2- x_2)) \}_p} \varphi (u - (x_1 , x_2)) , d\mathbf{x}, \quad \varphi \in \mathcal{H}_\xi.
\end{align*}Alternatively, by using our expressions for the matrix coefficients we can see this operator as the matrix $\widehat{f}(\xi) \in \C^{\|(\xi_4, \xi_5)\|_p|\xi_5|_p \times \|(\xi_4, \xi_5)\|_p |\xi_5|_p}$ defined by the expression \begin{align*}
    \widehat{f}(\xi)_{h h'}&:= \int_{\mathbb{G}^{5,3}} f(\mathbf{x}) \pi_\xi^* (\mathbf{x})_{hh'} d \mathbf{x} \\ &= \int_{\Z_p^5} f(\mathbf{x}) e^{ -2 \pi i \{ \xi \cdot \mathbf{x}+  \xi_4 x_2 (h_1)  +  \xi_5(x_4 h_1   + \frac{1}{2} x_2 (h_1)^2 + x_2 (h_1) + x_3(h_2)) \}_p } \1_{h - h' + \Z_p^2 / p^{-\vartheta(\xi_4,\xi_5)} \Z_p \times p^{-\vartheta(\xi_5)}\Z_p } (x_1,x_2) d\mathbf{x}, 
\end{align*}where $\varphi \in \mathcal{H}_\xi$, which in terms of the $\Z_p^5$-Fourier transform is expressed as \begin{align*}
    \widehat{f}&(\xi)_{h h'} = \mathcal{F}_{\Z_p^5} [\1_{h - h' + \Z_p^2 / p^{\vartheta(\xi_4, \xi_5)} \Z_p \times p^{\vartheta(\xi_5)}\Z_p } (x_1,x_2)f] (\xi_1 ,\xi_2 + h_1 + \xi_5 (\frac{h_1^2}{2} + h_1), \xi_3 + \xi_5 h_2, \xi_4 , \xi_5 ) \\ &=\mathcal{F}_{\Z_p^5} [\1_{h - h' + \Z_p^2 / p^{\vartheta(\xi_4, \xi_5)} \Z_p \times p^{\vartheta(\xi_5)}\Z_p } ] *_{\widehat{\Z}_p^5} \mathcal{F}_{\Z_p^5}[f] (\xi_1 ,\xi_2 + h_1 + \xi_5 (\frac{h_1^2}{2} + h_1), \xi_3 + \xi_5 h_2, \xi_4 , \xi_5 ).
\end{align*}
With this Fourier series representation, a linear invariant operator $T$ can be written as a pseudo-differential operator \begin{align*}
    T_{\sigma}f(\mathbf{x}) &= \sum_{ \| (\xi_4, \xi_5)\|_p =1 } \sigma_T (\xi) \widehat{f}(\xi) e^{2 \pi i \{ \xi \cdot \mathbf{x}\}_p} +\\& \quad \quad  \sum_{\, |\xi_5|_p >1 } \sum_{(\xi_3 , \xi_4) \in \Q_p^2 / p^{\vartheta(\xi_5)}\Z_p^{-1}} \sum_{(\xi_1, \xi_2) \in \Q_p^2 / p^{\vartheta(\xi_3 , \xi_4)} \Z_p^2} \| (\xi_3, \xi_4) \|_p |\xi_5|_p Tr[ \pi_\xi (\mathbf{x}) \sigma_T(\xi) \widehat{f}(\xi)],
\end{align*}

Notice how this time $\mathbb{G}^{5,3}$ is not an stratified group. Still, given any collection $\textbf{W}:=\{W_1 , W_2, W_3$\}  with $$W_i = w_i^1 X_1 + w_i^2 X_2 + w_i^3 X_3,$$which spans $\mathfrak{g}^{5,3} / [\mathfrak{g}^{5,3}, \mathfrak{g}^{5,3}] $, we can define its associated Vladimirov sub-Laplacian as $$\mathcal{L}_{sub, \textbf{W}}^\alpha f (x) : = (\partial^\alpha_{W_1} + \partial^\alpha_{W_2} + \partial^\alpha_{W_3}) f (x), \,\,\, f \in \mathcal{D}(\mathbb{G}^{5,3}).$$For $\alpha>0$ the following theorem proves how our conjecture for this operator holds true on $\mathbb{G}^{5,3}$ too, even when we lack here the graded or stratified structure 
\begin{teo}
The Vladimirov sub-Laplacian $\mathcal{L}_{sub, \textbf{W}}^\alpha$ associated with the collection $\textbf{W}$ is a globally hypoelliptic operator, which is invertible in the space of mean zero functions. Moreover, the space $L^2(\mathbb{G}^{5,3})$ can be written as the direct sum $$L^2(\mathbb{G}^{5,3}) = \overline{\bigoplus_{\xi \in \widehat{\mathbb{G}}^{5,3}} \bigoplus_{h' \in  I_\xi} \mathcal{V}_{\xi}^{h'}}, \, \,\, \mathcal{V}_{\xi} = \bigoplus_{h' \in  I_\xi} \mathcal{V}_{\xi}^{h'}, $$where each finite-dimensional sub-space$$\mathcal{V}_{\xi}^{h'}:= \mathrm{span}_\C \{ (\pi_{\xi})_{hh'} \, : \, h \in I_\xi \},$$is an invariant sub-space of $\mathcal{L}_{sub, \textbf{W}}^\alpha$, and its spectrum restricted to $\mathcal{V}_{\xi }^{h'}$ is given by  \[Spec(\mathcal{L}_{sub, \textbf{W}}^\alpha|_{\mathcal{V}_{\xi}^{h'}})=\sum_{i=1}^3| (\xi_1 + \tau_1, \xi_2 + \tau_2 + \frac{1}{2}(h_1')^2 \xi_5, \xi_3 + (h_2')\xi_5 ) \cdot W_i|_p^\alpha - 3\frac{1 - p^{-1}}{1 - p^{-(\alpha + 1)}},\]and the corresponding eigenfunctions are given by $$\mathscr{e}_{\xi , h' , \tau} (\mathbf{x}) := e^{2 \pi i \{ \xi \cdot \mathbf{x}+\xi_4 x_2 (h_1')+ \xi_5 x_3(h'_2) + \xi_5(x_4 h_1'   + \frac{1}{2} x_2 (h_1')^2) + \tau \cdot (x_1 , x_2) \}_p } ,$$where $1 \leq | \tau_1 |_p \leq \|(\xi_4, \xi_5)\|_p, \,\,1 \leq | \tau_2 |_p \leq | \xi_5|_p .$  
\end{teo}
\begin{proof}
As we anticipated, we will be following the same arguments as in Theorem \ref{teosubLapG22}. We focus again on the sub-spaces $$\mathcal{V}_{\xi}^{h'}:= \mathrm{span}_\C \{ (\pi_{\xi})_{hh'} \, : \, h \in I_\xi \}.$$By Theorem \ref{TeoRepresentationsHd}, an alternative description of this space could be \begin{align*}
    &\mathcal{V}_{\xi}^{h'} = e^{2 \pi i \{ \xi \cdot \mathbf{x} + \xi_4 u_1 x_2 + \xi_5(x_4 (h_1')   + \frac{1}{2} x_2 (h_1')^2+x_3 (h_2')) \}_p} \mathcal{D}_{\xi_4 , \xi_5} (\Z_p^2) \\ & =\{  e^{2 \pi i \{ \xi \cdot \mathbf{x} + \xi_4 (h_1') x_2 + \xi_5(x_4 (h_1')   + \frac{1}{2} x_2 (h_1')^2+x_3 (h_2')) \}_p} f \, : \, f((x_1,x_2) + u ) = f(x_1, x_2), \, \, \, \text{for} \, \, \, u \in K_\xi  \},
\end{align*}where $$K_\xi:=:= p^{-\vartheta(\xi_4 , \xi_5)} \Z_p \times p^{-\vartheta(\xi_5)} \Z_p.$$
In this way, it is not hard to see how the following set defines an orthonormal basis for the space: $$\mathscr{e}_{\xi , h' , \tau} (\mathbf{x}) := e^{2 \pi i \{ \xi \cdot \mathbf{x}+\xi_4 x_2 (h_1')+ \xi_5 x_3(h'_2) + \xi_5(x_4 h_1'   + \frac{1}{2} x_2 (h_1')^2) + \tau \cdot (x_1 , x_2) \}_p } ,$$where $1 \leq | \tau_1 |_p \leq \|(\xi_4, \xi_5)\|_p, \,\,1 \leq | \tau_2 |_p \leq | \xi_5|_p .$
With this we can check how \begin{align*}
    \partial_{W_i}^\alpha \mathscr{e}_{\xi , h' , \tau} (\mathbf{x}) &= \Big(\int_{\Z_p} \frac{e^{2 \pi i \{ -t(( \xi_1 + \tau_1) w_i^1 + (\xi_2 + h_1' \xi_4 + \xi_5 \frac{(h_1')^2}{2} + \tau_2) w_i^2 + (\xi_3 + \xi_5 h_2') w_i^3)  \}_p } - 1}{|t|_p^{\alpha + 1}} dt \Big) \mathscr{e}_{\xi , h' , \tau} (\mathbf{x}),
\end{align*}
so that \begin{align*}
    \partial_{W_i}^\alpha & \mathscr{e}_{\xi , h' , \tau} (\mathbf{x}) =\1_{\Q_p \setminus \Z_p}((\xi_1 , \xi_2 + h' \xi_4, \xi_3 + h' \xi_5) \cdot W_i) \\ & \times \Big(| (\xi_1 + \tau, \xi_2 + h_1' \xi_4 + \xi_5 \frac{(h_1')^2}{2} + \tau_2, \xi_3 + h' \xi_5) \cdot W_i|_p^\alpha - \frac{1 - p^{-1}}{1 - p^{-(\alpha + 1)}}\Big) \mathscr{e}_{\xi , h' , \tau}(\mathbf{x}).
\end{align*}Thus we can write $$\mathcal{L}_{sub, \textbf{W}}^\alpha = \Big(\sum_{i=1}^3| (\xi_1 + \tau, \xi_2 + h_1' \xi_4 + \xi_5 \frac{(h_1')^2}{2} + \tau_2, \xi_3 ) \cdot W_i|_p^\alpha - 3\frac{1 - p^{-1}}{1 - p^{-(\alpha + 1)}}\Big)\mathscr{e}_{\xi , h' , \tau} (\mathbf{x}).$$
Finally, in order to prove the global hypoellipticity of $\mathcal{L}_{sub, \textbf{W}}^\alpha$, it should be clear that \begin{align*}
    \| \sigma_{\mathcal{L}_{sub, \textbf{W}}^\alpha}(\xi) \|_{op} &= \| \mathcal{L}_{sub, \textbf{W}}^\alpha|_{\mathcal{V}_{\xi}} \|_{op}  = \max_{h'  \in I_\xi }  \| \mathcal{L}_{sub, \textbf{W}}^\alpha \|_{\mathcal{V}_{\xi}^{h'}} \|_{op},
\end{align*}
and $$\| \sigma_{\mathcal{L}_{sub, \textbf{W}}^\alpha}(\xi) \|_{inf} = \| \mathcal{L}_{sub, \textbf{W}}^\alpha |_{\mathcal{V}_{\xi}} \|_{inf}  = \min_{h'  \in I_\xi }  \| \mathcal{L}_{sub, \textbf{W}}^\alpha |_{\mathcal{V}_{\xi}^{h'}} \|_{inf},$$so, we can conclude that $$\| \sigma_{\mathcal{L}_{sub, \textbf{W}}^\alpha}(\xi) \|_{op} \leq \| \xi\|_p^\alpha, \, \, \text{and} \, \, \, \| (\xi_1 , \xi_2, \xi_3) \|_p^\alpha \leq \| \sigma_{\mathcal{L}_{sub, \textbf{W}}^\alpha}(\xi) \|_{inf}.$$
This concludes the proof.
\end{proof}

\section{The Cartan group $\mathbb{G}^{5,4}$}
Let $p>3$ be a prime number. In this section we consider the group $\mathbb{G}^{5,4}(\Z_p)$, or just $\mathbb{G}^{5,4}$ for simplicity, defined here as $\Z_p^{5}$ together with the non-commutative operation\begin{align*}
    \mathbf{x} \star \textbf{y} &:=(x_1 + y_1)X_1 + (x_2 + y_2)X_2 + ( x_3 + y_3 + x_1 y_2)X_3 \\ &+ ( x_4 + y_4+ \frac{1}{2}x_1^2 y_2 +x_1 y_3)X_4 + ( x_5 + y_5 + \frac{1}{2} x_1 y_2^2 +  x_2y_3  + x_1 x_2 y_2)X_5,
\end{align*}with associated inverse element $$\mathbf{x}^{-1} = (-x_1, - x_2, - x_3 + x_1x_2, - x_4 + x_1x_3 - x_2\frac{x_1^2}{2}, -x_5 + x_2x_3 - \frac{x_2^2}{2}x_1).$$
We can identify this group with the exponential image of the $\Z_p$-Lie algebra $\mathfrak{g}^{5,4}$ defined by the commutation relations$$[X_1,X_2] = X_3, \quad [X_1 , X_3]= X_4,  \quad [X_2 , X_3] = X_5.$$ Some authors call this algebra \emph{the Cartan algebra}, and to its corresponding exponential image \emph{the Cartan group} $\mathbb{G}^{5,4}$. In the following theorem, we will provide an explicit description of the unitary dual of the Cartan group over the $p$-adic integers.  

\begin{rem}
    Recall how, given any $p$-adic number $\lambda$, the symbol $\vartheta(\lambda)$ denotes the $p$-adic valuation of $\lambda$. If $\lambda=(\lambda_1,...,\lambda_d) \in \Q_p^d$ then $$\vartheta(\lambda) = \min_{1 \leq j \leq d} \vartheta(\lambda_j).$$
\end{rem}
\begin{teo}\normalfont
The unitary dual $\widehat{\mathbb{G}}^{5,4}$ of the compact Cartan group $\mathbb{G}^{5,4}$ can be identified with the union of the five following disjoint subsets of $\widehat{\Z}_p^{5}$: $\widehat{\mathbb{G}}^{5,4}=A_1 \cup A_2 \cup A_3 \cup A_4 \cup A_5$, where
$$A_1 =  \{\xi \in \widehat{\Z}_p^{5}  \, : \, |\xi_5|_p=1 \leq |\xi_4|_p <|\xi_3|_p \, \wedge \,  (\xi_1 , \xi_2 , \xi_3) \in \widehat{\mathbb{H}}_1 , \, \text{or}  \,  |\xi_3|_p = 1 \,  \wedge  \,  \xi_1 \in \Q_p / p^{\vartheta(\xi_4)} \Z_p \},$$
$$A_2:= \{\xi \in \widehat{\Z}_p^5 \, : |\xi_4|_p =1 <|\xi_5 |_p, \, \, |\xi_3|_p \leq |\xi_5|_p, \,\,\, \xi_2 \in \Q_p / p^{\vartheta(\xi_5)}\Z_p,\,\, \xi_1 \in \Q_p/\Z_p \},$$ $$A_3:= \{\xi \in \widehat{\Z}_p^5 \, : |\xi_3|_p =1 <|\xi_5 |_p, \, \, |\xi_4|_p > |\xi_5|_p, \,\,\, \xi_1 \in \Q_p / p^{\vartheta(\xi_4)}\Z_p,\,\, \xi_2 \in \Q_p/\Z_p \},$$ $$A_4:= \{\xi \in \widehat{\Z}_p^5 \, : 1 <|\xi_5 |_p,\,\, |\xi_4|_p >|\xi_5|_p, \,\,\, |\xi_4 |_p <|\xi_3|_p, \, \,( \xi_2, \xi_3) \in \Q_p^2 / p^{\vartheta(\xi_3)} \Z_p^2\},$$ $$A_5:= \{\xi \in \widehat{\Z}_p^5 \, : 1 <|\xi_5 |_p,\,\, |\xi_4|_p >|\xi_5|_p \geq |\xi_3|_p,  \, \,( \xi_2, \xi_3) \in \Q_p^2 / p^{\vartheta(\xi_4)} \Z_p^2\}.$$Moreover, if we define the polynomials $P_\xi (\mathbf{x},u):=$ \[\begin{cases}
    \xi \cdot \mathbf{x} + \frac{\xi_3\xi_5}{\xi_4} x_2  u  + \xi_4(  x_2 \frac{u^2}{2} + x_3u ), \, & \, \, \text{if} \, \, \xi \in A_1, \\  \xi \cdot \mathbf{x}  + \frac{\xi_3 \xi_5}{2\xi_4} x_2^2 + \frac{1}{6\xi_4 \xi_5}( \xi_5^3 x_2^3+3\xi_5 x_2 p^{2\vartheta(\xi_3, \xi_4, \xi_5)}u^2 )+p^{\vartheta(\xi_3, \xi_4, \xi_5)}(\frac{\xi_5}{2\xi_4}x_2^2+ \frac{\xi_3}{\xi_4} x_2  + x_3)u, \, & \, \, \text{if} \, \, |\xi_5|_p >1,   
\end{cases} \]then each unitary irreducible representation can be realized in the finite dimensional Hilbert space \[\mathcal{H}_\xi :=
    \mathrm{span}_\C \{  \varphi_h:= \|(\xi_3, \xi_4, \xi_5)\|_p^{1/2} \1_{h+ p^{-\vartheta(\xi_3, \xi_4, \xi_5)}\Z_p} \, : \, h \in I_\xi := \Z_p / p^{-\vartheta(\xi_3, \xi_4, \xi_5)} \Z_p \}\]where $d_\xi := \mathrm{dim}_\C (\mathcal{H}_\xi) =
   \|(\xi_3, \xi_4, \xi_5)\|_p$, and the representation acts according to the formula \[\pi_\xi(\mathbf{x}) \varphi (u) := \begin{cases}
    e^{2 \pi i \{ P_\xi (\mathbf{x},u)  \}_p} \varphi (u + x_1), \, & \, \, \text{if} \, \, \xi \in A_1, \\  e^{2 \pi i \{ P_\xi (\mathbf{x},u)\}_p} \varphi \big(u + \frac{\xi_4x_1 + \xi_5 x_2}{p^{\vartheta(\xi_3, \xi_4, \xi_5)}} \big), \, & \, \, \text{if} \, \, \xi \in A_2 \cup A_3 \cup A_4 \cup A_5.    
\end{cases} \] With the previous realization, and the natural choice of basis for $\mathcal{H}_\xi$, the associated matrix coefficients are going to be\[(\pi_\xi)_{hh'} (\mathbf{x}) =\begin{cases}
    e^{2 \pi i \{ P_1 (\xi , \mathbf{x}, h' )  \}_p} \1_{h'-h + \| (\xi_3 , \xi_4)\|_p\Z_p} ( x_1), \, & \, \, \text{if} \, \, \xi \in A_1, \\
    e^{2 \pi i \{ P_\xi  (\mathbf{x},  h' )\}_p}  \1_{h'-h + p^{-\vartheta(\xi_3, \xi_4, \xi_5)}\Z_p}\big( \frac{\xi_4x_1 + \xi_5 x_2}{p^{\vartheta(\xi_3, \xi_4, \xi_5)}}\big), \, & \, \, \text{if} \, \, \xi \in A_2 \cup A_3 \cup A_4 \cup A_5,      
\end{cases} \]and the associated characters are $\chi_{\pi_\xi} (\mathbf{x}) :=$ \[\begin{cases}
    \| (\xi_3 , \xi_4)\|_p e^{2 \pi i \{ \xi \cdot \mathbf{x}  \}_p} \int_{\Z_p} e^{2 \pi i \{  \xi_4x_2 \frac{u^2}{2} + ( \xi_3 x_2 + \xi_4 x_3)u  \}_p}du, \, & \, \, \text{if} \, \, \xi \in A_1, \\ \frac{\1_{ \Z_p } (\xi_4x_1 +\xi_5x_2)e^{2 \pi i \{ \xi \cdot \mathbf{x} + \frac{\xi_3 \xi_5}{2\xi_4}x_2^2 + \frac{\xi_5^2}{6\xi_4}x_2^3 \}_p}}{\| (\xi_3,\xi_4,  \xi_5)\|_p^{-1}}     \int_{p^{\vartheta(\xi_3,\xi_4,  \xi_5)}\Z_p}e^{2 \pi i \{ \frac{x_2}{2\xi_4}u^2+ (\frac{\xi_5}{2\xi_4}x_2^2 + x_3)u + \frac{\xi_3}{\xi_4}x_2u\}_p} du, \, & \, \, \text{if} \, \, |\xi_5|_p>1.    
\end{cases} \]Sometimes we will use the notation $$\mathcal{V}_\xi := \mathrm{span}_\C \{ (\pi_\xi)_{hh'} (\mathbf{x}) \, : \, h,h' \in \Z_p / p^{-\vartheta(\xi_3,\xi_4,  \xi_5)}\Z_p \}.$$
\end{teo}

\begin{proof}
Let us start by the level one representations, that is, those who are trivial on $\mathbb{G}^{5,4}(p^1 \Z_p)$, and must therefore descend to a representation of the finite group of Lie type $\mathbb{G}^{5,4}/\mathbb{G}^{5,4}(p^1 \Z_p) \cong \mathbb{G}^{5,4}( \mathbb{F}_p)$. For this group, the Kirillov orbit method gives us $4$ kind of co-adjoint orbits, and therefore $4$ different kinds of representations, which we can realize in the following way: 

\begin{enumerate}
    \item First we have the $0$-dimensional orbits, corresponding to the characters of the abelian group $\mathbf{exp}(\mathfrak{g}^{5,4}/[\mathfrak{g}^{5,4},\mathfrak{g}^{5,4}])$, which are given by $$\chi_\xi (\mathbf{x}) = \chi_{(\xi_1, \xi_2, 1 ,1,1)} (\mathbf{x}) = e^{2 \pi i \{\xi \cdot \mathbf{x} \}_p}, \quad \| (\xi_1,\xi_2) \|_p \leq p.$$  
    \item Second, we have those representations which are trivial on the center of $\mathbb{G}^{5,4}$. If a finite-dimensional continuous representation $\pi$ of $\mathbb{G}^{5,4}$ is trivial on $\mathcal{Z}(\mathbb{G}^{5,4})$, then it must descend to a representation of $\mathbb{G}^{5,4}/\mathcal{Z}(\mathbb{G}^{5,4}) \cong \mathbb{H}_1$, and these were already described in Theorem \ref{TeoRepresentationsHd}, where it is proven that $$\widehat{\mathbb{H}}_1 = \{(\xi_1 , \xi_2, \xi_3, 1,1) \in \widehat{\Z}_p^5 \, : \, (\xi_1,\xi_2) \in \Q_p^2/p^{\vartheta(\xi_3)} \Z_p^2 \},$$and the associated representations are given by $$\pi_{ \xi}(\mathbf{x}) \varphi (u) = \pi_{ (1,1,\xi_3,1,1)}(\mathbf{x}) \varphi (u) := e^{2 \pi i \{  \xi_3 (x_3 +u x_2) \}_p} \varphi (u + x_1), \, \, \, \, \varphi \in \mathcal{H}_{\xi}, \quad \| \xi \|_p \leq p. 
    $$
    \item The third possibility is that the representation is not trivial on the center, but it is trivial on the variable $x_5$. In such case, the representation descends to a representation of $\mathbb{G}^{5,4}/\mathbf{exp}(\Z_p X_5) \cong \mathcal{B}_4 (\Z_p)$, and we discussed already how these have the form \[ \pi_{\xi}(\mathbf{x}) \varphi (u)=\pi_{(1,\xi_2, 1,\xi_4,1)}(\mathbf{x}) \varphi (u) := 
    e^{2 \pi i \{ \xi_2 x_2 + \xi_4 (x_4 +  u x_3 + \frac{u^2}{2} x_2) \}_p} \varphi (u + x_1) , \quad \|\xi \|_p \leq p .\] 
    \item For the fourth kind of co-adjoint orbit we can use the realization in \cite[p.p. 13]{bookUnitaryRep}: \begin{align*}
       \quad \quad \pi_\xi(\mathbf{x}) \varphi (u) &= \pi_{(\xi_1 , 1, 1, \xi_4, \xi_5)}(\mathbf{x}) \varphi (u) \\ &:= e^{2 \pi i \{\xi_1 x_1 + \frac{1}{6\xi_4 \xi_5}( \xi_5^3 x_2^3+3\xi_5 x_2 u^2 + 3 \xi_5^2  x_2^2 u)+x_3u \}_p} \varphi(u + \xi_4 x_1 + \xi_5 x_2),
    \end{align*}where the representation space is given by $$\mathcal{H}_\xi  := \mathrm{span}_\C\{\varphi_h := |\xi_5|_p^{1/2}\1_{h + \Z_p} \, : \, h \in p^{\vartheta(\xi_5)}    \Z_p / \Z_p \}.$$
\end{enumerate}
More generally, we can consider $\mathbb{G}^{5,4}$ as an extension of the group $\mathcal{B}_4$, which is itself an extension of $\mathbb{H}_1$, so the representation theory of $\mathbb{G}^{5,4}$ should also include that of $\mathcal{B}_4$, which are precisely those representations indexed by the set $A_1$. See again Theorem \ref{TeoRepresentationsB4} and \cite{velasquezrodriguez2024Engelspectrumvladimirovsublaplaciancompact} for more details.  On the other hand, if a representation is not trivial on the variable $x_5$, then we must have $|  \xi_5 |_p>1$, and we can realize these in the finite dimensional Hilbert spaces  $$\mathcal{H}_\xi' := \mathrm{span}_\C \{\| (\xi_3,\xi_4, \xi_5)\|_p^{1/2}\1_{h +  \Z_p}(u) \, : \, h \in I_\xi':= p^{\vartheta(\xi_3 , \xi_4, \xi_5)}\Z_p /    \Z_p\}, $$acting according to the formula \[\pi_\xi'(\mathbf{x}) \varphi (u) := e^{2 \pi i \{ \xi \cdot \mathbf{x} + \frac{\xi_3}{\xi_4} x_2 u + \frac{\xi_3 \xi_5}{2\xi_4} x_2^2  + \xi_4( \frac{x_1^2}{2}x_2) + \xi_5 ( x_1\frac{x_2^2}{2}) + \frac{1}{6\xi_5}(3x_1 u^2 + 3 \xi_4  x_1^2 u + \xi_4^2 x_1^3) + x_3 u \}_p} \varphi(u + \xi_4 x_1 + \xi_5 x_2).\]To simplify our notation let us write $$ P_\xi'(\mathbf{x} , u) :=  \xi \cdot \mathbf{x} + \frac{\xi_3}{\xi_4} x_2 u + \frac{\xi_3 \xi_5}{2\xi_4} x_2^2 + \frac{1}{6\xi_4 \xi_5}( \xi_5^3 x_2^3+3\xi_5 x_2 u^2 + 3 \xi_5^2  x_2^2 u)+x_3u.$$Then we can see how $\pi_\xi'$ is indeed a representation since:\begin{align*}
    P_\xi'&(\textbf{y}, u + \xi_4 x_1 + \xi_5 x_2) = \xi \cdot \mathbf{y}+ \xi_3 x_1 y_2 + \frac{\xi_3}{\xi_4} y_2 u + \frac{\xi_3 \xi_5}{2\xi_4} y_2^2+\frac{\xi_3 \xi_5}{\xi_4}x_2 y_2 + \frac{1}{6\xi_4 \xi_5}(  \xi_5^3 y_2^3) \\& \quad \quad \quad \quad \quad \quad \quad  + \frac{1}{6 \xi_4 \xi_5 }(3\xi_5 y_2 (u+\xi_4 x_1 + \xi_5 x_2 )^2 + 3 \xi_5^2  y_2^2 (u+\xi_4 x_1 + \xi_5 x_2 ))+y_3(u+\xi_4 x_1 + \xi_5 x_2 )\\&= \xi \cdot \mathbf{y}+ \xi_3 x_1 y_2 + \xi_5 x_1 \frac{y_2^2}{2}+ \frac{\xi_3}{\xi_4} y_2 u + \frac{\xi_3 \xi_5}{2\xi_4} y_2^2+\frac{\xi_3 \xi_5}{\xi_4}x_2 y_2 + \frac{1}{6\xi_4 \xi_5}(  \xi_5^3 y_2^3 + 3\xi_5^3x_2 y_2^2)  \\&  + \frac{1}{6 \xi_4 \xi_5 }(3\xi_5 y_2 (u^2+2(\xi_4 x_1 + \xi_5 x_2 )u +(\xi_4 x_1 + \xi_5 x_2 )^2 ) + 3 \xi_5^2  y_2^2 u)+y_3(u+\xi_4 x_1 + \xi_5 x_2 )=\xi \cdot \mathbf{y}\\& + \xi_3 x_1 y_2 +\xi_4(x_1y_3)+ \xi_5 (x_2y_3+x_1 \frac{y_2^2}{2})+ \frac{\xi_3}{\xi_4} y_2 u + \frac{\xi_3 \xi_5}{2\xi_4} y_2^2+\frac{\xi_3 \xi_5}{\xi_4}x_2 y_2 + \frac{1}{6\xi_4 \xi_5}(  \xi_5^3 y_2^3 + 3\xi_5^3x_2 y_2^2) \\&+   \frac{1}{6 \xi_4 \xi_5 }(3\xi_5 y_2 (u^2 +(\xi_4^2 x_1^2+2\xi_4\xi_5x_1x_2 + \xi_5^2 x_2^2 ) ) + 3 \xi_5^2(  y_2^2 +2x_2y_2)u)+(y_3+x_1y_2)u ,
\end{align*}which we can be rewritten as \begin{align*}
    &P_\xi'(\textbf{y}, u + \xi_4 x_1 + \xi_5 x_2)=\xi \cdot \mathbf{y}+ \xi_3 x_1 y_2+\xi_4 (x_1 y_3 + \frac{x_1^2}{2}y_2)+ \xi_5 ( x_2y_3+   x_1 \frac{y_2^2}{2}+x_1 x_2 y_2)\\ & + \frac{\xi_3}{\xi_4} y_2 u + \frac{\xi_3 \xi_5}{2\xi_4} y_2^2+\frac{\xi_3 \xi_5}{\xi_4}x_2 y_2 + \frac{1}{6\xi_4 \xi_5}(  \xi_5^3 y_2^3 + 3\xi_5^3x_2 y_2^2 + \xi_5^3 x_2^2 y_2 )  \\&  + \frac{1}{6 \xi_4 \xi_5 }(3\xi_5 y_2 u^2  +   3 \xi_5^2(  y_2^2 +2x_2y_2)u)+(y_3+x_1y_2)u,
\end{align*}and from this we obtain
$P_\xi'(\mathbf{x},u) + P_\xi'(\textbf{y}, u + \xi_4 x_1 + \xi_5 x_2) = P_\xi'(\mathbf{x} \star \textbf{y}, u),$ hence
\begin{align*}
    &\pi_{\xi}'(\mathbf{x})\pi_{\xi}'(\textbf{y})\varphi (u) =\pi_{\xi}(\mathbf{x})(e^{2 \pi i  \{ P_\xi'(\textbf{y}, u)\}_p}\varphi (u +\xi_4x_1 + \xi_5 x_2) )\\ &= e^{2 \pi i \{ P_\xi' (\mathbf{x} , u) + P_\xi' (\textbf{y},u + \xi_4 x_1 + x_2)\}_p}  \varphi (u + \xi_4 x_1 + x_2)\\&=e^{2 \pi i \{ P'(\mathbf{x} \star \textbf{y},u) \}_p} \varphi (u + \xi_4(\mathbf{x} \star \textbf{y})_1 + \xi_5 (\mathbf{x} \star \textbf{y})_2)\\ &= \pi_\xi' (\mathbf{x} \star \textbf{y}) \varphi(u).
\end{align*}It is not hard to see how this representation defines an unitary operator on $\mathcal{H}_\xi'$, since \begin{align*}
    &-P_\xi'(\mathbf{x} , u - (\xi_4 x_1 + \xi_5 x_2))= \xi \cdot -\mathbf{x} +\xi_3 x_1 x_2 - \frac{\xi_3}{\xi_4} x_2 u + \frac{\xi_3 \xi_5}{2\xi_4} x_2^2\\& - \frac{1}{6\xi_4 \xi_5}( \xi_5^3 x_2^3+3\xi_5 x_2 (u - (\xi_4 x_1 + \xi_5 x_2))^2 + 3 \xi_5^2  x_2^2 (u - (\xi_4 x_1 + \xi_5 x_2)))-x_3(u - (\xi_4 x_1 + \xi_5 x_2)),
    \end{align*}
    thus 
    \begin{align*}    
    &-P_\xi'(\mathbf{x} , u - (\xi_4 x_1 + \xi_5 x_2))= \xi \cdot -\mathbf{x} +\xi_3 x_1 x_2 +\xi_4 x_1 y_3 + \xi_5 x_2y_3 - \frac{\xi_3}{\xi_4} x_2 u + \frac{\xi_3 \xi_5}{2\xi_4} x_2^2\\& - \frac{1}{6\xi_4 \xi_5}( \xi_5^3 x_2^3+3\xi_5 x_2 (u - (\xi_4 x_1+ \xi_5 x_2)) (u -  \xi_4x_1))-x_3u\\&= \xi \cdot -\mathbf{x} +\xi_3 x_1 x_2 +\xi_4 (x_1 y_3 - \frac{x_1^2}{2}x_2) + \xi_5( x_2y_3 +x_1\frac{x_2^2}{2})- \frac{\xi_3}{\xi_4} x_2 u + \frac{\xi_3 \xi_5}{2\xi_4} x_2^2\\& - \frac{1}{6\xi_4 \xi_5}( \xi_5^3 x_2^3+3\xi_5 x_2 (u^2   - \xi_5 x_2 u ))+(-x_3 + x_1 x_2)u = P_\xi'(\mathbf{x}^{-1} , u).
\end{align*}
By using the natural choice of basis for this space, we can compute \begin{align*}
    &(\pi_{\xi }'  (\mathbf{x}))_{h h'}= (\pi_{\xi}' (\mathbf{x}) \varphi_h , \varphi_{h'})_{L^2 (\Q_p)} \\ &= \| (\xi_3,\xi_4, \xi_5) \|_p \int_{\Q_p} e^{2 \pi i \{ P_\xi '(\mathbf{x} , u) \}_p} \1_{h+\Z_p} (u +\xi_4 x_1 + \xi_5 x_2 ) \1_{h'+\Z_p} (u) du \\ &= \| (\xi_3, \xi_4, \xi_5) \|_p   \1_{h' - h +    \Z_p } (\xi_4x_1 + \xi_5 x_2) \int_{h' +    \Z_p}e^{2 \pi i \{ P_\xi' (\mathbf{x} , u) \}_p} du . 
\end{align*}From this it readily follows that \begin{align*}
    &\chi_{\pi_\xi'}(\mathbf{x})  = \sum_{h \in I_\xi' }\| (\xi_3,\xi_4 , \xi_5)\|_p  \1_{  \Z_p } (\xi_4x_1 + \xi_5 x_2) \int_{h +    \Z_p}e^{2 \pi i \{ P_\xi' (\mathbf{x} , u) \}_p} du \\&=\| (\xi_3,\xi_4,  \xi_5)\|_p  \1_{ \Z_p } (\xi_4x_1 +\xi_5x_2) e^{2 \pi i \{ \xi \cdot \mathbf{x} + \frac{\xi_3 \xi_5}{2\xi_4}x_2^2 + \frac{\xi_5^2}{6\xi_4}x_2^3 \}_p} \\& \quad \quad \quad \times  \int_{p^{\vartheta(\xi_3,\xi_4,  \xi_5)}\Z_p}e^{2 \pi i \{ \frac{x_2}{2\xi_4}u^2+ (\frac{\xi_5}{2\xi_4}x_2^2 + x_3)u + \frac{\xi_3}{\xi_4}x_2u\}_p} du. 
\end{align*}We can use the function with two purposes. First to check how these representations are indeed irreducible, and second, we can also count the number of different non-equivalent classes of unitary irreducible representations, by counting the number of different functions among these characters. We need to consider cases.

\textbf{Case 1:} When $|\xi_4|_p=1$ but $|\xi_5|_p \geq  |\xi_3|_p$ then
$$\chi_{\pi_\xi'}(\mathbf{x})= |\xi_5|_pe^{2 \pi i \{ \xi \cdot \mathbf{x} + \frac{\xi_3 \xi_5}{2}x_2^2 + \frac{\xi_5^2}{6}x_2^3 \}_p} \1_{ \Z_p } (\xi_5 x_2)  \int_{p^{\vartheta(\xi_5)}\Z_p}e^{2 \pi i \{ \frac{x_2}{2}u^2+ (\frac{\xi_5}{2}x_2^2 + x_3)u + \xi_3x_2u\}_p} du,$$so that we can rewrite
\begin{align*}
    &\chi_{\pi_\xi'}(\mathbf{x})= | \xi_5|_pe^{2 \pi i \{ \xi \cdot \mathbf{x} + \frac{\xi_5^2}{6\xi_4 }  x_2^3 \}_p} \1_{ \Z_p } (\xi_5 x_2)  \int_{p^{\vartheta(\xi_5)}\Z_p}e^{2 \pi i \{ \frac{x_2}{2}u^2+ (\frac{\xi_5}{2} x_2^2+ x_3)u + \xi_3x_2u\}_p} du \\&=| \xi_5|_pe^{2 \pi i \{ \xi \cdot \mathbf{x} + \frac{\xi_5^2}{6\xi_4 }  x_2^3 \}_p} \1_{ \Z_p } (\xi_5 x_2) \int_{\Z_p} e^{2 \pi i \{   \frac{1}{2}\xi_5^2 x_2 u^2 + \frac{1}{2} \xi_5^2  x_2^2 u+\xi_5 x_3u + \xi_3 \xi_5 x_2 u \}_p }du \\&=| \xi_5|_pe^{2 \pi i \{ \xi \cdot \mathbf{x} + \frac{\xi_5^2}{6\xi_4 }  x_2^3 \}_p} \1_{ \Z_p } (\xi_5 x_2) \int_{\Z_p} e^{2 \pi i \{   \frac{1}{2}\xi_5^2 x_2 u^2 + \xi_5 x_3u +\xi_3 \xi_5 x_2 u\}_p }du.
\end{align*}

This is the character of an irreducible representation since
\begin{align*}
    \int_{\mathbb{G}^{5,4}}  |\chi_{\pi_\xi'} (\mathbf{x})|^2 d \mathbf{x}  &= |\xi_5|_p^2 \int_{\Z_p^5} \1_{   \Z_p } (\xi_5 x_2) \Big| \int_{\Z_p} e^{2 \pi i \{   \frac{1}{2}\xi_5^2 x_2 u^2 + \xi_5 x_3u + \xi_3 \xi_5 x_2 u \}_p }du \Big|^2 d \mathbf{x} \\ &= |\xi_5|_p \int_{\Z_p^5}  \Big| \int_{\Z_p} e^{2 \pi i \{   \frac{1}{2}\xi_5 x_2 u^2 + (\xi_5 x_3 +\xi_3 x_2) u\}_p }du \Big|^2 d \mathbf{x} =1,
\end{align*}
where we used Lemma \ref{lemaaux} above. Notice how the appearance of the function  $\1_{   \Z_p } (\xi_5 x_2)$ implies that we must choose $\xi_2 \in \Q_p / p^{\vartheta(\xi_5)}\Z_p$. These are precisely the representations indexed by the set $A_2$.  

\textbf{Case 2:} If it happens that $|\xi_3|_p = 1$ and $|\xi_5|_p < |\xi_4|_p$, then we know how: 
\begin{align*}
    &\chi_{\pi_\xi'}(\mathbf{x})= |\xi_4|_p  \1_{ \Z_p } (\xi_4x_1 +\xi_5x_2) e^{2 \pi i \{ \xi \cdot \mathbf{x}  + \frac{\xi_5^2}{6\xi_4}x_2^3 \}_p} \int_{p^{\vartheta(\xi_4)}\Z_p}e^{2 \pi i \{ \frac{x_2}{2\xi_4}u^2 + (\frac{\xi_5}{2\xi_4}x_2^2 + x_3)u \}_p} du \\&= |\xi_4|_p  \1_{ \Z_p } (\xi_4x_1 +\xi_5x_2) e^{2 \pi i \{ \xi \cdot \mathbf{x}  + \frac{\xi_5^2}{6\xi_4}x_2^3 \}_p} \int_{\Z_p}e^{2 \pi i \{ \frac{x_2 \xi_4}{2}u^2+ (\frac{\xi_5}{2}x_2^2 + \xi_4 x_3)  u \}_p} du.
\end{align*}
To see how this is the character of an irreducible representation, define the auxiliary functions $$ f(x_2 , x_3) =\Big| \int_{\Z_p}e^{2 \pi i \{ \frac{x_2 \xi_4}{2}u^2+ (\frac{\xi_5}{2}x_2^2 + \xi_4 x_3)  u \}_p} du \Big|, \quad F(x_2 , x_3 , y ):= \Big| \int_{\Z_p}e^{2 \pi i \{ \frac{x_2 \xi_4}{2}u^2+ (\frac{\xi_5}{2}x_2 y + \xi_4 x_3)  u \}_p} du \Big| ,$$so that $F(x_1, x_2, x_2) = f(x_1 , x_2)$, and $$\inf_{y \in \Z_p} \int_{\Z_p^2} F(x_2, x_3, y)^2 dx_2 dx_3 \leq \int_{\Z_p^2} f(x_2 , x_3)^2 dx_2 dx_3 \leq \inf_{y \in \Z_p} \int_{\Z_p^2} F(x_2, x_3, y)^2 dx_2 dx_3.$$however, by  Lemma \ref{lemaauxG54}, for any $y$ we have $$\int_{\Z_p^2} F(x_2, x_3, y)^2 dx_2 dx_3 = |\xi_4|_p^{-1} = \int_{\Z_p^2} f(x_2 , x_3)^2 dx_2 dx_3. $$With this argument we can prove that
\begin{align*}
    \int_{\mathbb{G}^{5,4}}  |\chi_{\pi_\xi'} (\mathbf{x})|^2 d \mathbf{x}  &= |\xi_4|_p^2 \int_{\Z_p^5} \1_{   \Z_p } (\xi_4x_1 +\xi_5 x_2) \Big| e^{2 \pi i \{ \frac{x_2 \xi_4}{2}u^2+ (\frac{\xi_5}{2}x_2^2 + \xi_4 x_3)  u \}_p} du \Big|^2 d \mathbf{x}\\ &= |\xi_4|_p^2 \int_{\Z_p^5} \1_{ \frac{\xi_5}{\xi_4}x_2 + p^{-\vartheta(\xi_4)}  \Z_p } (x_1 ) \Big| \int_{\Z_p} e^{2 \pi i \{  \frac{\xi_4}{2 }  x_2 u^2 +x_3 \xi_4 u \}_p}du \Big|^2 d \mathbf{x}=1.
\end{align*} Now, to count all the different functions among these notices how when $|\xi_1|_p \leq |\xi_4|_p$ we get $$\xi_1 x_1 = \frac{\xi_1}{\xi_4}(\xi_5 x_2 + \xi_4 x_1) - \frac{\xi_1 \xi_5}{\xi_4}x_2=- \frac{\xi_1 \xi_5}{\xi_4}x_2 \in \Q_p / \Z_p,$$
so that we only need to count $\xi_1$ when $\xi_1 \in \Q_p / p^{\vartheta(\xi_4)}\Z_p$, and these correspond to the set $A_3$.

\textbf{Case 3:} Now we have the representations indexed by the set $A_4$, which in this case have the property $|\xi_3|_p > |\xi_4|_p > |\xi_5|_p$, and they are all irreducible since
\begin{align*}
    \int_{\mathbb{G}^{5,4}}  |\chi_{\pi_\xi'} (\mathbf{x})|^2 d \mathbf{x}  &= |\xi_3|_p^2 \int_{\Z_p^5}   \1_{ \Z_p } (\xi_4x_1 +\xi_5x_2) \Big|  \int_{p^{\vartheta(\xi_3)}\Z_p}e^{2 \pi i \{ \frac{x_2}{2\xi_4}u^2+ (\frac{\xi_5}{2\xi_4}x_2^2 + x_3)u + \frac{\xi_3}{\xi_4}x_2u\}_p} du \Big|^2 d \mathbf{x}\\ &= |\xi_3|_p^2 \int_{\Z_p^5} \1_{ \frac{\xi_5}{\xi_4}x_2 + p^{-\vartheta(\xi_4)}  \Z_p } (x_1 ) \Big|  \int_{\Z_p}e^{2 \pi i \{ \frac{\xi_3^2 x_2}{2\xi_4}u^2+ (\frac{\xi_5 \xi_3}{2\xi_4}x_2^2 + \xi_3 x_3)u + \frac{\xi_3^2}{\xi_4}x_2u\}_p} du \Big|^2 d \mathbf{x} \\ &= \frac{|\xi_3|_p^2}{|\xi_4|_p} \int_{\Z_p^2}   \Big|  \int_{\Z_p}e^{2 \pi i \{ \frac{\xi_3^2 x_2}{2\xi_4}u^2+ (\frac{\xi_5 \xi_3}{2\xi_4}x_2^2 + \xi_3 x_3)u + \frac{\xi_3^2}{\xi_4}x_2u\}_p} du \Big|^2 d x_2 dx_3=1,
\end{align*}where for the last equality we have used Lemma \ref{lemaauxG54}. 

\textbf{Case 4:} These are the representations indexed by $A_5$, and they have the property $|\xi_4|_p > |\xi_5|_p \geq  |\xi_3|_p,$ where 
\begin{align*}
    \int_{\mathbb{G}^{5,4}}  |\chi_{\pi_\xi'} (\mathbf{x})|^2 d \mathbf{x}  &= |\xi_4|_p^2 \int_{\Z_p^5}   \1_{ \Z_p } (\xi_4x_1 +\xi_5x_2) \Big|  \int_{p^{\vartheta(\xi_4)}\Z_p}e^{2 \pi i \{ \frac{x_2}{2\xi_4}u^2+ (\frac{\xi_5}{2\xi_4}x_2^2 + x_3)u + \frac{\xi_3}{\xi_4}x_2u\}_p} du \Big|^2 d \mathbf{x}\\ &= |\xi_4|_p^2 \int_{\Z_p^5} \1_{ \frac{\xi_5}{\xi_4}x_2 + p^{-\vartheta(\xi_4)}  \Z_p } (x_1 ) \Big|  \int_{\Z_p}e^{2 \pi i \{ \frac{\xi_4 x_2}{2}u^2+ (\frac{\xi_5 }{2}x_2^2 + \xi_4 x_3)u + \frac{\xi_3}{}x_2u\}_p} du \Big|^2 d \mathbf{x} \\ &= |\xi_4|_p \int_{\Z_p^2}   \Big|  \int_{\Z_p}e^{2 \pi i \{ \frac{\xi_4 x_2}{2}u^2+ (\frac{\xi_5 }{2}x_2^2 + \xi_4 x_3)u + \xi_3x_2u\}_p} du \Big|^2 d x_2 dx_3=1,
\end{align*}where once again we apply Lemma \ref{lemaauxG54} for the last equality. 

Finally, notice how these are all the desired representations since:
\begin{align*}
    &\sum_{[\pi_\xi] \in \widehat{\mathbb{G}}^{5,4} \cap \mathbb{G}^{5,4}(p^n \Z_p)^\bot } d_\xi^2 = \sum_{\xi \in  A_1 \, : \, \| \xi \|_p \leq p^n  }d_\xi^2 + \sum_{\xi \in A_2 \cup A_3 \cup A_4 \cup A_5 \, : \, \| \xi \|_p \leq p^n  }d_\xi^2   \\ &= p^{4n} + \sum_{1<|\xi_5|_p \leq p^n}\sum_{1 \leq |\xi_4|_p \leq |\xi_5|_p} \sum_{\xi_2 \in p^{-n}\Z_p /  p^{\vartheta(\xi_5)}\Z_p} \sum_{1 \leq |\xi_1|_p \leq p^n}  (|\xi_5 |_p )^2 \\ & \quad \quad + \sum_{1 < |\xi_5|_p \leq p^n} \sum_{ |\xi_5|_p<  |\xi_4|_p \leq p^n} \sum_{\xi_1 \in p^{-n}\Z_p / p^{\vartheta(\xi_4)}\Z_p} \sum_{1 \leq |\xi_2|_p \leq p^n} |\xi_4|_p^2 \\ & \quad \quad +  \sum_{1 < |\xi_5|_p \leq p^n } \sum_{|\xi_5|_p < |\xi_4|_p  \leq p^n} \sum_{ |\xi_4|_p<|\xi_3|_p\leq p^n} \sum_{(\xi_1 , \xi_2) \in \Q_p^2/p^{\vartheta(\xi_3)}\Z_p^2 } |\xi_3|_p^2 \\ &\quad \quad + \sum_{1 < |\xi_5|_p \leq p^n } \sum_{|\xi_5|_p< |\xi_4|_p \leq p^n} \sum_{1 \leq  |\xi_3|_p \leq |\xi_5|_p} \sum_{(\xi_1 , \xi_2) \in \Q_p^2/p^{\vartheta(\xi_4)}\Z_p^2 } |\xi_4|_p^2, 
\end{align*} 
so that working out the expression we obtain \begin{align*}
    &\sum_{[\pi_\xi] \in \widehat{\mathbb{G}}^{5,4} \cap \mathbb{G}^{5,4}(p^n \Z_p)^\bot } d_\xi^2= p^{4n} + \sum_{1<|\xi_5|_p \leq p^n}\sum_{1 \leq |\xi_4|_p \leq |\xi_5|_p}p^{2n}|\xi_5|_p^2  + \sum_{1 < |\xi_5|_p \leq p^n} \sum_{  |\xi_5|_p < |\xi_4|_p \leq p^n} p^{2n} |\xi_4|_p \\ & \quad \quad +  \sum_{1 < |\xi_5|_p \leq p^n } \sum_{|\xi_5|_p < |\xi_4|_p \leq p^n}  p^{2n}(p^n - |\xi_4|_p) + \sum_{1 < |\xi_5|_p \leq p^n }  p^{2n} |\xi_5|_p(p^n - |\xi_5|_p)\\ &= p^{4n} + \sum_{1<|\xi_5|_p \leq p^n}\sum_{|\xi_5|_p<  |\xi_4|_p \leq p^n}p^{3n} +  \sum_{1 < |\xi_5|_p \leq p^n }  p^{3n}|\xi_5|_p  \\ &= p^{4n} + p^{4n} \sum_{1<|\xi_5|_p \leq p^n} 1 =p^{5n}= |\mathbb{G}^{5,4}/\mathbb{G}^{5,4}(p^n \Z_p)|.
\end{align*}
To finish our proof, we need to define the mappings
$$\delta_\lambda : L^2(\Q_p) \to L^2(\Q_p), \quad \delta_\lambda (f) (u) := f(\lambda u).$$Then, for $|\xi_5|_p>1$ e can define the alternative representations $$\pi_\xi(\mathbf{x}) = \delta_{p^{\vartheta(\xi_3, \xi_4, \xi_5)}} \circ \pi_\xi' (\mathbf{x}) \circ \delta_{p^{-\vartheta(\xi_3, \xi_4, \xi_5)}},$$which act on the representation space $$\mathcal{H}_\xi := \spn_\C \{\varphi_h \, : \,\,  h \in \Z_p / p^{-\vartheta(\xi_3,\xi_4,\xi_5)} \Z_p\}, \quad \varphi_h(u) := \|(\xi_3, \xi_4, \xi_5) \|_p^{1/2} \1_{h + p^{-\vartheta(\xi_3,\xi_4,\xi_5)} \Z_p}(u). $$Explicitly, we can write 
\begin{align*}
    \pi_\xi (\mathbf{x})\varphi(u) &= \delta_{p^{\vartheta(\xi_3, \xi_4, \xi_5)}} \circ \pi_\xi' (\mathbf{x}) \circ \delta_{p^{-\vartheta(\xi_3, \xi_4, \xi_5)}} \varphi(u)\\& = \delta_{p^{\vartheta(\xi_3, \xi_4, \xi_5)}} \circ \pi_\xi' (\mathbf{x})  \varphi(p^{-\vartheta(\xi_3, \xi_4, \xi_5)}u) \\ &= \delta_{p^{\vartheta(\xi_3, \xi_4, \xi_5)}} \circ e^{2 \pi i \{ P_\xi ' (\mathbf{x}, u )\}_p}  \varphi(p^{-\vartheta(\xi_3, \xi_4, \xi_5)}(u+ \xi_4x_1 + \xi_5 x_2)) \\ &=  e^{2 \pi i \{ P_\xi ' (\mathbf{x}, p^{\vartheta(\xi_3, \xi_4, \xi_5)} u )\}_p}  \varphi\big(p^{-\vartheta(\xi_3, \xi_4, \xi_5)}(p^{\vartheta(\xi_3, \xi_4, \xi_5)}u)+ \frac{\xi_4x_1 + \xi_5 x_2}{p^{\vartheta(\xi_3, \xi_4, \xi_5)}}\big)\\ &= e^{2 \pi i \{ P_\xi ' (\mathbf{x}, p^{\vartheta(\xi_3, \xi_4, \xi_5)} u )\}_p}  \varphi\big( u + \frac{\xi_4x_1 + \xi_5 x_2}{p^{\vartheta(\xi_3, \xi_4, \xi_5)}}\big),
\end{align*}so that with this realization the matrix coefficients are simply \begin{align*}
    (\pi_\xi (\mathbf{x}))_{hh'} = (\pi_\xi (\mathbf{x})\varphi_h , \varphi_{h'})_{L^2(\Z_p)} = e^{2 \pi i \{ P_\xi ' (\mathbf{x}, p^{\vartheta(\xi_3, \xi_4, \xi_5)} h' )\}_p}  \1_{h'-h + p^{-\vartheta(\xi_3, \xi_4, \xi_5)}\Z_p}\big( \frac{\xi_4x_1 + \xi_5 x_2}{p^{\vartheta(\xi_3, \xi_4, \xi_5)}}\big).
\end{align*}
The proof is finished.
\end{proof}

We will use again our calculations of the unitary dual to describe the spectrum of the Vladimirov-sub-Laplacian, but this time it will be necessary to employ The following lemma, whose proof can be found in \cite{Albeverio2010}: 

\begin{lema}\label{lemaauxcomp}
    Let $\Phi: \Z_p \to \Z_p$ be a parabolic ball–morphism which is continuously differentiable as function on $\Q_p$. Then it holds: $$D^\alpha [f \circ \Phi] (x) = |\Phi'(x)|_p^\alpha D^\alpha[f] \circ \Phi (x).$$
\end{lema}Similar to the previous groups, for a collection $\textbf{W}:=\{W_1 , W_2\}$  with $$W_i = w_i^1 X_1 + w_i^2 X_2 ,$$which spans $\mathfrak{g}^{5,4} / [\mathfrak{g}^{5,4}, \mathfrak{g}^{5,4}] $, we can define its associated Vladimirov sub-Laplacian as $$\mathcal{L}_{sub, \textbf{W}}^\alpha f (x) : = (\partial^\alpha_{W_1} + \partial^\alpha_{W_2} ) f (\mathbf{x}), \,\,\, f \in \mathcal{D}(\mathbb{G}^{5,4}).$$For $\alpha>0$ the following theorem proves how our conjecture for this operator holds true on $\mathbb{G}^{5,4}$ too. For a matter of simplicity, we will consider here only the case $\mathbf{W}=\{ X_1, X_2\}$, and we denote $$\mathcal{L}_{sub}^\alpha f (\mathbf{x}) : = (\partial^\alpha_{X_1} + \partial^\alpha_{X_2} ) f (\mathbf{x}), \,\,\, f \in \mathcal{D}(\mathbb{G}^{5,4}),$$ but similar conclusions will hold for more general collections $\mathbf{W}$. The argument would be the same, but the calculations much longer. 
\begin{teo}\label{tepsubLapG54}
The Vladimirov sub-Laplacian $\mathcal{L}_{sub}^\alpha$ is a globally hypoelliptic operator, which is invertible in the space of mean zero functions. Moreover, the space $L^2(\mathbb{G}^{5,4})$ can be written as the direct sum $$L^2(\mathbb{G}^{5,4}) = \overline{\bigoplus_{\xi \in \widehat{\mathbb{G}}^{5,4}} \bigoplus_{h' \in  I_\xi} \mathcal{V}_{\xi}^{h'}}, \, \,\, \mathcal{V}_{\xi} = \bigoplus_{h' \in  I_\xi} \mathcal{V}_{\xi}^{h'}, $$where each finite-dimensional sub-space$$\mathcal{V}_{\xi}^{h'}:= \mathrm{span}_\C \{ (\pi_{\xi})_{hh'} \, : \, h \in I_\xi \},$$is an invariant sub-space of $\mathcal{L}_{sub, \textbf{W}}^\alpha$, and the spectrum of $\sigma_{\mathcal{L}_{sub}^\alpha} (\xi)$ is composed by the numbers $$ |\xi_1 + \tau  + \frac{1}{2}\xi_5  (h')^2|_p^\alpha  +|\xi_2 + \tau + \frac{\xi_3 \xi_5}{\xi_4}h'|_p^\alpha -2\frac{1 - p^{-1}}{1 - p^{-(\alpha + 1)}},$$where $$1 \leq  |\tau|_p \leq \| (\xi_3, \xi_4 )\|_p, \quad h' \in \Z_p / p^{-\vartheta(\xi_3,\xi_4)}\Z_p.$$ 
\end{teo}

\begin{proof}
In order to prove Theorem we want to use this calculate the associated symbol of $\mathscr{L}_{sub}^\alpha$, and its respective invariant sub-spaces. Let us start by computing the symbols of the directional VT operators $\partial_{X_1}^{\alpha}, \partial_{X_2}^{\alpha}$,  $\alpha>0$, by using what we got for the matrix coefficients. We only need to consider the case $ |\xi_5|_p > 1$.

We proved how $$(\pi_\xi)_{hh'} (\mathbf{x})=e^{2 \pi i \{ P_\xi (\mathbf{x} , h')\}_p} \1_{h'-h + p^{-\vartheta(\xi_3, \xi_4, \xi_5)} \Z_p}( \frac{\xi_4x_1 + \xi_5 x_2}{p^{\vartheta(\xi_3, \xi_4, \xi_5)}}),$$and we know how each function $\1_{h'-h + p^{-\vartheta(\xi_3, \xi_5)} \Z_p}$ has the Fourier series expansion $$\1_{h'-h + p^{-\vartheta(\xi_3, \xi_4, \xi_5)} \Z_p} (u) = \sum_{1 \leq |\tau|_p \leq  \| (\xi_3, \xi_4, \xi_5) \|_p }\| (\xi_3, \xi_4, \xi_5) \|_p^{-1} e^{ -2 \pi i \{ \tau (h' -h ) \}_p} e^{ -2 \pi i \{ \tau u \}_p} , $$ so that any $f \in \mathcal{V}_\xi^{h'}$ can be written as \begin{align*}
    f(u)=\sum_{h \in \Z_p / p^{- (\xi_3 , \xi_4,\xi_5)} \Z_p } & C_h \|(\xi_3, \xi_4, \xi_5) \|_p^{1/2}\1_{h'-h + p^{-\vartheta(\xi_3,\xi_4, \xi_5)}\Z_p} (u) \\&=  \sum_{1 \leq |\tau|_p \leq  \| (\xi_3, \xi_4, \xi_5) \|_p }\Big( \sum_{h \in \Z_p / p^{-\vartheta(\xi_3,\xi_4,\xi_5)}\Z_p}\| (\xi_3, \xi_4, \xi_5) \|_p^{-1/2} e^{ -2 \pi i \{ \tau (h' -h ) \}_p} C_h \Big) e^{ -2 \pi i \{ \tau u \}_p} \\ &=\sum_{1 \leq |\tau|_p \leq  \| (\xi_3, \xi_4, \xi_5) \|_p } (M_{h'} \mathbf{C}_f)_{\tau} e^{ 2 \pi i \{ \tau u \}_p}.
\end{align*}In order to calculate the symbol of $\partial_{X_1}^{\alpha}$, let us introduce some notation: let $\psi_\xi^1 : \Z_p \to \C^{d_\xi \times d_\xi}$ be the matrix function with entries $$(\psi_{\xi}^1)_{hh'} (u):= e^{2 \pi i \{ \xi_1 u \}_p}\1_{h' - h + p^{-\vartheta(\xi_3,\xi_4,\xi_5)}\Z_p  }\big(\frac{\xi_4 u}{ p^{\vartheta(\xi_3, \xi_4, \xi_5)}} \big), \, \, \, u \in \Z_p.$$ The associated associated symbol $\sigma_{\partial_{X_1}^{\alpha}} (\xi) = \partial_{X_1}^{\alpha} \pi_{\xi}^1|_{x=0}$ is a Toeplitz matrix which can be written as  \begin{align*}
    \sigma_{\partial_{X_1}^{\alpha}}& (\xi)_{hh'} = (\partial_u^\alpha \psi_\xi^1 (u)|_{u=0})_{hh'} \\&= \sum_{1 \leq |\tau|_p \leq  \| (\xi_3, \xi_4,\xi_5) \|_p }\Big(\big| \xi_1 + \frac{\xi_4\tau}{p^{\vartheta(\xi_3, \xi_4, \xi_5)}}\big|_p^\alpha - \frac{1 - p^{-1}}{1 - p^{-(\alpha + 1)}} \Big) \times (M_{h'} \mathbf{C}_f)_{\tau} e^{ 2 \pi i \{  (\xi_1 +\tau \frac{\xi_4}{p^{\vartheta(\xi_3, \xi_4 , \xi_5)}})u\}_p} |_{u=0} \\&= \sum_{1 \leq |\tau|_p \leq  \| (\xi_3, \xi_4,\xi_5) \|_p }\Big(\xi_1 + \big|\frac{\xi_4\tau}{p^{\vartheta(\xi_3, \xi_4, \xi_5)}}\big|_p^\alpha - \frac{1 - p^{-1}}{1 - p^{-(\alpha + 1)}} \Big) \times (M_{h'} \mathbf{C}_f)_{\tau}.
\end{align*}Similarly, let $\psi_\xi^2 : \Z_p \to \C^{d_\xi \times d_\xi}$ be the matrix function with entries \begin{align*}
    (\psi_{\xi}^2)_{hh'} (u)&:= e^{2 \pi i \{ \xi_2 u  + \frac{\xi_3 \xi_5}{2\xi_4} u^2 + \frac{1}{6\xi_4 \xi_5}( \xi_5^3 u^3+3\xi_5  p^{2\vartheta(\xi_3, \xi_4, \xi_5)}(h')^2u )+p^{\vartheta(\xi_3, \xi_4, \xi_5)}(\frac{\xi_5}{2\xi_4}u^2+ \frac{\xi_3}{\xi_4} u )(h') \}_p}\\ &\quad \quad \quad \times \1_{h' - h + \| (\xi_3 , \xi_4, \xi_5) \|_p \Z_p }(\frac{\xi_5}{p^{\vartheta(\xi_3, \xi_4, \xi_5)}}u),
\end{align*}where we use the following notation for short: $P_{\xi, \tau,h'}(u):=$
$$ \xi_2 u+ \frac{\xi_5 \tau}{p^{\vartheta(\xi_3, \xi_4, \xi_5)}}u  + \frac{\xi_3 \xi_5}{2\xi_4} u^2 + \frac{1}{6\xi_4 \xi_5}( \xi_5^3 u^3+3\xi_5  p^{2\vartheta(\xi_3, \xi_4, \xi_5)}(h')^2u )+p^{\vartheta(\xi_3, \xi_4, \xi_5)}(\frac{\xi_5}{2\xi_4}u^2+ \frac{\xi_3}{\xi_4} u )(h'),$$so that $\frac{d}{du}P_{\xi, \tau,h'}(u)=$ $$\xi_2 + \frac{\xi_5 \tau}{p^{\vartheta(\xi_3, \xi_4, \xi_5)}}  + \frac{\xi_3 \xi_5}{2\xi_4} 2u + \frac{1}{6\xi_4 \xi_5}( \xi_5^3 3u^2+3\xi_5  p^{2\vartheta(\xi_3, \xi_4, \xi_5)}(h')^2 )+p^{\vartheta(\xi_3, \xi_4, \xi_5)}(\frac{\xi_5}{2\xi_4}2u+ \frac{\xi_3}{\xi_4}  )(h'),$$and $$\frac{d}{du}P_{\xi, \tau,h'}(u)|_{u=0}= \xi_2 + \frac{\xi_5 \tau}{p^{\vartheta(\xi_3, \xi_4, \xi_5)}}   + \frac{1}{6\xi_4 \xi_5}( 3\xi_5  p^{2\vartheta(\xi_3, \xi_4, \xi_5)}(h')^2 )+p^{\vartheta(\xi_3, \xi_4, \xi_5)} \frac{\xi_3}{\xi_4} (h')$$
The associated associated symbol $\sigma_{\partial_{X_2}^{\alpha}} (\xi) = \partial_{X_2}^{\alpha} \pi_{\xi}|_{x=0}$ can be obtained as an application of \begin{align*}
    &\sigma_{\partial_{X_2}^{\alpha}} (\xi)_{hh'} = (\partial_u^\alpha \psi_\xi^2 (u)|_{u=0})_{hh'} \\&= \sum_{1 \leq |\tau|_p \leq  \| (\xi_3, \xi_4,\xi_5) \|_p }  (M_{h'} \mathbf{C}_f)_{\tau} \partial_u^{\alpha}(e^{ 2 \pi i \{ P_{\xi, \tau,h'}(u) \}_p}) |_{u=0} \\ &= \sum_{1 \leq |\tau|_p \leq  \| (\xi_3, \xi_4) \|_p }\Big( |\frac{d}{du}P_{\xi, \tau,h'}(u)|_{u=0}|_p^\alpha -\frac{1 - p^{-1}}{1 - p^{-(\alpha + 1)}}\Big)  (M_{h'} \mathbf{C}_f)_{\tau}.
\end{align*}This proves how, when $\xi \in A_2$, the spectrum of $\sigma_{\mathcal{L}_{sub}^\alpha} (\xi)$ is composed by the numbers \begin{align*}
    |\xi_1 +&  \frac{\xi_4}{p^{\vartheta(\xi_3, \xi_4, \xi_5)}}\tau |_p^\alpha  \\&+|\xi_2 + \frac{\xi_5 \tau}{p^{\vartheta(\xi_3, \xi_4, \xi_5)}}   + \frac{1}{6\xi_4 \xi_5}( 3\xi_5  p^{2\vartheta(\xi_3, \xi_4, \xi_5)}(h')^2 )+p^{\vartheta(\xi_3, \xi_4, \xi_5)} \frac{\xi_3}{\xi_4} (h')|_p^\alpha -2\frac{1 - p^{-1}}{1 - p^{-(\alpha + 1)}},
\end{align*}where $$1 \leq  |\tau|_p \leq \| (\xi_3, \xi_4 ,\xi_5)\|_p, \quad h' \in \Z_p / p^{-\vartheta(\xi_3, \xi_4,\xi_5 )}\Z_p.$$In this way how $$\| \mathscr{L}_{sub}^\alpha |_{\mathcal{V}_\xi^{h'}} \|_{op} = \max_{\tau,  h'}\lambda_{\xi, h', \tau} (\mathscr{L}_{sub}^\alpha) \lesssim  \| \xi\|_p^\alpha  .$$In the same way \begin{align*}
    \| \mathscr{L}_{sub}^\alpha |_{\mathcal{V}_\xi^{h'}} \|_{inf} &= \min_{\tau, h'}\lambda_{\xi, h', \tau} (\mathscr{L}_{sub}^\alpha) \gtrsim | \xi_1 |_p^\alpha + |\xi_2 |_p^\alpha.
\end{align*}
This implies that in each representation space we have the estimates $$| \xi_1 |_p^\alpha + |\xi_2 |_p^\alpha \lesssim \|\mathscr{L}_{sub}^\alpha |_{\mathcal{V}_\xi} \|_{inf} \leq \|\mathscr{L}_{sub}^\alpha |_{\mathcal{V}_\xi} \|_{op} \lesssim \| \xi \|_p^\alpha.,$$but we know that $$\|\mathscr{L}_{sub}^\alpha |_{\mathcal{V}_\xi} \|_{inf} = \|\sigma_{\mathscr{L}_{sub}^\alpha}(\xi)\|_{inf}, \,\,\,\|\mathscr{L}_{sub}^\alpha |_{\mathcal{V}_\xi} \|_{op} = \|\sigma_{\mathscr{L}_{sub}^\alpha}(\xi)\|_{op},$$proving how $\mathscr{L}_{sub}^\alpha$ provides an example of globally hypoelliptic operator.

This concludes the proof.

\end{proof}

\section{The group $\mathbb{G}^{5,5}$}
Let $p>3$ be a prime number. In this section we consider the group $\mathbb{G}^{5,5}(\Z_p)$, or just $\mathbb{G}^{5,5}$ for simplicity, defined here as $\Z_p^{5}$ together with the non-commutative operation\begin{align*}
    &\mathbf{x} \star \textbf{y} :=\\& (x_1 + y_1, x_2 + y_2, x_3 + y_3 + x_1 y_2, x_4 + y_4+ \frac{1}{2}x_1^2 y_2 +x_1 y_3 , x_5 + y_5 + \frac{1}{6} x_1^3y_2 + \frac{1}{2} x_1^2 y_3 +  x_1y_4 ),
\end{align*}and associated inverse element $$\mathbf{x}^{-1} = (-x_1, -x_2, -x_3 + x_1x_2, -x_4 + x_1 x_3- \frac{1}{2} x_2 x_1^2, -x_5 + x_1 x_4- \frac{1}{2}x_3 x_1^2 + \frac{1}{6}x_2 x_1^3).$$

We can identify this group with the exponential image of the filiform $\Z_p$-Lie algebra $\mathfrak{g}^{5,5}$ defined by the commutation relations$$[X_1,X_2] = X_3, \quad [X_1 , X_3]= X_4, \quad [X_1 , X_4] = X_5.$$
In the following theorem, we provide an explicit description of the unitary dual of this group. 
\begin{teo}\normalfont
The unitary dual $\widehat{\mathbb{G}}^{5,5}$ of the compact group $\mathbb{G}^{5,5}$ can be identified with union of the following 2 disjoints subsets of $\widehat{\Z}_p^{5}$: $\widehat{\mathbb{G}}^{5,5}= A_1 \cup A_2 $, where $$A_1 =  \{\xi \in \widehat{\Z}_p^{5}  \, : \, |\xi_5|_p=1 \leq |\xi_4|_p <|\xi_3|_p \, \wedge \,  (\xi_1 , \xi_2 , \xi_3) \in \widehat{\mathbb{H}}_1 , \, \text{or}  \,  |\xi_3|_p = 1 \,  \wedge  \,  \xi_1 \in \Q_p / p^{\vartheta(\xi_4)} \Z_p \},$$
    $$A_2:= \{\xi \in \widehat{\Z}_p^5 \, : 1 <|\xi_5 |_p , \,\,\, (\xi_1 ,\xi_2) \in \Q_p^2/ p^{\vartheta(\xi_3, \xi_4, \xi_5)} \Z_p^2\}.$$Moreover, each unitary irreducible representation can be realized in the finite dimensional Hilbert space \[\mathcal{H}_\xi :=
    \mathrm{span}_\C \{  \varphi_h := \|(\xi_3, \xi_4, \xi_5)\|_p^{1/2} \1_{h + p^{-\vartheta(\xi_3, \xi_4, \xi_5)} \Z_p  } \,\, : \, h \in I_\xi:= \Z_p / p^{-\vartheta(\xi_3, \xi_4, \xi_5)} \Z_p  \} \subset L^2 (\Z_p),\]where $d_\xi := \mathrm{dim}_\C (\mathcal{H}_\xi) = \| (\xi_3 , \xi_4 ,\xi_5 ) \|_p $, acting according to the formula \[\pi_\xi(\mathbf{x}) \varphi (u) := e^{2 \pi i \{ \xi \cdot \mathbf{x} + \xi_3 x_2 u  + \xi_4(  x_2 \frac{u^2}{2} + ux_3 )  + \xi_5  ( x_4 u + x_3\frac{u^2}{2} + x_2 \frac{u^3}{6})\}_p} \varphi (u + x_1) .\]With this realization and the natural choice of basis for $\mathcal{H}_\xi$, the associated matrix coefficients are going to be  \[(\pi_\xi)_{hh'} (\mathbf{x}) =    e^{2 \pi i \{\xi \cdot \mathbf{x} +  \xi_3 x_2 (h')  + \xi_4(  x_2 \frac{(h')^2}{2} + (h')x_3 )  + \xi_5  ( x_4 (h') + x_3\frac{(h')^2}{2} + x_2 \frac{(h')^3}{6}) \}_p} \1_{h - h' + p^{-\vartheta(\xi_3, \xi_4, \xi_5)}\Z_p } (x_1),\]and the associated characters are \[\chi_{\pi_\xi} (\mathbf{x})=\| (\xi_3 , \xi_4, \xi_5)\|_p  e^{2 \pi i \{\xi \cdot \mathbf{x} \}_p} \1_{p^{-\vartheta(\xi_3, \xi_4, \xi_5)} \Z_p } (x_1) \int_{\Z_p }e^{2 \pi i \{ x_2(\xi_3u + \xi_4 \frac{u^2}{2} + \xi_5 \frac{u^3}{6}) + x_3(\xi_4 u + \xi_5 \frac{u^2}{2})+ x_4(\xi_5 u) \}_p} du .\] Sometimes we will use the notation $$\mathcal{V}_\xi := \mathrm{span}_\C \{ (\pi_\xi)_{hh'}  \, : \, h,h' \in I_\xi \}.$$
\end{teo}

\begin{proof}
Once again, we can consider $\mathbb{G}^{5,5}$ as an extension of the group $\mathcal{B}_4$, so the representation theory of $\mathbb{G}^{5,5}$ should also include that of $\mathcal{B}_4$. The unitary irreducible representations of $\mathcal{B}_4$ can be realized in the following finite dimensional sub-space $\mathcal{H}_\xi$ of $L^2(\Z_p)$:
\[ \mathcal{H}_\xi :=  \mathrm{span}_\C \{ \| (\xi_3 , \xi_4)\|_p^{1/2} \1_{h + p^{-\vartheta(\xi_3, \xi_4)} \Z_p } \,\, : \, h \in \Z_p / p^{-\vartheta(\xi_3, \xi_4)} \Z_p \},    \]
where $d_\xi:= \mathrm{dim}_\C (\mathcal{H}_\xi) = \max\{ |\xi_3|_p , |\xi_4|_p \}$, and the representation acts on functions $\varphi \in \mathcal{H}_{\xi}$ according to the formula 
\[\pi_{\xi}(\mathbf{x}) \varphi (u) := 
    e^{2 \pi i \{\xi_1 x_1 + \xi_2 x_2 + \xi_3 (x_3 +  u x_2)  + \xi_4 (x_4 +  u x_3 + \frac{u^2}{2} x_2) \}_p} \varphi (u + x_1)  . \] On the other hand, if a representation is not trivial on the variable $x_5$, then we must have $|  \xi_5 |_p>1$, and we can realize these in the finite dimensional Hilbert spaces  $$\mathcal{H}_\xi := \mathrm{span}_\C \{\| (\xi_3, \xi_4, \xi_5)\|_p^{1/2}\1_{h + p^{-\vartheta(\xi_3, \xi_4, \xi_5)} \Z_p}(u) \, : \, h \in I_\xi\}, \, \, d_\xi :=\mathrm{dim}_\C (\mathcal{H}_\xi) = \| (\xi_3 , \xi_4, \xi_5)\|_p ,$$acting according to the formula \[\pi_\xi(\mathbf{x}) \varphi (u) := e^{2 \pi i \{ \xi \cdot \mathbf{x} + \xi_3 x_2 u  + \xi_4(  x_2 \frac{u^2}{2} + ux_3 )  + \xi_5  ( x_4 u + x_3\frac{u^2}{2} + x_2 \frac{u^3}{6})\}_p} \varphi (u + x_1) .\]Let us write $$ P_\xi(\mathbf{x} , u) :=  \xi \cdot \mathbf{x} +  \xi_3 x_2 u  + \xi_4(  x_2 \frac{u^2}{2} + ux_3 )  + \xi_5  ( x_4 u + x_3\frac{u^2}{2} + x_2 \frac{u^3}{6}).$$
Then we can see how $\pi_\xi$ is indeed a representation since:
\begin{align*}
    &P_\xi ( \textbf{y} , u + x_1) = \\ & \xi \cdot \textbf{y} + \xi_3 y_2 (u + x_1)  + \xi_4(  y_2 \frac{(u + x_1)^2}{2} + (u+x_1)y_3 ) + \xi_5  ( y_4 (u + x_1) + y_3 \frac{(u +x_1)^2}{2} + y_2 \frac{(u + x_1)^3}{6} ) \\ &=\xi_1 y_1 + \xi_2 y_2 + \xi_3(y_3 x_1 y_2) + \xi_4( (y_4 + x_1 \frac{y_2^2}{2} + x_1 y_3)  + y_2 \frac{u^2}{2} + (y_3 + x_1 y_2 )u ) \\ & \, \, \, \, \, +\xi_5( (y_5 + \frac{x_1^3}{6}y_3+ \frac{x_1^2}{2}y_3 + x_1y_4) + (y_3 + x_1 y_2 )\frac{u^2}{2} + (y_4+ \frac{x_1^2}{2}y_2+ x_1 y_3)u+ y_2\frac{u^3}{6}),
\end{align*} and thus it clearly holds $$P_\xi(\mathbf{x},u) + P_\xi(\textbf{y}, u + x_1 ) = P_\xi(\mathbf{x} \star \textbf{y}, u).$$From this we obtain
\begin{align*}
    &\pi_{\xi}(\mathbf{x})\pi_{\xi}(\textbf{y})\varphi (u) =\pi_{\xi}(\mathbf{x})(e^{2 \pi i  \{ P_\xi(\textbf{y}, u)\}_p}\varphi (u + y_1 ) )\\ &= e^{2 \pi i \{ P_\xi (\mathbf{x} , u) + P_\xi (\textbf{y},u + x_1)\}_p}  \varphi (u + x_1 + y_1)\\&=e^{2 \pi i \{ P(\mathbf{x} \star \textbf{y},u) \}_p} \varphi (u + (\mathbf{x} \star \textbf{y})_1)\\ &= \pi_\xi (\mathbf{x} \star \textbf{y}) \varphi(u).
\end{align*}It is not hard to see how this representation defines an unitary operator on $\mathcal{H}_\xi$. By using the natural choice of basis for this space, we can compute \begin{align*}
    (\pi_{\xi } (\mathbf{x}))_{h h'}&= (\pi_{\xi} (\mathbf{x}) \varphi_h , \varphi_{h'})_{L^2 (\Z_p^2)} \\ &= \| (\xi_3 , \xi_4, \xi_5)\|_p\int_{\Z_p} e^{2 \pi i \{ P_\xi (x , u) \}_p} \1_h (u +x_1) \1_{h'} (u) du \\ &= \| (\xi_3 , \xi_4, \xi_5)\|_p  e^{2 \pi i \{\xi \cdot \mathbf{x} \}_p} \1_{h - h' + p^{-\vartheta(\xi_3, \xi_4, \xi_5)}\Z_p } (x_1)  \\ & \, \, \quad \times \int_{h' + p^{-\vartheta(\xi_3, \xi_4, \xi_5)}\Z_p}e^{2 \pi i \{  \xi_3 x_2 u  + \xi_4(  x_2 \frac{u^2}{2} + ux_3 )  + \xi_5  ( x_4 u + x_3\frac{u^2}{2} + x_2 \frac{u^3}{6}) \}_p} du \\ &= e^{2 \pi i \{\xi \cdot \mathbf{x} +  \xi_3 x_2 (h')  + \xi_4(  x_2 \frac{(h')^2}{2} + (h')x_3 )  + \xi_5  ( x_4 (h') + x_3\frac{(h')^2}{2} + x_2 \frac{(h')^3}{6}) \}_p} \1_{h' - h + p^{-\vartheta(\xi_3, \xi_4, \xi_5)}\Z_p } (x_1). 
\end{align*}From this it readily follows that \begin{align*}
    \chi_{\pi_\xi}(\mathbf{x}) &= \sum_{h \in I_\xi }\| (\xi_3 , \xi_4, \xi_5)\|_p  e^{2 \pi i \{\xi \cdot \mathbf{x} \}_p} \1_{h' - h +K_\xi } (x_1) \int_{h' + K_\xi}e^{2 \pi i \{  \xi_3 x_2 u  + \xi_4(  x_2 \frac{u^2}{2} + ux_3 )  + \xi_5  ( x_4 u + x_3\frac{u^2}{2} + x_2 \frac{u^3}{6}) \}_p} du \\&=\| (\xi_3 , \xi_4, \xi_5)\|_p  e^{2 \pi i \{\xi \cdot \mathbf{x} \}_p} \1_{p^{-\vartheta(\xi_3, \xi_4, \xi_5)} \Z_p } (x_1) \int_{\Z_p }e^{2 \pi i \{ x_2(\xi_3u + \xi_4 \frac{u^2}{2} + \xi_5 \frac{u^3}{6}) + x_3(\xi_4 u + \xi_5 \frac{u^2}{2})+ x_4(\xi_5 u) \}_p} du . 
\end{align*}We can use the above function with two purposes: first to check how these representations are indeed irreducible, and second, we can also count the number of different non-equivalent classes of unitary irreducible representations, by counting the number of different functions among these characters. Let us start by checking how these representations are indeed irreducible. With this end in mind, let us define $$f_\xi(x_2 , x_3 , x_4):=\int_{\Z_p }e^{2 \pi i \{ x_2( \xi_2 + \xi_3u + \xi_4 \frac{u^2}{2} + \xi_5 \frac{u^3}{6}) + x_3(\xi_3 + \xi_4 u + \xi_5 \frac{u^2}{2})+ x_4(\xi_4 + \xi_5 u) \}_p} du.$$ 
Then, computing its Fourier transform as an element of $L^2 (\Z_p^3)$ is given by 
\begin{align*}
    \mathcal{F}_{\Z_p^3} &[f](\alpha ,\beta , \gamma) = \int_{\Z_p^3}\int_{\Z_p}e^{2 \pi i \{ x_2( \xi_2 + \xi_3u + \xi_4 \frac{u^2}{2} + \xi_5 \frac{u^3}{6} - \alpha) + x_3(\xi_3 + \xi_4 u + \xi_5 \frac{u^2}{2}- \beta)+ x_4(\xi_4 + \xi_5 u - \gamma) \}_p} du  dx_2 dx_3 dx_4 \\ &= \int_{\Z_p} \1_{\Z_p}( \xi_2 + \xi_3u + \xi_4 \frac{u^2}{2} + \xi_5 \frac{u^3}{6} - \alpha)\1_{\Z_p}(\xi_3 +  \xi_4 u + \xi_5 \frac{u^2}{2}- \beta)\1_{\Z_p}(\xi_4 +  \xi_5 u - \gamma) du.
\end{align*}
Now let $A_\xi (\alpha, \beta , \gamma)$ be the intersection of the sets $$ \{ u: \, \xi_2 + \xi_3u + \xi_4 \frac{u^2}{2} + \xi_5 \frac{u^3}{6} - \alpha \in \Z_p \}, \,\{ u: \, \xi_3 +  \xi_4 u + \xi_5 \frac{u^2}{2}- \beta \in \Z_p\}, \,\{ u : \, \xi_4 +  \xi_5 u - \gamma \in \Z_p\} $$We know that 
\begin{itemize}
    \item $h_1(u):=\xi_2 + \xi_3u + \xi_4 \frac{u^2}{2} + \xi_5 \frac{u^3}{6} - \alpha$ satisfies $h_1 (u+v) = h_1 (u)$ for $|v|_p \leq \|(\xi_3, \xi_4 , \xi_5) \|_p^{-1}.$
    \item $h_2(u):=\xi_3 +  \xi_4 u + \xi_5 \frac{u^2}{2}- \beta$ satisfies $h_2 (u+v) = h_2 (u)$ for $|v|_p \leq \|( \xi_4 , \xi_5) \|_p^{-1}.$
    \item $h_3(u):=\xi_4 +  \xi_5 u - \gamma$ satisfies $h_3 (u+v) = h_3 (u)$ for $|v|_p \leq | \xi_5|_p^{-1}.$
\end{itemize}
In this way, we can think on the intersection $A_\xi(\alpha, \beta, \gamma)$ as a ball of radius $\|(\xi_3, \xi_4 , \xi_5) \|_p^{-1}$ whenever the intersection is non-empty. So, if we define $\mathbf{N}_\xi (\alpha , \beta, \gamma)$ as the  either $1$ or $0$, depending on whether the intersection is non-empty, we have $N_\xi = N_\xi^2$ and \begin{align*}
    \| f_\xi \|^2_{L^2} &= \sum_{\alpha, \beta, \gamma} |\mathcal{F}_{\Z_p^3} [f_\xi](\alpha ,\beta , \gamma)|^2 \\ &= \sum_{\alpha, \beta, \gamma} N_\xi(\alpha , \beta, \gamma)^2 \| (\xi_3, \xi_4, \xi_5)\|_p^{-2} \\&= \| (\xi_3, \xi_4, \xi_5)\|_p^{-1}\sum_{\alpha, \beta, \gamma} N_\xi(\alpha , \beta, \gamma) \| (\xi_3, \xi_4, \xi_5)\|_p^{-1}=\| (\xi_3, \xi_4, \xi_5)\|_p^{-1},
\end{align*}where the last equality holds because \begin{align*}
    &\sum_{\alpha, \beta, \gamma} N_\xi(\alpha , \beta, \gamma) \| (\xi_3, \xi_4, \xi_5)\|_p^{-1}= \\ &\sum_{\alpha, \beta, \gamma} \int_{\Z_p} \1_{\Z_p}( \xi_2 + \xi_3u + \xi_4 \frac{u^2}{2} + \xi_5 \frac{u^3}{6} - \alpha)\1_{\Z_p}(\xi_3 +  \xi_4 u + \xi_5 \frac{u^2}{2}- \beta)\1_{\Z_p}(\xi_4 +  \xi_5 u - \gamma) du = 1.
\end{align*}

Next we need to count our representations. Here we consider again cases. First, assume $|\xi_2|_p \leq \| (\xi_3, \xi_4 , \xi_5) \|_p$, and take $\theta \in \Z_p$ with $|\theta|_p = \| (\xi_3, \xi_4 , \xi_5) \|_p^{-1}$. Notice how \begin{align*}
    &\mathcal{F}_{\Z_p^3} [f_\xi](\alpha ,\beta , \gamma)\\&  = \sum_{a \in I_\xi } \int_{a + K_\xi} \1_{\Z_p}( \xi_2 + \xi_3u + \xi_4 \frac{u^2}{2} + \xi_5 \frac{u^3}{6} - \alpha)\1_{\Z_p}(\xi_3 +  \xi_4 u + \xi_5 \frac{u^2}{2}- \beta)\1_{\Z_p}(\xi_4 +  \xi_5 u - \gamma) du \\&=\sum_{a \in I_\xi }  \| (\xi_3, \xi_4, \xi_5)\|_p^{-1} \1_{\Z_p}( \xi_2 + \xi_3a + \xi_4 \frac{a^2}{2} + \xi_5 \frac{a^3}{6} - \alpha)\1_{\Z_p}(\xi_3 +  \xi_4 a + \xi_5 \frac{u^2}{2}- \beta)\1_{\Z_p}(\xi_4 +  \xi_5 a - \gamma) du\\ &=\sum_{a \in I_\xi }  \| (\xi_3, \xi_4, \xi_5)\|_p^{-1} \1_{d_\xi \Z_p}( \theta(\xi_2 + \xi_3a + \xi_4 \frac{a^2}{2} + \xi_5 \frac{a^3}{6} - \alpha))\\ & \quad \quad \quad \quad \times \1_{d_\xi \Z_p}(\theta(\xi_3 +  \xi_4 a + \xi_5 \frac{u^2}{2}- \beta))\1_{d_\xi \Z_p}(\theta(\xi_4 +  \xi_5 a - \gamma)) du\\ &= \textbf{N}_\xi(\alpha , \beta, \gamma)  \| (\xi_3 , \xi_4, \xi_5 )\|_p^{-1},
\end{align*}where $\textbf{N}_\xi(\alpha , \beta, \gamma)$ equals the number of solutions in $I_\xi$ to the system of equations 
\begin{itemize}
    \item $Q_1(u):= \theta(\xi_2 + \xi_3u + \xi_4 \frac{u^2}{2} + \xi_5 \frac{u^3}{6} - \alpha) \equiv 0 \, \, \mathrm{mod} \,  \|(\xi_3, \xi_4 , \xi_5) \|_p.$
    \item $Q_2 (u):=\theta(\xi_3 +  \xi_4 u + \xi_5 \frac{u^2}{2}- \beta)\equiv 0 \, \, \mathrm{mod} \,  \|(\xi_3, \xi_4 , \xi_5) \|_p. $
    \item $Q_3 (u):=\theta( \xi_4 +  \xi_5 u - \gamma) \equiv 0 \, \, \mathrm{mod} \,  \|(\xi_3, \xi_4 , \xi_5) \|_p. $
\end{itemize}Using Hensel's lemma, we can see how for every fixed $\gamma$ there is always at most one solution $u_\gamma$ for a single pair $(\alpha_\gamma, \beta_\gamma)$, so $\textbf{N}_\xi =1$ and we just proved how the Fourier transform of $f$ is independent of $\xi_2 \in \widehat{\Z}_p$ if $|\xi_2|_p$. In the case $|\xi_2|_p> \| (\xi_3, \xi_4, \xi_5)\|_p$, the only way these three equations have a solution is that $|\alpha - \xi_2 |_p \leq \| (\xi_3 , \xi_4 , \xi_5)\|_p$. This shows how this time $f_\xi$ does depend on $\xi_2$, and we only need to consider $\xi_2 \in \Q_p / p^{\vartheta(\xi_3, \xi_4, \xi_5)}\Z_p$.

Finally, to check how these are all the desired representations, just notice how \begin{align*}
    \sum_{\{ [\pi_\xi] \, : \| \xi\|_p \leq p^n\}} d_\xi^2 & = \sum_{\{ [\pi_\xi] \, : \| \xi\|_p \leq p^n, \, \, |\xi_5|_p = 1\}} d_\xi^2 + \sum_{\{ [\pi_\xi] \, : \| \xi\|_p \leq p^n, \, \, |\xi_5|_p > 1\}} d_\xi^2 \\ &= p^{4n} + \sum_{1 <|\xi_5|_p \leq p^n} \sum_{1 \leq \| (\xi_3, \xi_4)\|_p \leq p^n} \sum_{(\xi_1 , \xi_2) \in \Q_p^2 / p^{\vartheta(\xi_3, \xi_4, \xi_5)} \Z_p^2} \|(\xi_3, \xi_4 , \xi_5) \|_p^2 \\ &=p^{4n} + p^{2n}\sum_{1 <|\xi_5|_p \leq p^n} \sum_{1 \leq \| (\xi_3, \xi_4)\|_p \leq p^n} 1 \\ &= p^{4n} + p^{4n}(p^n - 1) = p^{5n}= |\mathbb{G}^{5,5}/\mathbb{G}^{5,5}(p^n \Z_p)|.
\end{align*}
This concludes the proof.
\end{proof}
With the above theorem, functions $f \in L^2 (\mathbb{G}^{5,5})$ have the Fourier series representation \begin{align*}
    f(\mathbf{x}) &= \sum_{\, \|(\xi_3 , \xi_4, \xi_5) \|_p = 1 } \widehat{f}(\xi) e^{2 \pi i \{ \xi \cdot \mathbf{x} \}_p} + \sum_{|\xi_5|_p = 1 \leq |\xi_4|_p < |\xi_3|_p } \sum_{(\xi_1 , \xi_2) \in \Q_p^2 / p^{\vartheta(\xi_3)}\Z_p^2} |\xi_3|_p Tr[\pi_\xi (\mathbf{x}) \widehat{f}(\xi)] \\ & \quad \quad +\sum_{\,|\xi_4|_p > |\xi_3|_p=1=|\xi_5|_p } \sum_{\xi_2 \in \widehat{\Z}_p} \sum_{\xi_1 \in \Q_p/ p^{\vartheta(\xi_4)}\Z_p}|\xi_4|_p Tr[\pi_\xi (x) \widehat{f}(\xi)] \\ &+\sum_{1< |\xi_5|_p}\sum_{(\xi_3,\xi_4) \in \widehat{\Z}_p^2} \sum_{(\xi_1,\xi_2)\in \Q_p^2/ p^{\vartheta(\xi_3, \xi_4, \xi_5)}\Z_p^2} \|(\xi_3 , \xi_4 , \xi_5)  \|_p Tr[\pi_{\xi} (\mathbf{x}) \widehat{f}(\xi)],
\end{align*}where the Fourier transform $\widehat{f}(\xi)$ is the linear operator defined as \begin{align*}
    \widehat{f}&(\xi) \varphi (u):= \int_{\mathbb{G}^{5,5}} f(\mathbf{x}) \pi_\xi^* (\mathbf{x}) \varphi(u) d \mathbf{x} \\ &= \int_{\Z_p^5} f(\mathbf{x}) e^{2 \pi i \{ \xi \cdot \mathbf{x} + \xi_3 x_2 (u - x_1)  + \xi_4(  x_2 \frac{(u - x_1)^2}{2} + (u - x_1)x_3 )  + \xi_5  ( x_4 (u - x_1) + x_3\frac{(u - x_1)^2}{2} + x_2 \frac{(u - x_1)^3}{6})\}_p} \varphi (u - x_1)d\mathbf{x}, \quad \varphi \in \mathcal{H}_\xi.
\end{align*}Alternatively, by using our expressions for the matrix coefficients we can see this operator as the matrix $\widehat{f}(\xi) \in \C^{\|(\xi_3,\xi_4, \xi_5)\|_p \times \|(\xi_3,\xi_4, \xi_5)\|_p}$ defined by the expression \begin{align*}
    \widehat{f}(\xi)_{h h'}&:= \int_{\mathbb{G}^{5,5}} f(\mathbf{x}) \pi_\xi^* (\mathbf{x})_{hh'} d \mathbf{x} \\ &= \int_{\Z_p^5} f(\mathbf{x}) e^{-2 \pi i \{\xi \cdot \mathbf{x} +  \xi_3 x_2 (h)  + \xi_4(  x_2 \frac{(h)^2}{2} + (h)x_3 )  + \xi_5  ( x_4 (h) + x_3\frac{(h)^2}{2} + x_2 \frac{(h)^3}{6}) \}_p} \1_{h' - h + p^{-\vartheta(\xi_3, \xi_4, \xi_5)}\Z_p } (x_1) d\mathbf{x}, 
\end{align*}where $\varphi \in \mathcal{H}_\xi$, which in terms of the $\Z_p^5$-Fourier transform is expressed as \begin{align*}
    \widehat{f}&(\xi)_{h h'} = \mathcal{F}_{\Z_p^5} [\1_{h' - h + p^{-\vartheta(\xi_3, \xi_4, \xi_5)} \Z_p  } (x_1)f] (\xi_1 ,\xi_2 + h_1 + \xi_5 (\frac{h_1^2}{2} + h_1), \xi_3 + \xi_5 h_2, \xi_4 , \xi_5 ) \\ &=\mathcal{F}_{\Z_p^5} [\1_{h' - h +  p^{-\vartheta(\xi_3, \xi_4, \xi_5)} \Z_p } ] *_{\widehat{\Z}_p^5} \mathcal{F}_{\Z_p^5}[f] (\xi_1 ,\xi_2 + \xi_3 h + \xi_4 \frac{h^2}{2} + \xi_5 \frac{h^3}{6}, \xi_3 + \xi_4 h + \xi_5 \frac{h^2}{2}, \xi_4 , \xi_5 ).
\end{align*}

Given any collection $\textbf{W}:=\{W_1 , W_2, W_3$\}  with $$W_i = w_i^1 X_1 + w_i^2 X_2 + w_i^3 X_3,$$which spans $\mathfrak{g}^{5,5} / [\mathfrak{g}^{5,5}, \mathfrak{g}^{5,5}] $, we can define its associated Vladimirov sub-Laplacian as $$\mathcal{L}_{sub, \textbf{W}}^\alpha f (x) : = (\partial^\alpha_{W_1} + \partial^\alpha_{W_2} + \partial^\alpha_{W_3}) f (x), \,\,\, f \in \mathcal{D}(\mathbb{G}^{5,5}).$$For $\alpha>0$ the following is the spectral theorem corresponding to this operator .
\begin{teo}
The Vladimirov sub-Laplacian $\mathcal{L}_{sub, \textbf{W}}^\alpha$ associated with the collection $\textbf{W}$ is a globally hypoelliptic operator, which is invertible in the space of mean zero functions. Moreover, the space $L^2(\mathbb{G}^{5,5})$ can be written as the direct sum $$L^2(\mathbb{G}^{5,5}) = \overline{\bigoplus_{\xi \in \widehat{\mathbb{G}}^{5,5}} \bigoplus_{h' \in  I_\xi} \mathcal{V}_{\xi}^{h'}}, \, \,\, \mathcal{V}_{\xi} = \bigoplus_{h' \in  I_\xi:=\Z_p/ p^{-\vartheta(\xi_3, \xi_4, \xi_5)} \Z_p} \mathcal{V}_{\xi}^{h'}, $$where each finite-dimensional sub-space$$\mathcal{V}_{\xi}^{h'}:= \mathrm{span}_\C \{ (\pi_{\xi})_{hh'} \, : \, h \in I_\xi \},$$is an invariant sub-space of $\mathcal{L}_{sub, \textbf{W}}^\alpha$, and its spectrum restricted to $\mathcal{V}_{\xi }^{h'}$ is given by  \[Spec(\mathcal{L}_{sub, \textbf{W}}^\alpha|_{\mathcal{V}_{\xi}^{h}})=\sum_{i=1}^3| (\xi_1 + \tau, \xi_2  + \xi_3h'+ \xi_4\frac{1}{2}(h')^2+\xi_5 \frac{(h')^3}{6} , \xi_3 + \xi_4h' +\xi_5 \frac{(h')^2}{2}) \cdot W_i|_p^\alpha - 3\frac{1 - p^{-1}}{1 - p^{-(\alpha + 1)}},\]and the corresponding eigenfunctions are given by $$\mathscr{e}_{\xi , h' , \tau} (\mathbf{x}) := e^{2 \pi i \{\xi \cdot \mathbf{x} +  \xi_3 x_2 (h')  + \xi_4(  x_2 \frac{(h')^2}{2} + (h')x_3 )  + \xi_5  ( x_4 (h') + x_3\frac{(h')^2}{2} + x_2 \frac{(h')^3}{6}) + \tau x_1\}_p} ,$$where $1 \leq | \tau |_p \leq \|(\xi_3, \xi_4, \xi_5)\|_p.$  
\end{teo}

\section{The group $\mathbb{G}^{5,6}$}
Let $p>3$ be a prime number. In this final section we consider the group $\mathbb{G}^{5,6}(\Z_p)$, or just $\mathbb{G}^{5,6}$ for simplicity, defined here as $\Z_p^{5}$ together with the non-commutative operation\begin{align*}
    &\mathbf{x} \star \textbf{y} := \\& (x_1 + y_1)X_1 + (x_2 + y_2)X_2 + ( x_3 + y_3 + x_1 y_2)X_3 \\ &+ ( x_4 + y_4+ \frac{1}{2}x_1^2 y_2 +x_1 y_3)X_4 + ( x_5 + y_5 + \frac{1}{6} x_1^3y_3 + x_1y_4 + \frac{1}{2} x_1 y_2^2 +  x_2y_3  + x_1 x_2 y_2)X_5,
\end{align*}and associated inverse element $$\mathbf{x}^{-1} = (-x_1, -x_2, -x_3 + x_1 x_2, -x_4 + x_1 x_3 -\frac{1}{2}x_2 x_1^2, - x_5 +x_1 x_4 + x_2 x_3 -\frac{1}{2} x_3 x_1^2 - \frac{1}{2}x_2^2 x_1 + \frac{1}{6} x_2 x_1^3).$$We can identify this group with the exponential image of the $\Z_p$-Lie algebra $\mathfrak{g}^{5,6}$ defined by the commutation relations$$[X_1,X_2] = X_3, \quad [X_1 , X_3]= X_4, \quad [X_1 , X_4] = X_5, \quad [X_2 , X_3] = X_5.$$
In the following theorem, we provide an explicit description of the unitary dual of this group, which is actually the last one in the list of $5$-dimensional groups which we intend to consider in this work. But before that, we wil need to make use of the following lemma: 

\begin{lema}\label{lemaauxG56}
    Let $(x_2, x_3, x_4) \in \Z_p^3$, and take $\xi_3 , \xi_4, \xi_5 \in \widehat{\Z}_p$. Then $$\int_{\Z_p^3} \Big| \int_{ \Z_p}e^{2 \pi i \{(\xi_3   x_2 +  \xi_4 x_3 + \xi_5 x_4)u +  \xi_4 x_2 \frac{u^2}{2} \}_p} du \Big|^2 dx_2 dx_3 dx_4  = \|(\xi_3 , \xi_4, \xi_5) \|_p^{-1}.$$
\end{lema}
\begin{proof}
Same arguments as in Lemma \ref{lemaaux}. 
\end{proof}

\begin{teo}\label{teorepG56}
The unitary dual $\widehat{\mathbb{G}}^{5,6}$ of $\mathbb{G}^{5,6}$ can be identified with the union of the following two disjoint subsets of $\widehat{\Z}_p^{5}$: $\widehat{\mathbb{G}}^{5,6}=A_1 \cup A_2$, where
$$A_1 =  \{\xi \in \widehat{\Z}_p^{5}  \, : \, |\xi_5|_p=1 \leq |\xi_4|_p <|\xi_3|_p \, \wedge \,  (\xi_1 , \xi_2 , \xi_3) \in \widehat{\mathbb{H}}_1 , \, \text{or}  \,  |\xi_3|_p = 1 \,  \wedge  \,  \xi_1 \in \Q_p / p^{\vartheta(\xi_4)} \Z_p \},$$
$$A_2 = \{\xi \in \widehat{\Z}_p^5 \, : \, |\xi_5|_p>1, \, (\xi_3, \xi_4) \in \Q_p^2/ p^{\vartheta( \xi_5)} \Z_p^2, \, \, \, (\xi_1 , \xi_2) \in  \Q_p^2 / p^{\vartheta(\xi_3, \xi_4, \xi_5)} \Z_p^2\}.$$
Moreover, if we denote $K_\xi:= p^{-\vartheta(\xi_3 , \xi_4, \xi_5)}\Z_p \times p^{-\vartheta(\xi_5)}\Z_p$ and $I_\xi:= \Z_p^2 / K_\xi$, each unitary irreducible representation can be realized in the finite dimensional Hilbert space $$\mathcal{H}_\xi := \mathrm{span}_\C \{\|(\xi_3, \xi_4, \xi_5)\|_p^{1/2}|\xi_5|_p^{1/2}\1_{h + K_\xi}(u) \, : \, h \in I_\xi\}, \, \, d_\xi :=\mathrm{dim}_\C (\mathcal{H}_\xi) = \| (\xi_3 , \xi_4, \xi_5)\|_p |\xi_5 |_p.$$If we define the polynomials \begin{align*}
    P_\xi (\mathbf{x} , u) &:=\xi \cdot \mathbf{x}+ \xi_3 x_2 u_1 +\xi_4(  x_2 \frac{u_1^2}{2} + u_1x_3 ) + \xi_5 \big(\frac{1}{2}x_2^2 x_1 + \frac{1}{6}x_2 x_1^3 +x_3 u_2 + x_1 x_2 u_2\\ & + \frac{1}{6}x_1^3 u_2 + \frac{1}{2} x_1 u_2^2 +x_4 u_1 + \frac{1}{2}x_2 x_1^2 u_1 + \frac{1}{2} x_1 x_2 u_1^2 + \frac{1}{2}x_1^3 u_1^2 + \frac{1}{2} x_1^2 u_1 u_2 + \frac{1}{2} x_1 u_2 u_1^2 \big)
\end{align*} then the representations act according to the formula \[\pi_\xi(\mathbf{x}) \varphi (u) :=
  e^{2 \pi i \{ P_\xi (\mathbf{x} , u) \}_p} \varphi (u + (x_1, x_2)).\]With this realization and the natural choice of basis for $\mathcal{H}_\xi$, the associated matrix coefficients are going to be   \[ (\pi_\xi)_{hh'} (\mathbf{x})= e^{2 \pi i \{ P_\xi (\mathbf{x} , h')\}_p} \1_{h' - h + p^{-\vartheta(\xi_3, \xi_4, \xi_5)} \Z_p \times p^{-\vartheta( \xi_5)} \Z_p } (x_1,x_2), \]and the associated characters are $\chi_{\pi_\xi} (\mathbf{x}) =$ \[ \| (\xi_3 , \xi_4, \xi_5)\|_p |\xi_5|_p  e^{2 \pi i \{\xi \cdot \mathbf{x} \}_p} \1_{p^{-\vartheta(\xi_3 , \xi_4, \xi_5)}\Z_p } (x_1) \1_{p^{-\vartheta(\xi_5)}\Z_p } (x_2)\1_{p^{-\vartheta(\xi_5)}\Z_p } (x_3)\int_{\Z_p }e^{ 2 \pi i \{ \xi_4 x_2 \frac{u_1^2}{2}+ (\xi_3x_2 +  \xi_4 x_3  + \xi_5 x_4) u_1  \}_p} du_1.   \]Sometimes we will use the notation $$\mathcal{V}_\xi := \mathrm{span}_\C \{ (\pi_\xi)_{hh'} \, : \, h,h' \in I_\xi \}.$$
\end{teo}

\begin{proof}
First, we consider the representations which are trivial on $H=\textbf{exp}(\Z_p X_5)$. Since $\mathbb{G}^{5,6}/H \cong \mathcal{B}_4$ we can also consider $\mathbb{G}^{5,6}$ as an extension of $\mathcal{B}_4$, so the representation theory of $\mathbb{G}^{5,6}$ should also include that of $\mathcal{B}_4$. See again Theorem \ref{TeoRepresentationsB4} and the full calculations in \cite{velasquezrodriguez2024Engelspectrumvladimirovsublaplaciancompact}. On the other hand, if a representation is not trivial on the variable $x_5$, we must have $|  \xi_5 |_p>1$, and we can see how $\pi_\xi$ is indeed a representation since $P_\xi(\mathbf{x}, u) + P_\xi ( \textbf{y} , u + (x_1 , x_2))=P_\xi (\mathbf{x} \star \textbf{y}, u)$. This defines an unitary operator and, by using the natural choice of basis for this space, we can compute \begin{align*}
    (\pi_{\xi }& (\mathbf{x}))_{h h'}= (\pi_{\xi} (\mathbf{x}) \varphi_h , \varphi_{h'})_{L^2 (\Z_p^2)} \\ &= \| (\xi_3, \xi_4, \xi_5)\|_p |\xi_5|_p \int_{\Z_p^2} e^{2 \pi i \{ P_\xi (x , u) \}_p} \1_h (u +(x_1,x_2)) \1_{h'} (u) du \\ &= \| ( \xi_3, \xi_4, \xi_5)\|_p|\xi_5|_p  e^{2 \pi i \{\xi \cdot \mathbf{x} \}_p} \1_{h' - h + K_\xi } (x_1,x_2) \int_{h' + K_\xi}e^{2 \pi i \{ P_\xi (\mathbf{x}, u) \}_p} du \\ &=  e^{2 \pi i \{P_\xi (\mathbf{x}, h') \}_p} \1_{h' - h + p^{-\vartheta(\xi_3, \xi_4 , \xi_5)} \Z_p \times  p^{-\vartheta(\xi_5)} \Z_p } (x_1,x_2)   . 
\end{align*}From this it readily follows that \begin{align*}
    \chi_{\pi_\xi}(\mathbf{x})   &= \sum_{h \in I_\xi }\| (\xi_3, \xi_4 , \xi_5)\|_p|\xi_5|_p  e^{2 \pi i \{\xi \cdot \mathbf{x} \}_p} \1_{K_\xi } (x_1,x_2) \times \int_{h + K_\xi}e^{2 \pi i \{  P_\xi (\mathbf{x}, u) \}_p} du \\&=\| (\xi_3 , \xi_4, \xi_5)\|_p|\xi_5|_p  e^{2 \pi i \{\xi \cdot \mathbf{x} \}_p} \1_{p^{-\vartheta(\xi_3, \xi_4 , \xi_5)} \Z_p \times  p^{-\vartheta(\xi_5)} \Z_p } (x_1,x_2)  \times \int_{\Z_p^2 }e^{P_\xi (\mathbf{x}, u)} du \\&=\| (\xi_3 , \xi_4, \xi_5)\|_p|\xi_5|_p  e^{2 \pi i \{\xi \cdot \mathbf{x} \}_p} \1_{p^{-\vartheta(\xi_3, \xi_4 , \xi_5)} \Z_p \times p^{-\vartheta(\xi_5)} \Z_p } (x_1,x_2)  \times \int_{\Z_p^2 }e^{P_\xi (\mathbf{x}, u) \}_p} du . 
\end{align*}We can use the function with two purposes. First to check how these representations are indeed irreducible, and second, we can also count the number of different non-equivalent classes of unitary irreducible representations, by counting the number of different functions among these characters. Let us start by noticing how, if $1 < \| (\xi_3 , \xi_4) \| _p \leq |\xi_5|_p,$ then we get \begin{align*}
    \chi_{\pi_\xi} (\mathbf{x}) &=|\xi_5|_p^2  e^{2 \pi i \{\xi \cdot \mathbf{x} \}_p} \1_{p^{-\vartheta(\xi_5)}\Z_p^2 } (x_1,x_2) \int_{\Z_p^2 }e^{2 \pi i \{P_\xi (\mathbf{x}, u) \}_p} du  \\ & = |\xi_5|_p^2  e^{2 \pi i \{\xi \cdot \mathbf{x} \}_p} \1_{p^{-\vartheta(\xi_5)}\Z_p^2 } (x_1,x_2) \1_{p^{-\vartheta(\xi_5)}\Z_p^2 } (x_3,x_4),
\end{align*}so that all of these $\pi_\xi$ are irreducible but equivalent to the same $\pi_\xi$ with $|\xi_3|_p = |\xi_4|_p = 1$. In the case $|\xi_5|_p < \|(\xi_3 , \xi_4) \|_p$ we get \begin{align*}
    \chi_{\pi_\xi} (\mathbf{x}) &=\|(\xi_3, \xi_4) \|_p |\xi_5|_p  e^{2 \pi i \{\xi \cdot \mathbf{x} \}_p} \1_{p^{-\vartheta(\xi_3, \xi_4)}\Z_p } (x_1) \1_{p^{-\vartheta(\xi_5)}}(x_2) \int_{\Z_p^2 }e^{ 2 \pi i \{\xi_3 x_2 u_1 + \xi_4 x_2 \frac{u_1^2}{2}+ \xi_4 x_3 u_1 + \xi_5 (x_4 u_1 + x_3 u_2)  \}_p} du  \\ & = \| (\xi_3 , \xi_4)\|_p |\xi_5|_p  e^{2 \pi i \{\xi \cdot \mathbf{x} \}_p} \1_{p^{-\vartheta(\xi_3 , \xi_4)}\Z_p } (x_1) \1_{p^{-\vartheta(\xi_5)}\Z_p } (x_2)\1_{p^{-\vartheta(\xi_5)}\Z_p } (x_3)\\ &\quad \quad \times \int_{\Z_p }e^{ 2 \pi i \{ \xi_4 x_2 \frac{u_1^2}{2}+ (\xi_3x_2 +  \xi_4 x_3  + \xi_5 x_4) u_1  \}_p} du_1  .
\end{align*}

By making a standard change of variable we can see how \begin{align*}
    &\int_{\mathbb{G}^{5,4}} |\chi_{\pi_\xi} (\mathbf{x})|^2  d\mathbf{x}
    \\ & =\int_{\Z_p^5} \| (\xi_3 , \xi_4)\|_p^2|\xi_5|_p^2 \1_{p^{-\vartheta(\xi_3, \xi_4)} } (x_1,) \1_{p^{-\vartheta(\xi_5)}}(x_2) \1_{p^{-\vartheta( \xi_5)} \Z_p }(x_3) \big|  \int_{\Z_p }e^{ 2 \pi i \{ \xi_4 x_2 \frac{u_1^2}{2}+ (\xi_3x_2 +  \xi_4 x_3  + \xi_5 x_4) u_1  \}_p} du_1 \big|^2   d\mathbf{x} \\ & = \| (\xi_3 , \xi_4)\|_p^2|\xi_5|_p^2 p^{\vartheta(\xi_3, \xi_4)} p^{2 \vartheta(\xi_5)} \int_{\Z_p^5}   \big| \int_{\Z_p }e^{ 2 \pi i \{ \xi_4 x_2 \frac{u_1^2}{2}+ (\xi_3x_2 +  \xi_4 x_3  + \xi_5 x_4) u_1  \}_p} du_1   \big|^2  d\mathbf{x}  \\ &=\| (\xi_3 , \xi_4)\|_p \int_{\Z_p^5}   \big| \int_{\Z_p }e^{ 2 \pi i \{ \xi_4 x_2 \frac{u_1^2}{2}+ (\xi_3x_2 +  \xi_4 x_3  + \xi_5 x_4) u_1  \}_p} du_1  \big|^2  d\mathbf{x}   =1,
\end{align*}where the last equality follows from Lemma \ref{lemaauxG56}. In other words, $\pi_\xi$ is irreducible. Finally, it is easy to see how the list of different functions among the characters $\chi_{\pi_\xi}$ can be indexed by: 
\begin{itemize}
    \item If $|\xi_5|_p = 1$, then $\pi_\xi$ must descend to a representation of $\mathcal{B}_4$, so these are indexed by: $$A_1 = \{\xi \in \widehat{\Z}_p^5 \, : \, | \xi_5|_p=1, \, (\xi_1, \xi_2, \xi_3, \xi_4) \in \widehat{\mathcal{B}}_4\}.$$
    \item If $|\xi_5|_p >1$ then the different characters can be indexed by $$A_2:= \{\xi \in \widehat{\Z}_p^5 \, : 1 <|\xi_5 |_p , \,\,\, (\xi_3, \xi_4) \in \Q_p^2/ p^{\vartheta(\xi_5)} \Z_p^2, \, \, \, (\xi_1 , \xi_2) \in  \Q_p^2 /p^{\vartheta(\xi_3 , \xi_4, \xi_5)}\Z_p^2\}.$$
\end{itemize}
These are all the desired representations since:
\begin{align*}
    \sum_{[\pi_\xi] \in \widehat{\mathbb{G}}^{5,5}} &d_\xi^2 = \sum_{\xi \in  A_1 \, : \, \| \xi \|_p \leq p^n  }d_\xi^2 + \sum_{\xi \in A_2 \, : \, \| \xi \|_p \leq p^n  }d_\xi^2  \\ &= p^{4n} + \sum_{1<|\xi_5|_p \leq p^n} \sum_{(\xi_3, \xi_4) \in \Q_p^2 / p^{\vartheta( \xi_5)} \Z_p^2} \sum_{(\xi_1 , \xi_2) \in \Q_p^2 / p^{\vartheta(\xi_3, \xi_4, \xi_5)} \Z_p^2} (\|(\xi_3 ,\xi_4, \xi_5) \|_p |\xi_5|_p)^2 \\ &= p^{4n} + p^{4n} \sum_{1<|\xi_5|_p \leq p^n} 1 \\&=p^{5n}= |\mathbb{G}^{5,6}/\mathbb{G}^{5,6}(p^n \Z_p)|. 
\end{align*} 
This concludes the proof.
\end{proof}

According to Theorem \ref{teorepG56}, functions $f \in L^2 (\mathbb{G}^{5,6})$ have the Fourier series representation \begin{align*}
    f(\mathbf{x}) &= \sum_{ \| (\xi_3, \xi_4, \xi_5)\|_p =1 } \widehat{f}(\xi) e^{2 \pi i \{ \xi \cdot \mathbf{x}\}_p} + \sum_{|\xi_5|_p = 1, \,\, (\xi_1, \xi_2, \xi_3, \xi_4) \in \widehat{\mathcal{B}}_4} \| (\xi_3, \xi_4) \|_p Tr[ \pi_\xi (\mathbf{x}) \widehat{f}(\xi)] \\ &+ \sum_{\, |\xi_5|_p >1 } \sum_{(\xi_3 , \xi_4) \in \Q_p^2 / p^{\vartheta(\xi_5)}\Z_p^{-1}} \sum_{(\xi_1, \xi_2) \in \Q_p^2 / p^{ \vartheta(\xi_3, \xi_4 , \xi_5)} \Z_p^2} \| (\xi_3, \xi_4, \xi_5) \|_p |\xi_5|_p Tr[ \pi_\xi (\mathbf{x}) \widehat{f}(\xi)]  ,
\end{align*}where the Fourier transform $\widehat{f}(\xi)$ is the linear operator defined as \begin{align*}
    \widehat{f}(\xi) \varphi (u)&:= \int_{\mathbb{G}^{5,6}} f(\mathbf{x}) \pi_\xi^* (\mathbf{x}) \varphi(u) d \mathbf{x} \\ &= \int_{\Z_p^5} f(\mathbf{x}) e^{2 \pi i \{ P_\xi (\mathbf{x}^{-1}, u) \}_p} \varphi (u - (x_1 , x_2)) , d\mathbf{x}, \quad \varphi \in \mathcal{H}_\xi.
\end{align*}Alternatively, if we define $$\Tilde{P}_\xi (\mathbf{x} , u) := P_\xi (\mathbf{x} , u)- \xi \cdot \mathbf{x}$$ by using our expressions for the matrix coefficients we can see this operator as the matrix $\widehat{f}(\xi) \in \C^{\|(\xi_3, \xi_4, \xi_5)\|_p|\xi_5|_p \times \|(\xi_3, \xi_4, \xi_5)\|_p |\xi_5|_p}$ defined by the expression \begin{align*}
    \widehat{f}(\xi)_{h h'}&:= \int_{\mathbb{G}^{5,6}} f(\mathbf{x}) \pi_\xi^* (\mathbf{x})_{hh'} d \mathbf{x} \\ &= \int_{\Z_p^5} f(\mathbf{x})  e^{2 \pi i \{ -P_\xi (\mathbf{x} , h)\}_p} \1_{h - h' + p^{-\vartheta(\xi_3, \xi_4, \xi_5)} \Z_p \times p^{-\vartheta( \xi_5)} \Z_p } (x_1,x_2) d\mathbf{x}, 
\end{align*}where $\varphi \in \mathcal{H}_\xi$, which in terms of the $\Z_p^5$-Fourier transform is expressed as \begin{align*}
    \widehat{f}&(\xi)_{h h'} = \mathcal{F}_{\Z_p^5} [e^{2 \pi i \{ -\Tilde{P}_\xi (\mathbf{x} , h)\}_p} \1_{h - h' + p^{-\vartheta(\xi_3, \xi_4, \xi_5)} \Z_p \times p^{-\vartheta( \xi_5)} \Z_p } (x_1,x_2) f] (\xi ) \\ &=\mathcal{F}_{\Z_p^5} [e^{2 \pi i \{ -\Tilde{P}_\xi (\mathbf{x} , h)\}_p} \1_{h - h' + p^{-\vartheta(\xi_3, \xi_4, \xi_5)} \Z_p \times p^{-\vartheta( \xi_5)} \Z_p } (\cdot, \cdot)  ] *_{\widehat{\Z}_p^5} \mathcal{F}_{\Z_p^5}[f] (\xi ).
\end{align*}
With this Fourier series representation, a linear invariant operator $T$ can be written as a pseudo-differential operator \begin{align*}
    T_\sigma f(\mathbf{x}) &= \sum_{ \| (\xi_3, \xi_4, \xi_5)\|_p =1 } \sigma(\xi)\widehat{f}(\xi) e^{2 \pi i \{ \xi \cdot \mathbf{x}\}_p} + \sum_{|\xi_5|_p = 1, \,\, (\xi_1, \xi_2, \xi_3, \xi_4) \in \widehat{\mathcal{B}}_4} \| (\xi_3, \xi_4) \|_p Tr[ \pi_\xi (\mathbf{x}) \sigma(\xi)\widehat{f}(\xi)] \\ &+ \sum_{\, |\xi_5|_p >1 } \sum_{(\xi_3 , \xi_4) \in \Q_p^2 / p^{\vartheta(\xi_5)}\Z_p^{-1}} \sum_{(\xi_1, \xi_2) \in \Q_p^2 / p^{ \vartheta(\xi_3, \xi_4 , \xi_5)} \Z_p^2} \| (\xi_3, \xi_4, \xi_5) \|_p |\xi_5|_p Tr[ \pi_\xi (\mathbf{x})\sigma(\xi) \widehat{f}(\xi)]  ,
\end{align*}

To conclude this section we are interested once again in studying the associated Vladimirov sub-Laplacian as $$\mathcal{L}_{sub}^\alpha f (x) : = (\partial^\alpha_{X_1} + \partial^\alpha_{X_2} ) f (x), \,\,\, f \in \mathcal{D}(\mathbb{G}^{5,6}).$$For $\alpha>0$ the following theorem proves how our conjecture for this operator holds true in the last one of the $5$-dimensional groups $\mathbb{G}^{5,6}$. 
\begin{teo}\label{teosubLapG56}

The Vladimirov sub-Laplacian $\mathcal{L}_{sub}^\alpha$ is a globally hypoelliptic operator, which is invertible in the space of mean zero functions. Moreover, the space $L^2(\mathbb{G}^{5,6})$ can be written as the direct sum $$L^2(\mathbb{G}^{5,6}) = \overline{\bigoplus_{\xi \in \widehat{\mathbb{G}}^{5,6}} \bigoplus_{h' \in  I_\xi} \mathcal{V}_{\xi}^{h'}}, \, \,\, \mathcal{V}_{\xi} = \bigoplus_{h' \in  I_\xi} \mathcal{V}_{\xi}^{h'}, $$where each finite-dimensional sub-space$$\mathcal{V}_{\xi}^{h'}:= \mathrm{span}_\C \{ (\pi_{\xi})_{hh'} \, : \, h \in I_\xi \},$$is an invariant sub-space of $\mathcal{L}_{sub, \textbf{W}}^\alpha$, and when restricted to $\mathcal{V}_{\xi }^{h'}$ we get the estimate \begin{align*}
    \| \mathscr{L}_{sub}^\alpha |_{\mathcal{V}_\xi^{h'}} \|_{inf} &= \min_{\tau, h'}\lambda_{\xi, h', \tau} (\mathscr{L}_{sub}^\alpha) \gtrsim | \xi_1 |_p^\alpha + |\xi_2 |_p^\alpha.
\end{align*}
\end{teo}
\begin{proof}
As we anticipated, we will be following the same arguments as in Theorem \ref{teosubLapG22}. In order to prove Theorem \ref{teosubLapG56} we want to use this calculate the associated symbol of $\mathscr{L}_{sub}^\alpha$, and its respective invariant sub-spaces. Let us start by computing the symbols of the directional VT operators $\partial_{X_1}^{\alpha}, \partial_{X_2}^{\alpha}$,  $\alpha>0$, by using what we got for the matrix coefficients. For $\partial_{X_1}^{\alpha}$, let us introduce some notation: let $\psi_\xi^1 : \Z_p \to \C^{d_\xi \times d_\xi}$ be the matrix function with entries $$(\psi_{\xi}^1)_{hh'} (t):= \delta_{h_2' h_2} e^{2 \pi i \{ (\xi_1 + \xi_5 (\frac{(h_2')^2}{2} + \frac{(h_1')^2(h_2')}{2}) )t+  \frac{\xi_5 (h_1')(h_2')}{2}t^2 + \xi_5(\frac{h_2'}{6} + \frac{(h_1')^2}{2})t^3\}_p}\1_{h_1' - h_1 + p^{- \vartheta (\xi_3 , \xi_4, \xi_5 )} \Z_p }(t), \, \, \, u \in \Z_p.$$ The associated associated symbol $\sigma_{\partial_{X_1}^{\alpha}} (\xi) = \partial_{X_1}^{\alpha} \pi_{\xi}|_{\mathbf{x}=0}$ is a matrix which can be written as  \[ \sigma_{\partial_{X_1}^{\alpha}} (\xi)_{hh'}  =\begin{cases}
 | \xi_1 |_p^\alpha - \frac{1 - p^{-1}}{1 - p^{-(\alpha +1)}}  \quad & \text{if} \quad | \xi_3 |_p = | \xi_4 |_p = |\xi_5|_p= 1,\\\partial^\alpha_t(e^{2 \pi i \{ \xi_1 t \}_p}\1_{h' - h' + p^{- \vartheta (\xi_3, \xi_4 )} \Z_p })|_{t=0} \quad & \text{if} \quad 1=|\xi_5|_p, \,\, \| (\xi_3, \xi_4 )\|_p>1, \, \, h \in \Z_p/p^{- \vartheta (\xi_3 , \xi_4)}\Z_p, \\\partial^\alpha_t(\psi_{\xi}^1)_{hh'})|_{t=0} \quad & \text{if} \quad 1<|\xi_5|_p, \,\, h \in p^{- \vartheta(\xi_3, \xi_4, \xi_5)}\Z_p \times p^{- \vartheta( \xi_5)}\Z_p, 
\end{cases}
\]where $\partial_{t}^\alpha$ is the VT operator $$\partial_{t}^\alpha \psi (t):= \int_{\Z_p} \frac{\psi(t - v) - \psi(t) }{|v|_p^{\alpha+1}} dv, \, \, \, \, \psi \in \mathcal{D}(\Z_p).$$ In a similar way, for $\partial_{X_2}^{\alpha}$ let us introduce the $$(\psi_{\xi}^2)_{hh'} (t):= \delta_{h_1' h_1} e^{2 \pi i \{ (\xi_2 + \xi_3 (h_1') +  \xi_4 \frac{(h_1')^2}{2}) t \}_p}\1_{h_2' - h_2 + p^{- \vartheta ( \xi_5 )} \Z_p }(t), \, \, \, u \in \Z_p.$$ The associated associated symbol $\sigma_{\partial_{X_2}^{\alpha}} (\xi) = \partial_{X_2}^{\alpha} \pi_{\xi}|_{\mathbf{x}=0}$ is a matrix which can be written as  \[ \sigma_{\partial_{X_2}^{\alpha}} (\xi)_{hh'}  =\begin{cases}
 | \xi_2 |_p^\alpha - \frac{1 - p^{-1}}{1 - p^{-(\alpha +1)}}  \quad & \text{if} \quad | \xi_3 |_p = | \xi_4 |_p = |\xi_5|_p= 1,\\\big( | \xi_2  + \xi_3    h +  \xi_4  \frac{h^2}{2} |_p^\alpha - \frac{1 - p^{-1}}{1 - p^{-(\alpha +1)}} \big) \delta_{hh'}  \quad & \text{if} \quad 1=|\xi_5|_p, \,\, \| (\xi_3, \xi_4 )\|_p>1, \, \, h \in \Z_p/p^{- \vartheta (\xi_3 , \xi_4)}\Z_p, \\\partial^\alpha_t(\psi_{\xi}^2)_{hh'})|_{t=0} \quad & \text{if} \quad 1<|\xi_5|_p, \,\, h \in p^{- \vartheta(\xi_3, \xi_4, \xi_5)}\Z_p \times p^{- \vartheta( \xi_5)}\Z_p. 
\end{cases}
\]

Summing up, the symbol of the Vladimirov Sub-Laplacian $\mathscr{L}^\alpha_{sub}$ is component-wise $\sigma_{\mathscr{L}^\alpha_{sub}} (\xi)_{hh'}  =$ \[ \begin{cases}
  | \xi_1 |_p^\alpha +| \xi_2 |_p^\alpha - 2\frac{1 - p^{-1}}{1 - p^{-(\alpha +1)}}  \quad & \text{if} \quad | \xi_3 |_p= |\xi_4|_p = |\xi_5|_p = 1, \\ (\partial_{u}^\alpha  +| \xi_2  + \xi_3    h +  \xi_4  \frac{h^2}{2} |_p^\alpha - \frac{1 - p^{-1}}{1 - p^{-(\alpha +1)}} )(e^{2 \pi i \{ \xi_1 u \}_p}\1_{h' - h + p^{- \vartheta (\xi_3 )} \Z_p }) (u))|_{u=0}, \quad & \text{if} \quad 1= |\xi_5|_p< \|(\xi_3,\xi_4)\|_p , \\ \partial^\alpha_t(\psi_{\xi}^1)_{hh'})|_{t=0} + \partial^\alpha_t(\psi_{\xi}^2)_{hh'})|_{t=0}, \quad & \text{if} \quad 1< |\xi_5|.
\end{cases}
\]
\end{proof}

\section{Final remarks}
\subsection{Relation to non-invariant vector fields on $\Z_p^d$} One of the main advantages of invariant operators is how, using the representation theory of the group, one can construct easily a fundamental solution in terms of the associated symbol, just like in Rockland's seminal paper \cite{Rockland1978}. To give an example, in the particular case when $\mathbb{G}=\mathfrak{g}=\Z_p^d$, the directional VT operators introduced in Definition \ref{defiDirectionalK} take the form $$\partial_V^\alpha f (x) := \frac{1 - p^{\alpha}}{1-p^{-(\alpha+1)}}\int_{\Z_p^d} \frac{f(x -tV(x)) - f(x)}{| t |_p^{ \alpha + 1}}dt, $$where $V:\Z_p^d \to \Z_p^d$ is an analytic vector field. For this operator, one can easily compute its associated symbol: 
\[\sigma_{\partial_{V}^\alpha}(x,\xi) = \begin{cases}
    0, \, & \, \, \text{if} \, \, |V(x) \cdot \xi|_p \leq 1,\\|V(x) \cdot \xi|_p^\alpha - \frac{1 - p^{-1}}{1 - p^{- (\alpha + 1)}}  & \, \, \text{if} \, \, |V(x) \cdot\xi|_p>1.
\end{cases}
 \]With the above calculation it is easy to check how the Vladimirov Laplacian $\mathscr{L}^\alpha$ on $\Z_p^d$ is a globally hypoelliptic operator with symbol $$\sigma_{\mathscr{L}^\alpha}(\xi) = \sum_{i=1}^d \Big( | \xi_i|_p^\alpha - \frac{1 - p^{-1}}{1 - p^{- (\alpha + 1)}} \Big), \quad \xi \in \widehat{\Z}_p^d, $$and as a direct consequence we can construct a fundamental solution for the operator given by the following convolution kernel: $$E_\alpha (x) = \sum_{\xi \in \widehat{\Z}_p^d} \sigma_{\mathscr{L}^\alpha}^{-1}(\xi) e^{2 \pi i \{ \xi \cdot x \}_p}.$$When the vector field $V$ is not constant, constructing a fundamental solution becomes more complicated, but there are some interesting cases where it is actually possible to obtain it explicitly. For instance, on $\Z_p^3$, consider the directional VT operator in the direction of the vector field $$V (x,y,z)= e_{2} +  x e_{3}=(0,1,x), $$defined via the formula \begin{align*}
    \partial_{V}^\alpha f(x,y,z) &= \int_{\Z_p} \frac{f((x,y,z) + t V(x,y,z)) - f(x,y,z)}{|t|_p^{\alpha + 1}}dt \\&= \int_{\Z_p} \frac{f(x, y + t ,z + t x ) - f(x,y,z)}{|t|_p^{\alpha + 1}}dt . 
\end{align*}From the perspective of the Fourier analysis on $\Z_p^3$, the resulting operator is non-invariant, and its associated symbol, corresponding to the Fourier transform on $\Z_p^{2d+1}$ is going to be $$\sigma_{\partial_{V}^\alpha} (\gamma)= \int_{\Z_p} \frac{e^{2 \pi i \{ \gamma \cdot t (0,1,x)  \}_p } - 1}{|t|_p^{\alpha + 1}}dt=|\gamma \cdot (0,1,x)|_p^\alpha - \frac{1 - p^{-1}}{1 - p^{-(\alpha + 1)}}=|\gamma_2 +x \gamma_3|_p^\alpha - \frac{1 - p^{-1}}{1 - p^{-(\alpha + 1)}}.$$With this symbolic representation it is not trivial to check the globall hypoellipticity of the operator, but fortunately we can use an alternative perspective. If instead we consider $f$ as a function on $\mathbb{H}_1$, which is simply $\Z_p^3$ with a different operation, then the operator $\partial_V^\alpha$ is the VT directional operator associated to $Y \in \mathfrak{h}_1 (\Z_p)$, which is invariant and therefore a Fourier multiplier with the group Fourier analysis of $\mathbb{H}_1$. Actually, if we consider an operator like $\partial_x^\alpha + \partial_V^\alpha$ on $\Z_p^3$, then  by using the group Fourier transform we can check how it coincides with the Vladimirov sub-Laplacian $\mathscr{L}^\alpha_{sub}$ on $\mathbb{H}_1$, for which we can construct explicitly a fundamental solution in terms of the $
\star$-convolution kernel $$E_\alpha (x,y,z) = \sum_{\xi \in \widehat{\mathbb{H}}_1 } |\xi_3|_p Tr[ \pi_\xi(x,y,z) \sigma_{\mathscr{L}^\alpha_{sub}}^{-1}(\xi) ] .$$

\subsection{Future research}

There are several things we can argue from the results obtained in this work. More importantly, our analysis here opens several questions that I think are worth to be explored. 
\begin{enumerate}
    \item As we saw before, in every case the unitary dual can be thought as a certain "sub-tree" of the full tree $\widehat{\Z}_p^5$. So, the guess is that, at least for the nilpotent case, in general dimension $d:= \mathrm{dim}_{\Z_p}(\mathbb{G})>5$ the unitary dual of $\mathbb{G}^{5,4}$ should be something similar, a sort of pruned version of the tree $\widehat{\Z}_p^d$. 
    \item It is not clear to me how to generate the representations for arbitrary groups. I know some works on representation theory and the orbit method for some kinds of profinite groups but, to the best knowledge of the author, nothing to produce something like the matrix coefficients in the general case, just the characters. Ideally we would have the matrix coefficients to perform Fourier analysis better, as the case of the Vladimirov sub-Laplacian illustrates. 
    \item A question I would like to answer is whether this operator satisfies some sub-elliptic estimates. 
    \item The definition of directional VT operators works the same on non-compact groups, where at least in the graded case one can also study the Vladimirov sub-Laplacian. In that setting one also has homogeneous structures and homogeneous operators very similar to Rockland operators. 
    \item Another important question I would like to ask is whether one can find a nice expression for the heat kernel of the Vladimirov sub-Laplacian, maybe in a similar way as in \cite{Bendikov2014}. 
    \item Our analysis here does not include any information about the multiplicity of the eigenvalues of the Vladimirov sub-Laplacian. We want to remark how, specially in the case of the Heisenberg group, this operator acts like some kind of $p$-adic Schrodinger operator, like the ones studied in the reference book \cite{vladiBook}.  
\end{enumerate}

\nocite{*}
\bibliographystyle{acm}
\bibliography{main}
\Addresses

\end{document}